\documentclass{amsart}%
\usepackage{amssymb}
\usepackage{amsfonts}
\usepackage{geometry}%
\usepackage{amsmath}%
\usepackage{graphicx}
\providecommand{\U}[1]{\protect\rule{.1in}{.1in}}
\newtheorem{theorem}{Theorem}
\theoremstyle{plain}
\newtheorem{acknowledgement}[theorem]{Acknowledgement}

\newtheorem{conjecture}[theorem]{Conjecture}

\newtheorem{definition}[theorem]{Definition}

\newtheorem{lemma}[theorem]{Lemma}

\newtheorem{proposition}[theorem]{Proposition}

\numberwithin{equation}{section}
\begin{document}
\title{A reprise of the NTV conjecture for the Hilbert transform}
\author{Eric T. Sawyer}
\address{McMaster University}
\email{sawyer@mcmaster.ca}

\begin{abstract}
We give a slightly different proof of the NTV conjecture for the Hilbert
transform that was proved by T. Hyt\"{o}nen, M. Lacey, E.T. Sawyer, C.-Y. Shen
and I. Uriartre-Tuero in \cite{LaSaShUr3}, \cite{Lac} and \cite{Hyt2},
building on previous work of F. Nazarov, S. Treil and A. Volberg in
\cite{NTV3}.

After modifying the decomposition of the main bilinear form, we give a new
proof of control of functional energy that is based on the potential Theorem 1
of \cite{Saw3}, rather than the Poisson Theorem 2 that is used in all other
proofs in the literature. This approach was pioneered in the first version of
Sawyer and Wick \cite{SaWi} on the ArXiv.  Then we alter the bottom-up corona
construction, the size functional, the straddling lemmas, and the use of
recursion of admissible collections of pairs of intervals, from M. Lacey
\cite{Lac}. However, the essence of control of the stopping form remains as in
the fundamental work of Lacey.

\end{abstract}
\maketitle
\tableofcontents

\section{Introduction}

F. Nazarov, S. Treil and A. Volberg formulated the two weight question for the
Hilbert transform \cite{Vol}, that in turn led to the famous NTV conjecture:

\begin{conjecture}
\cite{Vol}Given two positive locally finite Borel measures $\sigma$ and
$\omega$ on the real line $\mathbb{R}$, the Hilbert transform $Hf\left(
x\right)  =\operatorname{pv}\int_{\mathbb{R}}\frac{f\left(  y\right)  }%
{y-x}dy$ is bounded from $L^{2}\left(  \sigma\right)  $ to $L^{2}\left(
\omega\right)  $, i.e. the operator norm%
\begin{equation}
\mathfrak{N}_{H}\left(  \sigma,\omega\right)  \equiv\sup_{f\in L^{2}\left(
\mathbb{R};\sigma\right)  }\frac{1}{\left\Vert f\right\Vert _{L^{2}\left(
\sigma\right)  }}\left\Vert H\left(  f\sigma\right)  \right\Vert
_{L^{2}\left(  \omega\right)  }\ ,\label{Hilbert'}%
\end{equation}
is finite uniformly over appropriate truncations of $H$, if the two tailed
Muckenhoupt characteristic,%
\[
\mathcal{A}_{2}\left(  \sigma,\omega\right)  \equiv\sup_{I\in\mathcal{I}%
}\left(  \frac{1}{\left\vert I\right\vert }\int_{I}\mathbf{s}_{I}\left(
x\right)  ^{2}d\omega\left(  x\right)  \right)  \left(  \frac{1}{\left\vert
I\right\vert }\int_{I}\mathbf{s}_{I}\left(  y\right)  ^{2}d\sigma\left(
y\right)  \right)  \ ,
\]
is finite, where $\mathbf{s}_{I}\left(  x\right)  =\frac{\ell\left(  I\right)
}{\ell\left(  I\right)  +\left\vert x-c_{I}\right\vert }$, and if the testing
characteristic%
\begin{equation}
\mathfrak{T}_{H}\left(  \sigma,\omega\right)  \equiv\sup_{I\in\mathcal{I}%
}\frac{1}{\sqrt{\left\vert I\right\vert _{\sigma}}}\left\Vert H\mathbf{1}%
_{I}\sigma\right\Vert _{L^{2}\left(  \omega\right)  }\ ,\label{two testing}%
\end{equation}
is finite as well as its dual $\mathfrak{T}_{H}\left(  \omega,\sigma\right)  $.
\end{conjecture}

In a groundbreaking series of papers including \cite{NTV1},\cite{NTV2} and
\cite{NTV3}, Nazarov, Treil and Volberg used weighted Haar decompositions with
random grids of $\left(  r,\varepsilon\right)  -\operatorname{good}$ grids,
introduced their `pivotal condition', and proved the above conjecture under
the side assumption that this pivotal condition held. Subsequently in
\cite{LaSaUr2}, it was shown that the pivotal condition was not necessary for
the norm inequality to hold in general, a necessary `energy condition' was
introduced as a substitute, and a hybrid merging of these two conditions was
shown to be sufficient for use as a side condition. Eventually, these three
authors with C.-Y. Shen established the NTV conjecture in the absence of
common point masses in the measures $\sigma$ and $\omega$ in a two part paper;
M. Lacey, E. T. Sawyer, C.-Y. Shen and I. Uriarte-Tuero \cite{LaSaShUr3} and
M. Lacey \cite{Lac}.

The final assumption of no common point masses was removed shortly after by T.
Hyt\"{o}nen \cite{Hyt2} who replaced the two tailed Muckenhoupt characteristic
with a variant involving holes,%
\begin{align*}
\mathcal{A}_{2}^{\operatorname{hole}}\left(  \sigma,\omega\right)   &
\equiv\sup_{I\in\mathcal{I}}\left\{  \mathrm{P}\left(  I,\mathbf{1}%
_{\mathbb{R}\setminus I}\omega\right)  \left(  \frac{\left\vert I\right\vert
_{\sigma}}{\left\vert I\right\vert }\right)  +\left(  \frac{\left\vert
I\right\vert _{\omega}}{\left\vert I\right\vert }\right)  \mathrm{P}\left(
I,\mathbf{1}_{\mathbb{R}\setminus I}\sigma\right)  \right\}  ,\\
\mathrm{P}\left(  I,\mu\right)   & \equiv\int_{\mathbb{R}}\frac{\ell\left(
I\right)  }{\left(  \ell\left(  I\right)  +\left\vert y-c_{I}\right\vert
\right)  ^{2}}d\mu\left(  y\right)  ,\ \ \ \ \ c_{I}\equiv
\operatorname*{centre}\text{ of }I,
\end{align*}
and who also simplified some aspects of the proof. The reason for introducing
the holes is that if $\sigma$ and $\omega$ share a common point mass, then the
classical Muckenhoupt characteristic%
\[
A_{2}\left(  \sigma,\omega\right)  \equiv\sup_{I\in\mathcal{I}}\frac
{\left\vert I\right\vert _{\sigma}}{\left\vert I\right\vert }\frac{\left\vert
I\right\vert _{\omega}}{\left\vert I\right\vert },
\]
fails to be finite - simply let $I$ in the $\sup$ above shrink to a common
point mass. The $\operatorname{holed}$ Muckenhoupt condition trivially implies
the more elementary $\operatorname*{offset}$ Muckenhoupt condition
\[
A_{2}^{\operatorname*{offset}}\left(  \sigma,\omega\right)  \equiv
\sup_{\left(  Q,Q^{\prime}\right)  \in\mathcal{N}}\left(  \frac{1}{\left\vert
Q^{\prime}\right\vert }\int_{Q^{\prime}}d\omega\right)  \left(  \frac
{1}{\left\vert Q\right\vert }\int_{Q}d\sigma\right)  ,
\]
where $\mathcal{N}$ is the set of neighbouring pairs $\left(  Q,Q^{\prime
}\right)  $ of intervals of comparable side length, i.e. $\left(  Q,Q^{\prime
}\right)  \in\mathcal{D}$ for some dyadic grid $\mathcal{D}$ where $2^{-r}%
\ell\left(  Q^{\prime}\right)  \leq\ell\left(  Q\right)  \leq2^{r}\ell\left(
Q^{\prime}\right)  $, $Q\cap Q^{\prime}=\emptyset$, and $\overline{Q}%
\cap\overline{Q^{\prime}}\not =\emptyset$. Here $r$ is the goodness parameter
in \cite{NTV3} and \cite{LaSaShUr3}.

The purpose of this paper is to give a modified proof of the above results for
the Hilbert transform, and to highlight where the various characteristics
enter into the proof. Our proof is somewhat simplified in that it avoids the
recursion of admissible collections of pairs of intervals that was used in
Part II \cite{Lac}, and the recent approach to controlling functional energy
due to Sawyer and Wick \cite{SaWi}, is discussed in the setting $p=2$.
However, apart from some reorganization using nonlinear bounds, the essence of
control of the stopping form remains as in the fundamental work of M. Lacey
\cite{Lac}, and the control of functional energy is similar in spirit to that
in \cite{LaSaShUr3}, and can in fact be imported as a black box from
\cite{LaSaShUr3} if one wishes.

\begin{theorem}
\label{main}Let $\sigma$ and $\omega$ be positive locally finite Borel
measures on $\mathbb{R}$. The Hilbert transform $H_{\sigma}$ is bounded from
$L^{2}\left(  \sigma\right)  $ to $L^{2}\left(  \omega\right)  $ \emph{if and
only if} the two tailed Muckenhoupt characteristic with holes, and both of the
testing characteristics, are finite. More precisely, there are universal
positive constants $c,C$ such that for all pairs $\left(  \sigma
,\omega\right)  $ of positive locally finite Borel measures,
\[
c\leq\frac{\sqrt{\mathcal{A}_{2}^{\operatorname{hole}}\left(  \sigma
,\omega\right)  }+\mathfrak{T}_{H}\left(  \sigma,\omega\right)  +\mathfrak{T}%
_{H}\left(  \omega,\sigma\right)  }{\mathfrak{N}_{H}\left(  \sigma
,\omega\right)  }\leq C.
\]

\end{theorem}

\begin{acknowledgement}
We thank Brett Wick for fruitful discussions on both the far and stopping forms.
\end{acknowledgement}

\section{Reprise of the proof of the $NTV$ conjecture}

We assume the Haar supports of all functions $f\in L^{2}\left(  \sigma\right)
$ and $g\in L^{2}\left(  \omega\right)  $ are contained, along with their two
children, in the $\left(  r,\varepsilon\right)  $-$\operatorname{good}$ grid
$\mathcal{D}_{\operatorname{good}}$, where $J\in\mathcal{D}%
_{\operatorname{good}}$ if $\operatorname*{dist}\left(  J,\partial K\right)
\geq\frac{1}{2}\ell\left(  J\right)  ^{\varepsilon}\ell\left(  K\right)
^{1-\varepsilon}$ for all $K\supset J$ with $\ell\left(  K\right)  \geq
2^{r}\ell\left(  J\right)  $. See \cite{NTV3} and \cite{LaSaShUr3} for the
reduction of the norm inequality to good functions with goodness parameters
$r\in\mathbb{N}$ and $\varepsilon>0 $.

We will make use of the following bound for sums of Haar coefficients, which
is proved e.g. in \cite{LaSaUr2} and \cite{LaSaShUr3}. For any interval $J$
with center $c_{J}$, and any finite measure $\nu$, define the Poisson
integral,%
\[
\mathrm{P}\left(  J,\nu\right)  \equiv\int_{\mathbb{R}}\frac{\ell\left(
J\right)  }{\left(  \ell\left(  J\right)  +\left\vert y-c_{J}\right\vert
\right)  ^{2}}d\nu\left(  y\right)  .
\]

\begin{lemma}
[Monotonicity and Energy Lemma]\label{Energy Lemma}Suppose that $J\subset
I^{\prime}\subset2I^{\prime}\subset F$ are dyadic intervals. Then
\begin{align*}
\left\vert \left\langle H_{\sigma}\mathbf{1}_{F\setminus I^{\prime}%
},\bigtriangleup_{J}^{\omega}g\right\rangle _{\omega}\right\vert ^{2}  &
\lesssim\left(  \frac{\mathrm{P}\left(  J,\mathbf{1}_{F\setminus I^{\prime}%
}\sigma\right)  }{\ell\left(  J\right)  }\right)  ^{2}\left\Vert
\bigtriangleup_{J}^{\omega}x\right\Vert _{L^{2}\left(  \omega\right)  }%
^{2}\left\Vert \bigtriangleup_{J}^{\omega}g\right\Vert _{L^{2}\left(
\omega\right)  }^{2}\ ,\\
\left\vert \left\langle H_{\sigma}\mathbf{1}_{F\setminus I^{\prime}}%
,\sum_{J^{\prime}\subset J}a_{J^{\prime}}h_{J^{\prime}}^{\omega}\right\rangle
_{\omega}\right\vert ^{2}  & \lesssim\mathrm{P}\left(  J,\mathbf{1}%
_{F\setminus I^{\prime}}\sigma\right)  ^{2}\left\vert J\right\vert _{\omega
}\left\Vert \sum_{J^{\prime}\subset J}a_{J^{\prime}}h_{J^{\prime}}^{\omega
}\right\Vert _{L^{2}\left(  \omega\right)  }^{2}.
\end{align*}

\end{lemma}

We will also need the following critical Poisson Decay Lemma from \cite{Vol}.

\begin{lemma}
[Poisson Decay Lemma]\label{Poisson inequality}Suppose $J\subset I\subset K$
are dyadic intervals and that $d\left(  J,\partial I\right)  >2\ell\left(
J\right)  ^{\varepsilon}\ell\left(  I\right)  ^{1-\varepsilon}$ for some
$0<\varepsilon<\frac{1}{2} $. Then for a positive Borel measure $\mu$ we have%
\begin{equation}
\mathrm{P}(J,\mu\mathbf{1}_{K\setminus I})\lesssim\left(  \frac{\ell\left(
J\right)  }{\ell\left(  I\right)  }\right)  ^{1-2\varepsilon}\mathrm{P}%
(I,\mu\mathbf{1}_{K\setminus I}).\label{e.Jsimeq}%
\end{equation}

\end{lemma}

\subsection{First steps and outline of the decomposition of forms}

Now we can begin the proof of Theorem \ref{main}. Following \cite{NTV3} and
\cite{LaSaShUr3}, we expand the Hilbert transform bilinear form $\left\langle
H_{\sigma}f,g\right\rangle _{\omega}$ in terms of the Haar decompositions of
$f$ and $g$,%
\[
\left\langle H_{\sigma}f,g\right\rangle _{\omega}=\sum_{I,J\in\mathcal{D}%
}\left\langle H_{\sigma}\bigtriangleup_{I}^{\sigma}f,\bigtriangleup
_{J}^{\omega}g\right\rangle _{\omega},
\]
and then assuming the Haar supports of $f$ and $g$ lie in an $\left(
r,\varepsilon\right)  -\operatorname{good}$ grid $D_{\operatorname{good}}$
(see \cite{NTV3} for the origin of goodness, and also \cite{LaSaShUr3} for its
use here), we decompose the double sum above as follows,%
\begin{align*}
\left\langle H_{\sigma}f,g\right\rangle _{\omega}  & =\left\{  \sum
_{\substack{I,J\in\mathcal{D} \\J\subset_{\tau}I}}+\sum_{\substack{I,J\in
\mathcal{D} \\I\subset_{\tau}J}}+\sum_{\substack{I,J\in\mathcal{D} \\I\cap
J=\emptyset}}+\sum_{\substack{I,J\in\mathcal{D} \\J\subset I\text{ and }%
\ell\left(  J\right)  \geq2^{-\tau}\ell\left(  I\right)  }}+\sum
_{\substack{I,J\in\mathcal{D} \\I\subset J\text{ and }\ell\left(  I\right)
\geq2^{-\tau}\ell\left(  J\right)  }}\right\}  \left\langle H_{\sigma
}\bigtriangleup_{I}^{\sigma}f,\bigtriangleup_{J}^{\omega}g\right\rangle
_{\omega}\\
& \equiv\mathsf{B}_{\operatorname{below}}\left(  f,g\right)  +\mathsf{B}%
_{\operatorname{above}}\left(  f,g\right)  +\mathsf{B}%
_{\operatorname*{disjoint}}\left(  f,g\right)  +\mathsf{B}%
_{\operatorname*{comparable}}\left(  f,g\right)  +\mathsf{B}%
_{\operatorname*{comparable}}^{\ast}\left(  f,g\right)  ,
\end{align*}
where $\tau\in\mathbb{N}$ is fixed below, larger than the goodness parameter
$r$, and $J\subset_{\tau}I$ means that $J\subset I$ and $\ell\left(  J\right)
\leq2^{-\tau}\ell\left(  I\right)  $. The first two forms are symmetric, and
using a CZ-energy corona decomposition with parameter $\Gamma>1$, we
subsequently decompose the below form $\mathsf{B}_{\operatorname{below}%
}\left(  f,g\right)  $ into another four forms,
\[
\mathsf{B}_{\operatorname{below}}\left(  f,g\right)  =\mathsf{B}%
_{\operatorname{neigh}}\left(  f,g\right)  +\mathsf{B}_{\operatorname{far}%
}\left(  f,g\right)  +\mathsf{B}_{\operatorname*{para}}\left(  f,g\right)
+\mathsf{B}_{\operatorname*{stop}}\left(  f,g\right)  ,
\]
in which there is control of both averages of $f$ and a Poisson-Energy
characteristic $\Gamma\left(  \sigma,\omega\right)  $ in each corona.
Altogether, we will eventually have eleven forms in our decomposition of the
inner product $\left\langle H_{\sigma}f,g\right\rangle _{\omega}$. It turns
out that all of the forms, save for the far, paraproduct, and stopping forms,
are controlled directly by the parameter $\Gamma$ and the testing and holed
Muckenhoupt characteristics. Control of the far, paraproduct, and stopping
forms will require finiteness of the energy characteristics introduced in
\cite{LaSaUr2},%
\begin{equation}
\mathcal{E}_{2}\left(  \sigma,\omega\right)  ^{2}\equiv\sup_{\overset{\cdot
}{\bigcup}_{r=1}^{\infty}I_{r}\subset I}\left(  \frac{\mathrm{P}\left(
I_{r},\mathbf{1}_{I\setminus I_{r}}\sigma\right)  }{\ell\left(  I_{r}\right)
}\right)  ^{2}\mathsf{E}\left(  I_{r},\omega\right)  ^{2}\frac{\left\vert
I_{r}\right\vert _{\omega}}{\left\vert I\right\vert _{\sigma}}\text{ and its
dual }\mathcal{E}_{2}\left(  \omega,\sigma\right)  ,\label{energy char}%
\end{equation}
where the supremum is taken over all pairwise disjoint subdecompositions of an
interval $I$ into dyadic subintervals $I_{r}\in\mathcal{D}\left[  I\right]  $.
The holed Muckenhoupt characteristic is needed only to control the far form,
while the smaller offset Muckenhoupt characteristic suffices elsewhere.

It may help to make a couple of comments on the broad structure of the proof.
The testing characteristics are used to control forms in which the measures
$\sigma$ and $\omega$ `see each other' in inner products $\left\langle
H_{\sigma}\mathbf{1}_{K}\bigtriangleup_{I}^{\sigma}f,\bigtriangleup
_{J}^{\omega}g\right\rangle _{\omega}$ with $J\subset K$ (such as comparable
and paraproduct forms), while the holed and offset Muckenhoupt characteristics
are used to control forms in which the measures $\sigma$ and $\omega$ `do
\emph{not} see each other' in inner products $\left\langle H_{\sigma
}\mathbf{1}_{K}\bigtriangleup_{I}^{\sigma}f,\bigtriangleup_{J}^{\omega
}g\right\rangle _{\omega}$ with $K\cap J=\emptyset$ (such as disjoint,
neighbour, far and stopping forms).

In some of these cases, there arise parts of the form in which the coronas
associated to the measures are far apart, and it is then necessary to derive
additional \emph{geometric decay} in the distance between coronas, much as in
the well-known Colar-Stein Lemma.

A main tool for this is the Poisson Decay Lemma, which extracts geometric
decay from goodness of intervals in the Haar supports of the functions, as
pioneered in \cite{NTV3}. This decay is used in controlling the disjoint, far
and stopping forms. Additional sources of needed geometric decay arise from
the coronas themselves as follows. There are three coronas at play in the proof:

\begin{enumerate}
\item a top-down Calder\'{o}n-Zygmund corona that derives geometric decay from
a $\sigma$-Carleson condition, that holds automatically from the definition of
the stopping times,

\item a top-down Poisson-Energy corona that also derives geometric decay from
a $\sigma$-Carleson condition, but that is no longer automatic from the
definition, and instead requires the energy condition, as pioneered in
\cite{LaSaUr2},

\item and a bottom-up Energy corona that derives geometric decay from an
`$\omega$-energy Carleson condition', that again holds automatically from the
definition of the stopping times, as pioneered in \cite{Lac}.
\end{enumerate}

Since the first two corona constructions both yield a $\sigma$-Carleson
condition, it is convenient to bundle them together in a single CZ-PE corona
construction, see (\ref{energy stop crit}) for the associated stopping times.
This corona construction has two additional important properties, control of
averages of $f$ that is crucial for controlling the paraproduct form, and
control of a Poisson-Energy characteristic that is crucial for controlling the
stopping form. The $\sigma$-Carleson condition is used in the far, paraproduct
and stopping forms while the `$\omega$-energy Carleson condition' is used in
the stopping form.

\subsection{Control of the disjoint and comparable forms by testing and offset
Muckenhoupt characteristics}

The disjoint form $\mathsf{B}_{\operatorname*{disjoint}}\left(  f,g\right)  $,
and the comparable forms $\mathsf{B}_{\operatorname*{comparable}}\left(
f,g\right)  $ and $\mathsf{B}_{\operatorname*{comparable}}^{\ast}\left(
f,g\right)  $, are controlled in \cite[Section 4]{SaShUr12}\footnote{the proof
given there is for the $Tb$ theorem, and the reader can note that only the
offset Muckenhoupt condition is used} using only the testing characteristics,
the offset Muckenhoupt characteristic, and the following weak boundedness
characteristic,%
\[
\mathfrak{W}\left(  \sigma,\omega\right)  \equiv\sup
_{I,J\ \operatorname*{adjacent}:\ 2^{-\tau}\ell\left(  I\right)  \leq
\ell\left(  J\right)  \leq2^{\tau}\ell\left(  I\right)  }\left\vert \int
_{J}\left(  H_{\sigma}\mathbf{1}_{I}\right)  d\omega\right\vert .
\]
Similar results are in \cite{NTV3}, \cite{LaSaShUr3}, and \cite{Hyt2}. Note
that it follows from \cite[Lemma 2.4]{Hyt2} that $\mathfrak{W}\left(
\sigma,\omega\right)  \leq A_{2}^{\operatorname*{offset}}\left(  \sigma
,\omega\right)  $.

The forms $\mathsf{B}_{\operatorname{below}}\left(  f,g\right)  $ and
$\mathsf{B}_{\operatorname{above}}\left(  f,g\right)  $ are symmetric, and so
we need only bound the below form $\mathsf{B}_{\operatorname{below}}\left(
f,g\right)  $. To acccomplish this we decompose the below form $\mathsf{B}%
_{\operatorname{below}}$ further into neighbour, far, paraproduct and stopping
forms,%
\[
\mathsf{B}_{\operatorname{below}}=\mathsf{B}_{\operatorname{neigh}}%
+\mathsf{B}_{\operatorname{far}}+\mathsf{B}_{\operatorname*{para}}%
+\mathsf{B}_{\operatorname*{stop}},
\]
which are detailed as we progress through the proof.

Let%
\[
\mathcal{P}_{\operatorname{below}}\equiv\left\{  \left(  I,J\right)
\in\mathcal{D}_{\operatorname{good}}\times\mathcal{D}_{\operatorname{good}%
}:J\subset_{\tau}I\right\}
\]
be the set of pairs of $\operatorname{good}$ dyadic intervals $\left(
I,J\right)  $ with $J$ at least $\tau$ levels below and inside $I$. We begin
by splitting the $\operatorname{below}$ form into home and neighbour forms,
where $\theta K$ denotes the dyadic sibling of $K\in\mathcal{D}$,%
\begin{align*}
\mathsf{B}_{\operatorname{below}}\left(  f,g\right)   & =\sum_{\left(
I,J\right)  \in\mathcal{P}_{\operatorname{below}}}\left\langle H_{\sigma
}\left(  \mathbf{1}_{I_{J}}\bigtriangleup_{I}^{\sigma}f\right)
,\bigtriangleup_{J}^{\omega}g\right\rangle _{\omega}+\sum_{\left(  I,J\right)
\in\mathcal{P}_{\operatorname{below}}}\left\langle H_{\sigma}\left(
\mathbf{1}_{\theta I_{J}}\bigtriangleup_{I}^{\sigma}f\right)  ,\bigtriangleup
_{J}^{\omega}g\right\rangle _{\omega}\\
& \equiv\mathsf{B}_{\operatorname{home}}\left(  f,g\right)  +\mathsf{B}%
_{\operatorname{neigh}}\left(  f,g\right)  .
\end{align*}

\subsection{Control of the neighbour form by the offset Muckenhoupt
characteristic}

The neighbour form is easily controlled by the $A_{2}^{\operatorname*{offset}%
}$ condition using Lemma \ref{Energy Lemma} and the fact that the intervals
$J$ are good, namely we claim%
\[
\left\vert \mathsf{B}_{\operatorname{neigh}}\left(  f,g\right)  \right\vert
\leq C_{\varepsilon}\sqrt{A_{2}^{\operatorname*{offset}}\left(  \sigma
,\omega\right)  }\left\Vert f\right\Vert _{L^{2}(\sigma)}\left\Vert
g\right\Vert _{L^{2}(\omega)}.
\]
To see this, momentarily fix an integer $s\geq\tau$. We have
\[
\left\langle H_{\sigma}\left(  \mathbf{1}_{\theta\left(  I_{J}\right)
}\bigtriangleup_{I}^{\sigma}f\right)  ,\bigtriangleup_{J}^{\omega
}g\right\rangle _{\omega}=E_{\theta\left(  I_{J}\right)  }^{\sigma}\Delta
_{I}^{\sigma}f\cdot\left\langle H_{\sigma}\left(  \mathbf{1}_{\theta\left(
I_{J}\right)  }\right)  ,\bigtriangleup_{J}^{\omega}g\right\rangle _{\omega},
\]
and thus we can write%
\begin{equation}
\mathsf{B}_{\operatorname{neigh}}\left(  f,g\right)  =\sum_{I,J\in
\mathcal{D}_{\operatorname{good}}\text{ and }J\subset_{\tau}I}\left(
E_{\theta\left(  I_{J}\right)  }^{\sigma}\Delta_{I}^{\sigma}f\right)
\ \left\langle H_{\sigma}\left(  \mathbf{1}_{\theta\left(  I_{J}\right)
}\sigma\right)  ,\Delta_{J}^{\omega}g\right\rangle _{\omega}%
\ .\label{neighbour term}%
\end{equation}

Now we use the Poisson Decay Lemma. Assume that $J\subset_{\tau}I$. Let
$\frac{\ell\left(  J\right)  }{\ell\left(  I\right)  }=2^{-s}$ with $s\geq
\tau$ in the pivotal estimate in Lemma \ref{Energy Lemma} with $J\subset I$ to
obtain
\begin{align*}
\left\vert \langle H_{\sigma}\left(  \mathbf{1}_{\theta\left(  I_{J}\right)
}\right)  ,\Delta_{J}^{\omega}g\rangle_{\omega}\right\vert  &  \lesssim
\left\Vert \Delta_{J}^{\omega}g\right\Vert _{L^{2}\left(  \omega\right)
}\sqrt{\left\vert J\right\vert _{\omega}}\mathrm{P}\left(  J,\mathbf{1}%
_{\theta\left(  I_{J}\right)  }\sigma\right) \\
&  \lesssim\left\Vert \Delta_{J}^{\omega}g\right\Vert _{L^{2}\left(
\omega\right)  }\sqrt{\left\vert J\right\vert _{\omega}}\cdot2^{-\left(
1-2\varepsilon\right)  s}\mathrm{P}\left(  I,\mathbf{1}_{\theta\left(
I_{J}\right)  }\sigma\right)  .
\end{align*}
where we have used (\ref{e.Jsimeq}).

Now we pigeonhole pairs $\left(  I,J\right)  $ of intervals by requiring
$J\in\mathfrak{C}_{\mathcal{D}}^{\left(  s\right)  }\left(  I\right)  $, i.e.
$J\subset I$ and $\ell\left(  J\right)  =2^{-s}\ell\left(  I\right)  $, and we
further separate the two cases where $I_{J}=I_{\pm}$, the right and left
children of $I$. We have
\begin{align*}
A(I,s) &  \equiv\sum_{J\in\mathfrak{C}_{\mathcal{D}}^{\left(  s\right)
}\left(  I\right)  }\left\vert \langle H_{\sigma}\left(  \mathbf{1}%
_{\theta\left(  I_{J}\right)  }\Delta_{I}^{\sigma}f\right)  ,\Delta
_{J}^{\omega}g\rangle_{\omega}\right\vert \\
&  \leq\sum_{+,-}2^{-\left(  1-2\varepsilon\right)  s}|E_{I_{\mp}}^{\sigma
}\Delta_{I}^{\sigma}f|\ \mathrm{P}(I,\mathbf{1}_{I_{\mp}}\sigma)\sum
_{J\;:\;2^{s}\ell\left(  J\right)  =\ell\left(  I\right)  :\ J\subset I_{\pm}%
}\left\Vert \Delta_{J}^{\omega}g\right\Vert _{L^{2}\left(  \omega\right)
}\sqrt{\left\vert J\right\vert _{\omega}}\\
&  \lesssim\sum_{+,-}2^{-\left(  1-2\varepsilon\right)  s}|E_{I_{\mp}}%
^{\sigma}\Delta_{I}^{\sigma}f|\ \mathrm{P}(I,\mathbf{1}_{I_{\mp}}\sigma
)\sqrt{\left\vert I_{\pm}\right\vert _{\omega}}\sqrt{\sum_{J\in\mathfrak{C}%
_{\mathcal{D}}^{\left(  s\right)  }\left(  I_{\pm}\right)  }\left\Vert
\Delta_{J}^{\omega}g\right\Vert _{L^{2}\left(  \omega\right)  }^{2}}\,,
\end{align*}
where we have used Cauchy-Schwarz in the last line, and we also note that
\begin{equation}
\sum_{I\in\mathcal{D}}\sum_{J\in\mathfrak{C}_{\mathcal{D}}^{\left(  s\right)
}\left(  I\right)  }\left\Vert \Delta_{J}^{\omega}g\right\Vert _{L^{2}\left(
\omega\right)  }^{2}\leq\left\Vert g\right\Vert _{L^{2}(\omega)}%
^{2}\ .\label{g}%
\end{equation}

Using
\begin{equation}
\left\vert E_{\theta\left(  I_{J}\right)  }^{\sigma}\Delta_{I}^{\sigma
}f\right\vert \leq\sqrt{E_{\theta\left(  I_{J}\right)  }^{\sigma}\left\vert
\Delta_{I}^{\sigma}f\right\vert ^{2}}\leq\left\Vert \Delta_{I}^{\sigma
}f\right\Vert _{L^{2}\left(  \sigma\right)  }\ \left\vert \theta\left(
I_{J}\right)  \right\vert _{\sigma}^{-\frac{1}{2}},\label{e.haarAvg}%
\end{equation}
and the offset Muckenhoupt condition, we can thus estimate,
\begin{align*}
A(I,s)  & \lesssim\sum_{+,-}2^{-\left(  1-2\varepsilon\right)  s}\left\Vert
\Delta_{I}^{\sigma}f\right\Vert _{L^{2}\left(  \sigma\right)  }\sqrt
{\sum_{J\in\mathfrak{C}_{\mathcal{D}}^{\left(  s\right)  }\left(  I_{\pm
}\right)  }\left\Vert \Delta_{J}^{\omega}g\right\Vert _{L^{2}\left(
\omega\right)  }^{2}}\left\vert I_{\mp}\right\vert _{\sigma}^{-\frac{1}{2}%
}\mathrm{P}(I,\mathbf{1}_{I_{\mp}}\sigma)\sqrt{\left\vert I_{\pm}\right\vert
_{\omega}}\\
& \lesssim\sqrt{A_{2}^{\operatorname*{offset}}\left(  \sigma,\omega\right)
}2^{-\left(  1-2\varepsilon\right)  s}\left\Vert \Delta_{I}^{\sigma
}f\right\Vert _{L^{2}\left(  \sigma\right)  }\sqrt{\sum_{J\in\mathfrak{C}%
_{\mathcal{D}}^{\left(  s\right)  }\left(  I\right)  }\left\Vert \Delta
_{J}^{\omega}g\right\Vert _{L^{2}\left(  \omega\right)  }^{2}}\,,
\end{align*}
since $\mathrm{P}(I,\mathbf{1}_{I_{\mp}}\sigma)\lesssim\frac{\left\vert
I_{\mp}\right\vert _{\sigma}}{\left\vert I\right\vert }$ shows that
\[
\left\vert I_{\mp}\right\vert _{\sigma}^{-\frac{1}{2}}\mathrm{P}%
(I,\mathbf{1}_{I_{\mp}}\sigma)\ \sqrt{\left\vert I_{\pm}\right\vert _{\omega}%
}\lesssim\frac{\sqrt{\left\vert I_{\mp}\right\vert _{\sigma}}\sqrt{\left\vert
I_{\pm}\right\vert _{\omega}}}{\left\vert I\right\vert }\lesssim\sqrt
{A_{2}^{\operatorname*{offset}}\left(  \sigma,\omega\right)  }.
\]

An application of Cauchy-Schwarz to the sum over $I$ using (\ref{g}) then
shows that
\begin{align*}
& \sum_{I\in\mathcal{D}}A(I,s)\lesssim\sqrt{A_{2}^{\operatorname*{offset}%
}\left(  \sigma,\omega\right)  }2^{-\left(  1-2\varepsilon\right)  s}%
\sqrt{\sum_{I\in\mathcal{D}}\left\Vert \Delta_{I}^{\sigma}f\right\Vert
_{L^{2}\left(  \sigma\right)  }^{2}}\sqrt{\sum_{I\in\mathcal{D}}\sum
_{J\in\mathfrak{C}_{\mathcal{D}}^{\left(  s\right)  }\left(  I\right)
}\left\Vert \Delta_{J}^{\omega}g\right\Vert _{L^{2}\left(  \omega\right)
}^{2}}\\
& \lesssim\sqrt{A_{2}^{\operatorname*{offset}}\left(  \sigma,\omega\right)
}2^{-\left(  1-2\varepsilon\right)  s}\lVert f\rVert_{L^{2}(\sigma)}\sqrt
{\sum_{I\in\mathcal{D}}\sum_{J\in\mathfrak{C}_{\mathcal{D}}^{\left(  s\right)
}\left(  I\right)  }\left\Vert \Delta_{J}^{\omega}g\right\Vert _{L^{2}\left(
\omega\right)  }^{2}}\\
& \lesssim\sqrt{A_{2}^{\operatorname*{offset}}\left(  \sigma,\omega\right)
}2^{-\left(  1-2\varepsilon\right)  s}\lVert f\rVert_{L^{2}(\sigma)}\left\Vert
g\right\Vert _{L^{2}(\omega)}\,.
\end{align*}
Now we sum in $s\geq\tau$ to obtain,
\begin{align*}
\left\vert \mathsf{B}_{\operatorname{neigh}}\left(  f,g\right)  \right\vert  &
\leq\sum_{I,J\in\mathcal{D}\text{ and }J\subset_{\tau}I}\left\vert
\left\langle H_{\sigma}\left(  \mathbf{1}_{\theta\left(  I_{J}\right)
}\bigtriangleup_{I}^{\sigma}f\right)  ,\bigtriangleup_{J}^{\omega
}g\right\rangle _{\omega}\right\vert =\left\vert \sum_{s=\tau}^{\infty}%
\sum_{I\in\mathcal{D}}A(I,s)\right\vert \\
& \lesssim\sum_{s=\tau}^{\infty}2^{-\left(  1-2\varepsilon\right)  s}%
\sqrt{A_{2}^{\operatorname*{offset}}\left(  \sigma,\omega\right)  }\left\Vert
f\right\Vert _{L^{2}(\sigma)}\left\Vert g\right\Vert _{L^{2}(\omega)}\\
& =C_{\tau}\sqrt{A_{2}^{\operatorname*{offset}}\left(  \sigma,\omega\right)
}\left\Vert f\right\Vert _{L^{2}(\sigma)}\left\Vert g\right\Vert
_{L^{2}(\omega)}.
\end{align*}

\section{Reduction of the home form by coronas}

In order to control the home form, we must pigeonhole the pairs of intervals
$\left(  I,J\right)  \in\mathcal{P}_{\operatorname{below}}$ into a collection
of pairwise disjoint corona boxes in which both averages of $f$ are
controlled, and a Poisson-Energy characteristic is controlled. Then we split
the home form into two forms according to this decomposition, which we call
the diagonal and far forms. But first we need to construct the corona decomposition.

Fix $\Gamma>1$ and a large dyadic interval $T$. Define a sequence of stopping
times $\left\{  \mathcal{F}_{n}\right\}  _{n=0}^{\infty}$ depending on $T$,
$\sigma$ and $\omega$ recursively as follows. Let $\mathcal{F}_{0}=\left\{
T\right\}  $. Given $\mathcal{F}_{n}$, define $\mathcal{F}_{n+1}$ to consist
of the \textbf{maximal} intervals $I^{\prime}\in\mathcal{D}%
_{\operatorname{good}}$ for which there is $I\in\mathcal{F}_{n}$ with
$I^{\prime}\subset I$ and%
\begin{align}
& \text{\textbf{either} }\mathrm{P}\left(  I^{\prime},\mathbf{1}_{I\setminus
I^{\prime}}\sigma\right)  ^{2}\mathsf{E}\left(  J,\omega\right)  ^{2}%
\frac{\left\vert I^{\prime}\right\vert _{\omega}}{\left\vert I^{\prime
}\right\vert _{\sigma}}>\Gamma,\label{energy stop crit}\\
& \text{\textbf{or} }\frac{1}{\left\vert I^{\prime}\right\vert _{\sigma}}%
\int_{I^{\prime}}\left\vert \mathsf{P}_{\mathcal{D}\left[  I\right]  }%
^{\sigma}f\right\vert d\sigma>4\frac{1}{\left\vert I\right\vert _{\sigma}}%
\int_{I}\left\vert \mathsf{P}_{\mathcal{D}\left[  I\right]  }^{\sigma
}f\right\vert d\sigma,\nonumber
\end{align}
where
\[
\mathsf{E}\left(  J,\omega\right)  ^{2}\equiv\frac{1}{\left\vert J\right\vert
_{\omega}}\int_{J}\frac{1}{\ell\left(  J\right)  ^{2}}\left\vert x-\frac
{1}{\left\vert J\right\vert _{\omega}}\int_{J}zd\omega\left(  z\right)
\right\vert ^{2}d\omega\left(  x\right)  ,
\]
is the energy functional introduced in \cite{LaSaUr2}, and $\mathsf{P}%
_{\mathcal{D}\left[  I\right]  }^{\sigma}=\sum_{K\in\mathcal{D}:\ K\subset
I}\bigtriangleup_{K}^{\sigma}$. Set $\mathcal{F}\equiv\bigcup_{n=0}^{\infty
}\mathcal{F}_{n}$, which we refer to as the CZ-PE stopping times for $T$,
$\sigma$ and $\omega$, and which we note consists of $\operatorname{good}$
intervals, i.e. $\mathcal{F}\subset\mathcal{D}_{\operatorname{good}}$. Denote
the associated coronas by $\mathcal{C}_{\mathcal{F}}\left(  F\right)  $ and
the grandchildren at depth $m\in N$ of $F$ in the tree $\mathcal{F}$ by
$\mathfrak{C}_{\mathcal{F}}^{\left(  m\right)  }\left(  F\right)  $, with
$\mathfrak{C}_{\mathcal{F}}^{\left(  1\right)  }\left(  F\right)  $
abbreviated to $\mathfrak{C}_{\mathcal{F}}\left(  F\right)  $. We will
consistly use calligraphic font $\mathcal{C}$ to denote coronas, and fraktur
font $\mathfrak{C}$ to denote children. Finally, we define%
\[
\alpha_{\mathcal{F}}\left(  F\right)  \equiv\sup_{G\in\mathcal{F}:\ G\supset
F}E_{G}^{\sigma}\left\vert \mathsf{P}_{\mathcal{D}\left[  \pi_{\mathcal{F}%
}G\right]  }^{\sigma}f\right\vert ,\ \ \ \ \ \text{for }F\in\mathcal{F}.
\]

The point of introducing the corona decomposition $\mathcal{D}\left[
T\right]  =\overset{\cdot}{\bigcup}_{F\in\mathcal{F}}\mathcal{C}_{\mathcal{F}%
}\left(  F\right)  $ is that we obtain control of both the averages
$E_{I}^{\sigma}\left\vert \mathsf{P}_{\mathcal{D}\left[  F\right]
}f\right\vert \equiv\frac{1}{\left\vert I\right\vert _{\sigma}}\int
_{I}\left\vert \mathsf{P}_{\mathcal{D}\left[  F\right]  }f\right\vert d\sigma$
of projections of $f$, and the Poisson-Energy functional%
\[
\mathrm{P}\mathsf{E}_{F}\left(  I\right)  ^{2}\equiv\mathrm{P}\left(
I,\mathbf{1}_{F\setminus I}\sigma\right)  ^{2}\mathsf{E}\left(  I,\omega
\right)  ^{2}\frac{\left\vert I\right\vert _{\omega}}{\left\vert I\right\vert
_{\sigma}},
\]
in each $\operatorname{good}$ corona $\mathcal{C}_{\mathcal{F}}%
^{\operatorname{good}}\left(  F\right)  \equiv\mathcal{C}_{\mathcal{F}}\left(
F\right)  \cap\mathcal{D}_{\operatorname{good}}$, i.e.%
\[
\frac{E_{I}^{\sigma}\left\vert \mathsf{P}_{F}f\right\vert }{E_{F}^{\sigma
}\left\vert \mathsf{P}_{F}f\right\vert }\leq4\text{ and }\mathrm{P}%
\mathsf{E}_{F}\left(  I\right)  ^{2}\leq\Gamma,\ \ \ \ \ \text{for all }%
I\in\mathcal{C}_{\mathcal{F}}^{\operatorname{good}}\left(  F\right)  \text{
and }F\in\mathcal{F}.
\]
In particular this inequality shows that the \emph{Poisson-Energy
characteristic}%
\begin{equation}
\mathrm{P}\mathsf{E}_{\mathcal{F}}\left(  \sigma,\omega\right)  \equiv
\sup_{F\in\mathcal{F}}\mathrm{P}\mathsf{E}_{\mathcal{F}}^{F}\left(
\sigma,\omega\right)  ,\text{ where }\mathrm{P}\mathsf{E}_{\mathcal{F}}%
^{F}\left(  \sigma,\omega\right)  \equiv\sup_{I\in\mathcal{C}_{\mathcal{F}%
}^{\operatorname{good}}\left(  F\right)  }\mathrm{P}\mathsf{E}_{F}\left(
I\right) \label{PE char}%
\end{equation}
of $\sigma$ and $\omega$ with respect to the stopping times $\mathcal{F}$, is
dominated by the parameter $\Gamma$ chosen in (\ref{energy stop crit}). For
future reference we note that two more refinements of the characteristic will
appear in connection with the Stopping Child Lemma and the control of the
stopping form below.

\begin{description}
\item[Consequences of the energy condition] If we \emph{assume} the finiteness
of the energy characteristic (\ref{energy char}) (which is often referred to
as the \emph{energy condition}), and if we take $\Gamma>4\mathcal{E}%
_{2}\left(  \sigma,\omega\right)  $ in (\ref{energy stop crit}), we obtain a
$\sigma$-Carleson, or $\sigma$-sparse, condition for the CZ-energy stopping
times $\mathcal{F}$,%
\begin{align*}
& \sum_{F^{\prime}\in\mathfrak{C}_{\mathcal{F}}\left(  F\right)  }\left\vert
F^{\prime}\right\vert _{\sigma}\leq\frac{1}{\Gamma}\sum_{F^{\prime}%
\in\mathfrak{C}_{\mathcal{F}}\left(  F\right)  }\min\left\{  \frac
{\int_{F^{\prime}}\left\vert \mathsf{P}_{F}f\right\vert d\sigma}{E_{F}%
^{\sigma}\left\vert \mathsf{P}_{F}f\right\vert },\mathrm{P}\left(  F^{\prime
},\mathbf{1}_{F\setminus F^{\prime}}\sigma\right)  ^{2}\mathsf{E}\left(
F^{\prime},\omega\right)  ^{2}\left\vert F^{\prime}\right\vert _{\omega
}\right\} \\
& \ \ \ \ \ \ \ \ \ \ \ \ \ \ \ \ \ \ \ \ \leq\left(  \frac{1}{4}\left\vert
F\right\vert _{\sigma}+\frac{\mathcal{E}_{2}\left(  \sigma,\omega\right)
}{\Gamma}\left\vert F\right\vert _{\sigma}\right)  <\frac{1}{2}\left\vert
F\right\vert _{\sigma},\ \ \ \ \ \text{for all }F\in\mathcal{F}\ ,
\end{align*}
which can then be iterated to obtain \emph{geometric decay in generations},
\begin{equation}
\sum_{G\in\mathfrak{C}_{\mathcal{F}}^{\left(  m\right)  }\left(  F\right)
}\left\vert G\right\vert _{\sigma}\leq C_{\delta}2^{-\delta m}\left\vert
F\right\vert _{\sigma},\ \ \ \ \ \text{for all }m\in\mathbb{N}\text{ and }%
F\in\mathcal{F}\ .\label{sparse}%
\end{equation}
In addition we obtain the quasi-orthogonality inequality, see e.g.
\cite{LaSaShUr3} and \cite{SaShUr7}, both of which apply here since
$E_{F}^{\sigma}\left(  \left\vert \mathsf{P}_{F}f\right\vert \right)
=E_{F}^{\sigma}\left(  \left\vert f-E_{F}^{\sigma}f\right\vert \right)
\leq2E_{F}^{\sigma}\left(  \left\vert f\right\vert \right)  $,%
\begin{equation}
\sum_{F\in\mathcal{F}}\left\vert F\right\vert _{\sigma}\alpha_{\mathcal{F}%
}\left(  F\right)  ^{2}\leq C\int_{\mathbb{R}}\left\vert f\right\vert
^{2}d\sigma.\label{qorth}%
\end{equation}
The finiteness of the energy characteristic $\mathcal{E}_{2}\left(
\sigma,\omega\right)  $ will be needed to enforce (\ref{sparse}), that is in
turn needed to control the far, paraproduct and stopping forms, at which point
we can appeal to the following simple inequality from \cite{LaSaUr2}%
\footnote{In \cite{LaSaUr2},\ this inequality is stated with a one-tailed
Muckenhoupt characteristic $\mathcal{A}_{2}\left(  \sigma,\omega\right)  $ on
the right hand side, but as the reader easily observes, the display at the top
of page 319 in the proof given in \cite{LaSaUr2} holds for the smaller
characteristic $A_{2}^{\operatorname*{offset}}\left(  \sigma,\omega\right)
$.},%
\begin{equation}
\mathcal{E}_{2}\left(  \sigma,\omega\right)  \leq C\left(  \mathfrak{T}%
_{H}\left(  \sigma,\omega\right)  +A_{2}^{\operatorname*{offset}}\left(
\sigma,\omega\right)  \right)  ,\label{energy control}%
\end{equation}
that controls $\mathcal{E}_{2}\left(  \sigma,\omega\right)  $ by the testing
characteristic for the \emph{Hilbert transform} and the offset Muckenhoupt
characteristic. Unfortunately this simple inequality fails for most other
Calder\'{o}n-Zygmund operators in place of the Hilbert transform, including
Riesz transforms in higher dimensions, see \cite{Saw} and \cite{SaShUr11}, and
this failure limits the\ current proof to essentially just the Hilbert
transform and similar operators on the real line.
\end{description}

Now we can pigeonhole the pairs of intervals arising in the sum defining the
below form. Given the corona decomposition of $\mathcal{D}$ according to the
Calder\'{o}n-Zygmund stopping times $\mathcal{F}$ constructed above, we define
the analogous decomposition of $\mathcal{P}_{\operatorname{below}}$,%
\begin{align*}
& \mathcal{P}_{\operatorname{below}}=\bigcup_{F,G\in\mathcal{F}:\ G\subset
F}\left[  \mathcal{C}_{\mathcal{F}}\left(  F\right)  \times\mathcal{C}%
_{\mathcal{F}}\left(  G\right)  \right]  \cap\mathcal{P}_{\operatorname{below}%
}\\
& =\left\{  \bigcup_{F\in\mathcal{F}}\left[  \mathcal{C}_{\mathcal{F}}\left(
F\right)  \times\mathcal{C}_{\mathcal{F}}\left(  F\right)  \right]
\cap\mathcal{P}_{\operatorname{below}}\right\}  \bigcup\left\{  \bigcup
_{F,G\in\mathcal{F}:\ G\subsetneqq F}\left[  \mathcal{C}_{\mathcal{F}}\left(
F\right)  \times\mathcal{C}_{\mathcal{F}}\left(  G\right)  \right]
\cap\mathcal{P}_{\operatorname{below}}\right\} \\
& \equiv\mathcal{P}_{\operatorname{diag}}\bigcup\mathcal{P}%
_{\operatorname{far}}\ .
\end{align*}
Then we consider the corresponding decomposition of the home form into
diagonal and far forms,%
\begin{align*}
\mathsf{B}_{\operatorname{home}}\left(  f,g\right)   & =\sum_{\left(
I,J\right)  \in\mathcal{P}_{\operatorname{diag}}}\left\langle H_{\sigma
}\left(  \mathbf{1}_{I_{J}}\bigtriangleup_{I}^{\sigma}f\right)
,\bigtriangleup_{J}^{\omega}g\right\rangle _{\omega}+\sum_{\left(  I,J\right)
\in\mathcal{P}_{\operatorname{far}}}\left\langle H_{\sigma}\left(
\mathbf{1}_{I_{J}}\bigtriangleup_{I}^{\sigma}f\right)  ,\bigtriangleup
_{J}^{\omega}g\right\rangle _{\omega}\\
& \equiv\mathsf{B}_{\operatorname{diag}}\left(  f,g\right)  +\mathsf{B}%
_{\operatorname{far}}\left(  f,g\right)  .
\end{align*}

We next decompose the $\operatorname{far}$ form into corona pieces using
$\mathcal{P}_{\operatorname{far}}^{F,G}\equiv\left[  \mathcal{C}_{\mathcal{F}%
}\left(  F\right)  \times\mathcal{C}_{\mathcal{F}}\left(  G\right)  \right]
\cap\mathcal{P}_{\operatorname{below}}$,%

\begin{align*}
\mathsf{B}_{\operatorname{far}}\left(  f,g\right)   & =\sum_{F,G\in
\mathcal{F}:\ G\subsetneqq F}\left\langle H_{\sigma}\left(  \sum
_{I\in\mathcal{C}_{\mathcal{F}}\left(  F\right)  }\mathbf{1}_{I_{J}%
}\bigtriangleup_{I}^{\sigma}f\right)  ,\sum_{J\in\mathcal{C}_{\mathcal{F}%
}\left(  G\right)  :\ J\subset_{\tau}I}\bigtriangleup_{J}^{\omega
}g\right\rangle _{\omega}\\
& =\sum_{F,G\in\mathcal{F}:\ G\subsetneqq F}\sum_{\left(  I,J\right)
\in\mathcal{P}_{\operatorname{far}}^{F,G}}\left\langle H_{\sigma}\left(
\mathbf{1}_{I_{J}}\bigtriangleup_{I}^{\sigma}f\right)  ,\bigtriangleup
_{J}^{\omega}g\right\rangle _{\omega}=\sum_{F,G\in\mathcal{F}:\ G\subsetneqq
F}\mathsf{B}_{\operatorname{far}}^{F,G}\left(  f,g\right) \\
\text{where }\mathsf{B}_{\operatorname{far}}^{F,G}\left(  f,g\right)   &
\equiv\sum_{\left(  I,J\right)  \in\mathcal{P}_{\operatorname{far}}^{F,G}%
}\left\langle H_{\sigma}\left(  \mathbf{1}_{I_{J}}\bigtriangleup_{I}^{\sigma
}f\right)  ,\bigtriangleup_{J}^{\omega}g\right\rangle _{\omega}.
\end{align*}
Now for $m>\tau$ and $F\in\mathcal{F}$ we define%
\begin{align*}
\mathsf{B}_{\operatorname{far}}^{F,m}\left(  f,g\right)   & \equiv\sum
_{G\in\mathfrak{C}_{\mathcal{F}}^{\left(  m\right)  }\left(  F\right)
}\mathsf{B}_{\operatorname{far}}^{F,G}\left(  f,g\right)  =\sum_{G\in
\mathfrak{C}_{\mathcal{F}}^{\left(  m\right)  }\left(  F\right)  }%
\sum_{\left(  I,J\right)  \in\mathcal{P}_{\operatorname{far}}^{F,G}%
}\left\langle H_{\sigma}\left(  \mathbf{1}_{I_{J}}\bigtriangleup_{I}^{\sigma
}f\right)  ,\bigtriangleup_{J}^{\omega}g\right\rangle _{\omega}\\
& =\sum_{F^{\prime}\in\mathfrak{C}_{\mathcal{F}}\left(  F\right)  }\sum
_{G\in\mathfrak{C}_{\mathcal{F}}^{\left(  m-1\right)  }\left(  F^{\prime
}\right)  }\sum_{J\in\mathcal{C}_{\mathcal{F}}\left(  G\right)  }\left\langle
H_{\sigma}\left(  \sum_{I\in\mathcal{C}_{\mathcal{F}}\left(  F\right)
:\ J\subset_{\tau}I}\mathbf{1}_{I_{J}}\bigtriangleup_{I}^{\sigma}f\right)
,\bigtriangleup_{J}^{\omega}g\right\rangle _{\omega}.
\end{align*}

\subsection{Control of the far form by way of functional energy}

Here we will control the $\operatorname{far}$ form $\mathsf{B}%
_{\operatorname{far}}\left(  f,g\right)  $ first by the local testing and
functional energy characteristic. The $\operatorname{far}$ form is given by%
\begin{align*}
\mathsf{B}_{\operatorname{far}}\left(  f,g\right)   & =\sum_{m=1}^{\infty}%
\sum_{F\in\mathcal{F}}\sum_{G\in\mathfrak{C}_{\mathcal{F}}^{\left(  m\right)
}\left(  F\right)  }\sum_{J\in\mathcal{C}_{\mathcal{F}}\left(  G\right)
}\left\langle H_{\sigma}\left(  \sum_{I\in\mathcal{C}_{\mathcal{F}}\left(
F\right)  :\ J\subset_{\tau}I}\mathbf{1}_{I_{J}}\bigtriangleup_{I}^{\sigma
}f\right)  ,\bigtriangleup_{J}^{\omega}g\right\rangle _{\omega}\\
& =\sum_{G\in\mathcal{F}}\sum_{J\in\mathcal{C}_{\mathcal{F}}\left(  G\right)
}\left\langle H_{\sigma}\left(  \sum_{I\in\mathcal{D}:\ G\subsetneqq I\text{
and }J\subset_{\tau}I}\mathbf{1}_{I_{J}}\bigtriangleup_{I}^{\sigma}f\right)
,\bigtriangleup_{J}^{\omega}g\right\rangle _{\omega},
\end{align*}
which we write with the dummy variable $G$ replaced by $F$,%
\[
\mathsf{B}_{\operatorname{far}}\left(  f,g\right)  =\sum_{F\in\mathcal{F}}%
\sum_{J\in\mathcal{C}_{\mathcal{F}}\left(  F\right)  }\left\langle H_{\sigma
}\left(  \sum_{I\in\left(  F,T\right]  \text{ and }J\subset_{\tau}I}%
\mathbf{1}_{I_{J}}\bigtriangleup_{I}^{\sigma}f\right)  ,\bigtriangleup
_{J}^{\omega}g\right\rangle _{\omega}.
\]
Given any collection $\mathcal{H}\subset\mathcal{D}$ of intervals, and a
dyadic interval $J$, we define the corresponding Haar projection
$\mathsf{P}_{\mathcal{H}}^{\omega}$ and its localization $\mathsf{P}%
_{\mathcal{H};J}^{\omega}$ to $J$ by%
\begin{equation}
\mathsf{P}_{\mathcal{H}}^{\omega}=\sum_{H\in\mathcal{H}}\bigtriangleup
_{H}^{\omega}\text{ and }\mathsf{P}_{\mathcal{H};J}^{\omega}=\sum
_{H\in\mathcal{H}:\ H\subset J}\bigtriangleup_{H}^{\omega}%
\ .\label{def localization}%
\end{equation}

\begin{definition}
Given any interval $F\in\mathcal{D}$, we define the $\left(  r,\varepsilon
\right)  $\emph{-Whitney} collection $\mathcal{M}_{\left(  r,\varepsilon
\right)  -\operatorname*{deep}}\left(  F\right)  $ of $F$ to be the set of
dyadic subintervals $W\subset F$ that are maximal with respect to the property
that $W\subset_{r,\varepsilon}F$. Let $\mathcal{D}\left[  W\right]
\equiv\left\{  J\in\mathcal{D}:J\subset W\right\}  $.
\end{definition}

Clearly the intervals in $\mathcal{M}_{\left(  r,\varepsilon\right)
-\operatorname*{deep}}\left(  F\right)  $ form a pairwise disjoint
decomposition of $F $.

\begin{definition}
Let $\mathfrak{F}\left(  \sigma,\omega\right)  $ be the smallest constant in
the `\textbf{f}unctional energy' inequality below, holding for all $h\in
L^{2}\left(  \sigma\right)  $ and all collections $\mathcal{F}\subset
\mathcal{D}$, and where $x$ denotes the identity function on $\mathbb{R}$:
\begin{equation}
\sum_{F\in\mathcal{F}}\sum_{W\in\mathcal{M}_{\left(  r,\varepsilon\right)
-\operatorname*{deep}}\left(  F\right)  }\left(  \frac{\mathrm{P}\left(
W,h\mathbf{1}_{F^{c}}\sigma\right)  }{\ell\left(  W\right)  }\right)  ^{2}%
\int_{W}\left\vert \mathsf{P}_{\mathcal{C}_{\mathcal{F}}\left(  F\right)
\cap\mathcal{D}\left[  W\right]  }^{\omega}x\right\vert ^{2}d\omega
\leq\mathfrak{F}\left(  \sigma,\omega\right)  \lVert h\rVert_{L^{2}\left(
\sigma\right)  }\,.\label{func ener}%
\end{equation}

\end{definition}

There is a similar definition of the dual constant $\mathfrak{F}^{\ast}\left(
\omega,\sigma\right)  $. The Intertwining Proposition will control the
following Intertwining form,%
\[
\mathsf{B}_{\operatorname{Inter}}\left(  f,g\right)  \equiv\sum_{F\in
\mathcal{F}}\ \sum_{I:\ I\supsetneqq F}\ \left\langle H_{\sigma}\left(
\mathbf{1}_{I_{F}}\bigtriangleup_{I}^{\sigma}f\right)  ,\mathsf{P}%
_{\mathcal{C}_{\mathcal{F}}\left(  F\right)  }^{\omega}g\right\rangle
_{\omega}\ ,
\]
whose difference from $\mathsf{B}_{\operatorname{far}}\left(  f,g\right)  $ is%
\begin{align*}
\mathsf{B}_{\operatorname{far}}\left(  f,g\right)  -\mathsf{B}%
_{\operatorname{Inter}}\left(  f,g\right)   & =\sum_{F\in\mathcal{F}}%
\sum_{I\in\left(  F,T\right]  }\sum_{J\in\mathcal{C}_{\mathcal{F}}\left(
F\right)  \text{ and }J\subset_{\tau}I}\left\langle H_{\sigma}\left(
\mathbf{1}_{I_{J}}\bigtriangleup_{I}^{\sigma}f\right)  ,\bigtriangleup
_{J}^{\omega}g\right\rangle _{\omega}\\
& -\sum_{F\in\mathcal{F}}\ \sum_{I\in\left(  F,T\right]  }\ \sum
_{J\in\mathcal{C}_{\mathcal{F}}\left(  F\right)  }\left\langle H_{\sigma
}\left(  \mathbf{1}_{I_{F}}\bigtriangleup_{I}^{\sigma}f\right)
,\bigtriangleup_{J}^{\omega}g\right\rangle _{\omega}\\
& =\sum_{F\in\mathcal{F}}\sum_{I\in\left(  F,T\right]  }\sum_{\substack{J\in
\mathcal{C}_{\mathcal{F}}\left(  F\right)  \\\ell\left(  J\right)  \geq
\ell\left(  F\right)  -\tau\text{ and }J\subset_{\tau}I}}\left\langle
H_{\sigma}\left(  \mathbf{1}_{I_{F}}\bigtriangleup_{I}^{\sigma}f\right)
,\bigtriangleup_{J}^{\omega}g\right\rangle _{\omega}.
\end{align*}
This difference form is easily controlled using the proofs of the disjoint and
comparable forms above,
\begin{align*}
& \left\vert \mathsf{B}_{\operatorname{far}}\left(  f,g\right)  -\mathsf{B}%
_{\operatorname{Inter}}\left(  f,g\right)  \right\vert \leq\sum_{F\in
\mathcal{F}}\sum_{I\in\left(  F,T\right]  }\sum_{\substack{J\in\mathcal{C}%
_{\mathcal{F}}\left(  F\right)  \\\ell\left(  J\right)  \geq\ell\left(
F\right)  -\tau\text{ and }J\subset_{\tau}I}}\left\vert \left\langle
H_{\sigma}\left(  \mathbf{1}_{I_{F}}\bigtriangleup_{I}^{\sigma}f\right)
,\bigtriangleup_{J}^{\omega}g\right\rangle _{\omega}\right\vert \\
& \lesssim\left(  \mathfrak{T}_{H}\left(  \sigma,\omega\right)  +\mathcal{A}%
_{2}^{\operatorname{hole}}\left(  \sigma,\omega\right)  \right)  \ \left\Vert
f\right\Vert _{L^{2}\left(  \sigma\right)  }\left\Vert g\right\Vert
_{L^{2}\left(  \omega\right)  }\ .
\end{align*}

\begin{definition}
\label{sigma carleson n}A collection $\mathcal{F}$ of dyadic intervals is
$\sigma$\emph{-Carleson} if%
\[
\sum_{F\in\mathcal{F}:\ F\subset S}\left\vert F\right\vert _{\sigma}\leq
C_{\mathcal{F}}\left\vert S\right\vert _{\sigma},\ \ \ \ \ S\in\mathcal{F}.
\]
The constant $C_{\mathcal{F}}$ is referred to as the Carleson norm of
$\mathcal{F}$.
\end{definition}

We now show that the functional energy inequality (\ref{func ener}), together
with local interval testing, suffices to prove the following Intertwining
Proposition, \cite{SaShUr7} and \cite{SaShUr12}, when $\mathcal{F}$ is
$\sigma$-Carleson.

Let $\mathcal{F}$ be any subset of $\mathcal{D}$. For any $J\in\mathcal{D}$,
we define $\pi_{\mathcal{F}}^{0}J$ to be the smallest $F\in\mathcal{F}$ that
contains $J$. Then for $s\geq1$, we recursively define $\pi_{\mathcal{F}}%
^{s}J$ to be the smallest $F\in\mathcal{F}$ that \emph{strictly} contains
$\pi_{\mathcal{F}}^{s-1}J$. This definition satisfies $\pi_{\mathcal{F}}%
^{s+t}J=\pi_{\mathcal{F}}^{s}\pi_{\mathcal{F}}^{t}J$ for all $s,t\geq0$ and
$J\in\mathcal{D}$. In particular $\pi_{\mathcal{F}}^{s}J=\pi_{\mathcal{F}}%
^{s}F$ where $F=\pi_{\mathcal{F}}^{0}J$. In the special case $\mathcal{F}%
=\mathcal{D}$ we often suppress the subscript $\mathcal{F}$ and simply write
$\pi^{s}$ for $\pi_{\mathcal{D}}^{s}$. Finally, for $F\in\mathcal{F}$, we
write $\mathfrak{C}_{\mathcal{F}}\left(  F\right)  \equiv\left\{  F^{\prime
}\in\mathcal{F}:\pi_{\mathcal{F}}^{1}F^{\prime}=F\right\}  $ for the
collection of $\mathcal{F}$-children of $F$.

\begin{proposition}
[The Intertwining Proposition]\label{strongly adapted'}Suppose that
$\mathcal{F}$ is $\sigma$-Carleson. Then%
\[
\left\vert \sum_{F\in\mathcal{F}}\ \sum_{I:\ I\supsetneqq F}\ \left\langle
H_{\sigma}\left(  \mathbf{1}_{I_{F}}\bigtriangleup_{I}^{\sigma}f\right)
,\mathsf{P}_{\mathcal{C}_{\mathcal{F}}\left(  F\right)  }^{\omega
}g\right\rangle _{\omega}\right\vert \lesssim\left(  \mathfrak{F}\left(
\sigma,\omega\right)  +\mathfrak{T}_{H}\left(  \sigma,\omega\right)  \right)
\ \left\Vert f\right\Vert _{L^{2}\left(  \sigma\right)  }\left\Vert
g\right\Vert _{L^{2}\left(  \omega\right)  }.
\]

\end{proposition}

\begin{proof}
We write the left hand side of the display above as%
\[
\sum_{F\in\mathcal{F}}\ \sum_{I:\ I\supsetneqq F}\ \left\langle H_{\sigma
}\left(  \mathbf{1}_{I_{F}}\bigtriangleup_{I}^{\sigma}f\right)  ,g_{F}%
\right\rangle _{\omega}=\sum_{F\in\mathcal{F}}\ \left\langle H_{\sigma}\left(
\sum_{I:\ I\supsetneqq F}\mathbf{1}_{I_{F}}\bigtriangleup_{I}^{\sigma
}f\right)  ,g_{F}\right\rangle _{\omega}\equiv\sum_{F\in\mathcal{F}%
}\ \left\langle H_{\sigma}f_{F},g_{F}\right\rangle _{\omega}\ ,
\]
where%
\[
g_{F}=\mathsf{P}_{\mathcal{C}_{\mathcal{F}}\left(  F\right)  }^{\omega}%
g=\sum_{J\in\mathcal{C}_{\mathcal{F}}\left(  F\right)  }\bigtriangleup
_{J}^{\omega}g\text{ and }f_{F}\equiv\sum_{I:\ I\supsetneqq F}\mathbf{1}%
_{I_{F}}\bigtriangleup_{I}^{\sigma}f\ .
\]
Note that $g_{F}$ is supported in $F$, and that $f_{F}$ is constant on $F$.
The intervals $I$ occurring in this sum are linearly and consecutively ordered
by inclusion, along with the intervals $F^{\prime}\in\mathcal{F}$ that contain
$F$. More precisely, we can write%
\[
F\equiv F_{0}\subsetneqq F_{1}\subsetneqq F_{2}\subsetneqq...\subsetneqq
F_{n}\subsetneqq F_{n+1}\subsetneqq...F_{N}%
\]
where $F_{m}=\pi_{\mathcal{F}}^{m}F$ for all $m\geq1$. We can also write%
\[
F=F_{0}\subsetneqq I_{1}\subsetneqq I_{2}\subsetneqq...\subsetneqq
I_{k}\subsetneqq I_{k+1}\subsetneqq...\subsetneqq I_{K}=F_{N}%
\]
where $I_{k}=\pi_{\mathcal{D}}^{k}F$ for all $k\geq1$, and by convention we
set $I_{0}=F$. There is a (unique) subsequence $\left\{  k_{m}\right\}
_{m=1}^{N}$ such that%
\[
F_{m}=I_{k_{m}},\ \ \ \ \ 1\leq m\leq N.
\]

Recall that%
\[
f_{F}\left(  x\right)  =\sum_{k=1}^{\infty}\mathbf{1}_{\left(  I_{k}\right)
_{F}}\left(  x\right)  \bigtriangleup_{I_{k}}^{\sigma}f\left(  x\right)
=\sum_{k=1}^{\infty}\mathbf{1}_{I_{k}\setminus I_{k-1}}\left(  x\right)
\sum_{\ell=k+1}^{\infty}\bigtriangleup_{I_{\ell}}^{\sigma}f\left(  x\right)  .
\]
Assume now that $k_{m}\leq k<k_{m+1}$. Using a telescoping sum, we compute
that for
\[
x\in I_{k+1}\setminus I_{k}\subset F_{m+1}\setminus F_{m},
\]
we have
\[
\left\vert f_{F}\left(  x\right)  \right\vert =\left\vert \sum_{\ell
=k+2}^{\infty}\bigtriangleup_{I_{\ell}}^{\sigma}f\left(  x\right)  \right\vert
=\left\vert \mathbb{E}_{\theta I_{k+2}}^{\sigma}f-\mathbb{E}_{I_{K}}^{\sigma
}f\right\vert \lesssim\mathbb{E}_{F_{m+1}}^{\sigma}\left\vert f\right\vert \ .
\]
Now $f_{F}$ is constant on $F$ and%
\begin{align*}
\left\vert f_{F}\right\vert  & \leq\sum_{m=0}^{N}\left(  \mathbb{E}_{F_{m+1}%
}^{\sigma}\left\vert f\right\vert \right)  \ \mathbf{1}_{F_{m+1}\setminus
F_{m}}=\left(  \mathbb{E}_{F}^{\sigma}\left\vert f\right\vert \right)
\ \mathbf{1}_{F}+\sum_{m=0}^{N}\left(  \mathbb{E}_{\pi_{\mathcal{F}}^{m+1}%
F}^{\sigma}\left\vert f\right\vert \right)  \ \mathbf{1}_{\pi_{\mathcal{F}%
}^{m+1}F\setminus\pi_{\mathcal{F}}^{m}F}\\
& =\left(  \mathbb{E}_{F}^{\sigma}\left\vert f\right\vert \right)
\ \mathbf{1}_{F}+\sum_{F^{\prime}\in\mathcal{F}:\ F\subset F^{\prime}}\left(
\mathbb{E}_{\pi_{\mathcal{F}}F^{\prime}}^{\sigma}\left\vert f\right\vert
\right)  \ \mathbf{1}_{\pi_{\mathcal{F}}F^{\prime}\setminus F^{\prime}}\\
& \leq\alpha_{\mathcal{F}}\left(  F\right)  \ \mathbf{1}_{F}+\sum_{F^{\prime
}\in\mathcal{F}:\ F\subset F^{\prime}}\alpha_{\mathcal{F}}\left(
\pi_{\mathcal{F}}F^{\prime}\right)  \ \mathbf{1}_{\pi_{\mathcal{F}}F^{\prime
}\setminus F^{\prime}}\\
& \leq\alpha_{\mathcal{F}}\left(  F\right)  \ \mathbf{1}_{F}+\sum_{F^{\prime
}\in\mathcal{F}:\ F\subset F^{\prime}}\alpha_{\mathcal{F}}\left(
\pi_{\mathcal{F}}F^{\prime}\right)  \ \mathbf{1}_{\pi_{\mathcal{F}}F^{\prime}%
}\ \mathbf{1}_{F^{c}}\\
& =\alpha_{\mathcal{F}}\left(  F\right)  \ \mathbf{1}_{F}+\Phi\ \mathbf{1}%
_{F^{c}}\ ,\ \ \ \ \ \text{\ for all }F\in\mathcal{F},
\end{align*}
where%
\[
\Phi\equiv\sum_{F^{\prime\prime}\in\mathcal{F}}\ \alpha_{\mathcal{F}}\left(
F^{\prime\prime}\right)  \ \mathbf{1}_{F^{\prime\prime}}\ .
\]

Now we write%
\[
\sum_{F\in\mathcal{F}}\ \left\langle H_{\sigma}f_{F},g_{F}\right\rangle
_{\omega}=\sum_{F\in\mathcal{F}}\ \left\langle H_{\sigma}\left(
\mathbf{1}_{F}f_{F}\right)  ,g_{F}\right\rangle _{\omega}+\sum_{F\in
\mathcal{F}}\ \left\langle H_{\sigma}\left(  \mathbf{1}_{F^{c}}f_{F}\right)
,g_{F}\right\rangle _{\omega}\equiv I+II.
\]
Then local interval testing and quasi-orthogonality, together with the fact
that $f_{F}$ is a constant on $F$ bounded by $\alpha_{\mathcal{F}}\left(
F\right)  $, give%
\begin{align*}
& \left\vert I\right\vert =\left\vert \sum_{F\in\mathcal{F}}\int_{\mathbb{R}%
}\mathbf{1}_{F}\left(  x\right)  H_{\sigma}\left(  \mathbf{1}_{F}f_{F}\right)
\left(  x\right)  \ g_{F}\left(  x\right)  \ d\omega\left(  x\right)
\right\vert \\
& \leq\sqrt{\sum_{F\in\mathcal{F}}\int_{\mathbb{R}}\left\vert \alpha
_{\mathcal{F}}\left(  F\right)  \mathbf{1}_{F}\left(  x\right)  H_{\sigma
}\left(  \mathbf{1}_{F}\right)  \left(  x\right)  \right\vert ^{2}%
d\omega\left(  x\right)  }\ \sqrt{\sum_{F\in\mathcal{F}}\int_{\mathbb{R}%
}\left\vert g_{F}\left(  x\right)  \right\vert ^{2}d\omega\left(  x\right)
}\\
& \lesssim\mathfrak{T}_{H}^{\operatorname{loc}}\left(  \sigma,\omega\right)
\sqrt{\sum_{F\in\mathcal{F}}\int_{\mathbb{R}}\left\vert \alpha_{\mathcal{F}%
}\left(  F\right)  \right\vert ^{2}\left\vert F\right\vert _{\sigma}%
}\left\Vert g\right\Vert _{L^{2}\left(  \omega\right)  }\lesssim
\mathfrak{T}_{H}^{\operatorname{loc}}\left(  \sigma,\omega\right)  \left\Vert
f\right\Vert _{L^{2}\left(  \sigma\right)  }\left\Vert g\right\Vert
_{L^{2}\left(  \omega\right)  }.
\end{align*}

Now $\mathbf{1}_{F^{c}}f_{F}$ is supported outside $F$, and each $J$ in the
Haar support of $g_{F}=\mathsf{P}_{\mathcal{C}_{\mathcal{F}}\left(  F\right)
}^{\omega}g$ is either in $\mathcal{N}_{r}\left(  F\right)  =\left\{
J\in\mathcal{D}\left[  F\right]  :\ell\left(  J\right)  \geq2^{-r}\ell\left(
F\right)  \right\}  $, in which case the desired bound for term $I$ is
straightforward, or $J$ is $\left(  r,\varepsilon\right)  $-deeply embedded in
$F$, i.e. $J\subset_{r,\varepsilon}F$, and so $J\subset_{r,\varepsilon}W$ for
some $W\in\mathcal{M}_{\left(  r,\varepsilon\right)  -\operatorname*{deep}%
}\left(  F\right)  $. Thus with $\mathcal{C}_{\mathcal{F}}^{\flat}\left(
F\right)  \equiv\mathcal{C}_{\mathcal{F}}\left(  F\right)  \setminus
\mathcal{N}_{r}\left(  F\right)  $ we have%
\begin{align*}
\left\vert II\right\vert  & =\left\vert \sum_{F\in\mathcal{F}}\int
_{\mathbb{R}}H_{\sigma}\left(  \mathbf{1}_{F^{c}}f_{F}\right)  \left(
x\right)  \ \mathsf{P}_{\mathcal{C}_{\mathcal{F}}^{\flat}\left(  F\right)
}^{\omega}g\left(  x\right)  \ d\omega\left(  x\right)  \right\vert \\
& =\left\vert \sum_{F\in\mathcal{F}}\sum_{W\in\mathcal{M}_{\left(
r,\varepsilon\right)  -\operatorname*{deep}}\left(  F\right)  \cap
\mathcal{C}_{\mathcal{F}}\left(  F\right)  }\int_{\mathbb{R}}\mathsf{P}%
_{\mathcal{C}_{\mathcal{F}}^{\flat}\left(  F\right)  \cap\mathcal{D}\left[
W\right]  }^{\omega}H_{\sigma}\left(  \mathbf{1}_{F^{c}}f_{F}\right)  \left(
x\right)  \ \mathsf{P}_{\mathcal{C}_{\mathcal{F}}^{\flat}\left(  F\right)
\cap\mathcal{D}\left[  W\right]  }^{\omega}g\left(  x\right)  \ d\omega\left(
x\right)  \right\vert \\
& \lesssim\sqrt{\sum_{F\in\mathcal{F}}\sum_{W\in\mathcal{M}_{\left(
r,\varepsilon\right)  -\operatorname*{deep}}\left(  F\right)  \cap
\mathcal{C}_{\mathcal{F}}\left(  F\right)  }\int_{\mathbb{R}}\left\vert
\mathsf{P}_{\mathcal{C}_{\mathcal{F}}^{\flat}\left(  F\right)  \cap
\mathcal{D}\left[  W\right]  }^{\omega}H_{\sigma}\left(  \mathbf{1}_{F^{c}%
}f_{F}\right)  \left(  x\right)  \right\vert ^{2}d\omega\left(  x\right)  }\\
& \ \ \ \ \ \ \ \ \ \ \ \ \ \ \ \ \ \ \ \ \ \ \ \ \ \times\sqrt{\sum
_{F\in\mathcal{F}}\sum_{W\in\mathcal{M}_{\left(  r,\varepsilon\right)
-\operatorname*{deep}}\left(  F\right)  \cap\mathcal{C}_{\mathcal{F}}\left(
F\right)  }\int_{\mathbb{R}}\left\vert \mathsf{P}_{\mathcal{C}_{\mathcal{F}%
}^{\flat}\left(  F\right)  \cap\mathcal{D}\left[  W\right]  }^{\omega}g\left(
x\right)  \right\vert ^{2}d\omega\left(  x\right)  }.
\end{align*}
The second factor is bounded by $C\left\Vert g\right\Vert _{L^{2}\left(
\omega\right)  }$, and we use the second line in the Energy Lemma
\ref{Energy Lemma} on the first factor to bound it by,%
\begin{align}
& \sqrt{\sum_{F\in\mathcal{F}}\sum_{W\in\mathcal{M}_{\left(  r,\varepsilon
\right)  -\operatorname*{deep}}\left(  F\right)  \cap\mathcal{C}_{\mathcal{F}%
}\left(  F\right)  }\ \left(  \frac{\mathrm{P}\left(  W,\mathbf{1}_{F^{c}%
}f_{F}\sigma\right)  }{\ell\left(  W\right)  }\right)  ^{2}\int_{W}\left\vert
\mathsf{P}_{\mathcal{C}_{\mathcal{F}}\left(  F\right)  \cap\mathcal{D}\left[
W\right]  }^{\omega}x\right\vert ^{2}d\omega}\label{bound by}\\
& \leq\mathfrak{F}\left(  \sigma,\omega\right)  \lVert\Phi\rVert_{L^{p}\left(
\sigma\right)  }\lesssim\mathfrak{F}\left(  \sigma,\omega\right)  \lVert
f\rVert_{L^{2}\left(  \sigma\right)  },\nonumber
\end{align}
using the definition of the functional energy characteristic $\mathfrak{F}%
\left(  \sigma,\omega\right)  $, and the dyadic $\sigma$-maximal function
inequality $\left\Vert \Phi\right\Vert _{L^{p}\left(  \sigma\right)  }%
\lesssim\left\Vert M_{\sigma}f\right\Vert _{L^{p}\left(  \sigma\right)
}\lesssim\left\Vert f\right\Vert _{L^{p}\left(  \sigma\right)  }$. This
completes the proof of the Intertwining Proposition \ref{strongly adapted'}.
\end{proof}

Thus we have the following control of the $\operatorname{far}$ form,
\[
\left\vert \mathsf{B}_{\operatorname{far}}\left(  f,g\right)  \right\vert
\lesssim\left(  \mathfrak{F}\left(  \sigma,\omega\right)  +\mathfrak{T}%
_{H}\left(  \sigma,\omega\right)  +\mathcal{A}_{2}^{\operatorname{hole}%
}\left(  \sigma,\omega\right)  \right)  \ \left\Vert f\right\Vert
_{L^{2}\left(  \sigma\right)  }\left\Vert g\right\Vert _{L^{2}\left(
\omega\right)  }\ .
\]
Finally then, to control the functional energy characteristic $\mathfrak{F}%
\left(  \sigma,\omega\right)  $, we can of course appeal to the bound in
\cite{Hyt2}\footnote{The case of no common point masses is in \cite{LaSaShUr3}%
.} as a black box, to conclude that,%
\begin{equation}
\left\vert \mathsf{B}_{\operatorname{far}}\left(  f,g\right)  \right\vert
\lesssim\left(  \mathfrak{T}_{H,p}^{\ell^{2}}\left(  \sigma,\omega\right)
+\mathcal{A}_{2}^{\operatorname{hole}}\left(  \sigma,\omega\right)  \right)
\ \left\Vert f\right\Vert _{L^{2}\left(  \sigma\right)  }\left\Vert
g\right\Vert _{L^{2}\left(  \omega\right)  }\ ,\label{B far}%
\end{equation}
thus completing our treatment of the $\operatorname{far}$ form.

But instead, we prefer to use a different proof based on Theorem 1 of
\cite{Saw3}, rather than the Poisson estimate in Theorem 2 of \cite{Saw3} that
was used in \cite{LaSaShUr3} and generalized to handle holes in \cite{Hyt2}.
We will use the following smaller \emph{refined} functional energy inequality
introduced in \cite{SaWi},%
\[
\sum_{F\in\mathcal{F}}\sum_{W\in\mathcal{M}_{\left(  r,\varepsilon\right)
-\operatorname*{deep}}\left(  F\right)  \cap\mathcal{C}_{\mathcal{F}}\left(
F\right)  }\left(  \frac{\mathrm{P}\left(  W,h\mathbf{1}_{F^{c}}\sigma\right)
}{\ell\left(  W\right)  }\right)  ^{2}\int_{W}\left\vert \mathsf{P}%
_{\mathcal{C}_{\mathcal{F}}\left(  F\right)  \cap\mathcal{D}\left[  W\right]
}^{\omega}x\right\vert ^{2}d\omega\leq\mathfrak{F}\left(  \sigma
,\omega\right)  \lVert h\rVert_{L^{2}\left(  \sigma\right)  }\ ,
\]
in which the sum in $W\in\mathcal{M}_{\left(  r,\varepsilon\right)
-\operatorname*{deep}}\left(  F\right)  $ is restricted to the corona
$\mathcal{C}_{\mathcal{F}}\left(  F\right)  $. This smaller functional energy
also serves to control the Intertwining Proposition by an inspection of
(\ref{bound by}) above, and we can now give the simpler proof based on Theorem
1 in \cite{Saw3} instead of on Theorem 2 in \cite{Saw3} as done in all other
treatments of the $\operatorname*{far}$ form.

The interval $F=F_{W}\in\mathcal{F}$ is uniquely determined by the condition
$W\in\mathcal{M}_{\left(  r,\varepsilon\right)  -\operatorname*{deep}}\left(
F_{W}\right)  \cap\mathcal{C}_{\mathcal{F}}\left(  F\right)  $, and so we can
define the connected set of intervals $\mathcal{C}\left(  W\right)
\equiv\mathcal{C}_{\mathcal{F}}\left(  F_{W}\right)  \cap\mathcal{D}\left[
W\right]  $. Then using%
\[
\Lambda\equiv\overset{\cdot}{\bigcup}_{F\in\mathcal{F}}\mathcal{M}_{\left(
r,\varepsilon\right)  -\operatorname*{deep}}\left(  F\right)  \cap
\mathcal{C}_{\mathcal{F}}\left(  F\right)  ,
\]
where $\overset{\cdot}{\bigcup}$ denotes that the sets $\mathcal{M}_{\left(
r,\varepsilon\right)  -\operatorname*{deep}}\left(  F\right)  \cap
\mathcal{C}_{\mathcal{F}}\left(  F\right)  $ are pairwise disjoint, together
with the orthogonality of the projections $\left\{  \mathsf{P}_{\mathcal{C}%
_{\mathcal{F}}\left(  F\right)  \cap\mathcal{D}\left[  W\right]  }^{\omega
}x\right\}  _{\substack{F\in\mathcal{F} \\W\in\mathcal{M}_{\left(
r,\varepsilon\right)  -\operatorname*{deep}}\left(  F\right)  \cap
\mathcal{C}_{\mathcal{F}}\left(  F\right)  }}$ in $L^{2}\left(  \omega\right)
$, we see that this refined functional energy inequality is equivalent to the
boundedness from $L^{2}\left(  \sigma\right)  $ to $L^{2}\left(
\omega\right)  $ of the operator $Th\left(  x\right)  \equiv\int_{\mathbb{R}%
}K\left(  x,y\right)  h\left(  y\right)  d\sigma\left(  y\right)  $ with
kernel,
\begin{equation}
K\left(  x,y\right)  \equiv\sum_{W\in\Lambda}\frac{\mathbf{1}_{R\setminus
W^{\ast}}\left(  y\right)  }{\left(  \ell\left(  W\right)  +\left\vert
y-c_{W}\right\vert \right)  ^{2}}\mathsf{P}_{\mathcal{C}\left(  W\right)
}^{\omega}Z\left(  x\right)  ,\label{def K}%
\end{equation}
where $Z$ denotes the identity map $Z\left(  x\right)  =x$ on $\mathbb{R}$,
and $K^{\ast}$ is the double of $K$.

This is in turn equivalent to boundedness of the operator $T^{\ast}T$ from
$L^{2}\left(  \sigma\right)  $ to $L^{2}\left(  \sigma\right)  $, where the
kernel of $T^{\ast}T$ is given by%
\begin{align*}
& L\left(  y^{\prime},y\right)  =\int_{\mathbb{R}}\int_{\mathbb{R}}K\left(
y^{\prime},x\right)  K\left(  x,y\right)  d\omega\left(  x\right) \\
& =\int_{\mathbb{R}}\left\{  \sum_{W^{\prime}\in\Lambda}\frac{\mathbf{1}%
_{R\setminus W^{\prime\ast}}\left(  y^{\prime}\right)  }{\left(  \ell\left(
W^{\prime}\right)  +\left\vert y^{\prime}-c_{W^{\prime}}\right\vert \right)
^{2}}\mathsf{P}_{\mathcal{C}\left(  W^{\prime}\right)  }^{\omega}Z\left(
x\right)  \right\}  \left\{  \sum_{W\in\Lambda}\frac{\mathbf{1}_{R\setminus
W^{\ast}}\left(  y\right)  }{\left(  \ell\left(  W\right)  +\left\vert
y-c_{W}\right\vert \right)  ^{2}}\mathsf{P}_{\mathcal{C}\left(  W\right)
}^{\omega}Z\left(  x\right)  \right\}  d\omega\left(  x\right) \\
& =\sum_{W^{\prime}\in\Lambda}\frac{\mathbf{1}_{R\setminus W^{\prime\ast}%
}\left(  y^{\prime}\right)  }{\left(  \ell\left(  W^{\prime}\right)
+\left\vert y^{\prime}-c_{W^{\prime}}\right\vert \right)  ^{2}}\left\{
\int_{\mathbb{R}}\mathsf{P}_{\mathcal{C}\left(  W^{\prime}\right)  }^{\omega
}Z\left(  x\right)  \mathsf{P}_{\mathcal{C}\left(  W\right)  }^{\omega
}Z\left(  x\right)  d\omega\left(  x\right)  \right\}  \sum_{W\in\Lambda}%
\frac{\mathbf{1}_{R\setminus W^{\ast}}\left(  y\right)  }{\left(  \ell\left(
W\right)  +\left\vert y-c_{W}\right\vert \right)  ^{2}}\\
& =\sum_{W\in\Lambda}\left\{  \int_{\mathbb{R}}\left\vert \mathsf{P}%
_{\mathcal{C}\left(  W\right)  }^{\omega}Z\left(  x\right)  \right\vert
^{2}d\omega\left(  x\right)  \right\}  \frac{\mathbf{1}_{R\setminus W^{\ast}%
}\left(  y^{\prime}\right)  }{\left(  \ell\left(  W\right)  +\left\vert
y^{\prime}-c_{W}\right\vert \right)  ^{2}}\frac{\mathbf{1}_{R\setminus
W^{\ast}}\left(  y\right)  }{\left(  \ell\left(  W\right)  +\left\vert
y-c_{W}\right\vert \right)  ^{2}},
\end{align*}
where in the final line we have used that $W^{\prime}\neq W$ implies
$\mathcal{C}\left(  W^{\prime}\right)  \cap\mathcal{C}\left(  W\right)
=\emptyset$ and hence that $\mathsf{P}_{\mathcal{C}\left(  W^{\prime}\right)
}^{\omega}Z$ and $\mathsf{P}_{\mathcal{C}\left(  W\right)  }^{\omega}Z$ are
orthogonal in $L^{2}\left(  \omega\right)  $.

Now each of the kernels $J_{W}\left(  y^{\prime},y\right)  \equiv
\frac{\mathbf{1}_{R\setminus W^{\ast}}\left(  y^{\prime}\right)  }{\left(
\ell\left(  W\right)  +\left\vert y^{\prime}-c_{W}\right\vert \right)  ^{2}%
}\frac{\mathbf{1}_{R\setminus W^{\ast}}\left(  y\right)  }{\left(  \ell\left(
W\right)  +\left\vert y-c_{W}\right\vert \right)  ^{2}}$ is symmetric and
satisfies the following hypotheses uniformly in $W$,

\begin{enumerate}
\item $J_{W}\left(  y^{\prime},y\right)  $ is essentially decreasing in the
variable $y$ away from $y^{\prime}$, i.e.%
\[
J_{W}\left(  x,y_{1}\right)  \leq CJ_{W}\left(  x,y_{2}\right)
,\ \ \ \ \ \left\vert y_{1}-x\right\vert \geq\left\vert y_{2}-x\right\vert ,
\]

\item $J_{W}\left(  y^{\prime},y\right)  $ is roughly constant in the variable
$y$ on intervals away from $y^{\prime}$, i.e.%
\[
\frac{1}{C}\leq\frac{J_{W}\left(  y^{\prime},y_{1}\right)  }{J_{W}\left(
y^{\prime},y_{2}\right)  }\leq C,\ \ \ \ \ y_{1},y_{2}\in I\text{ with
}y^{\prime}\notin2I.
\]

\end{enumerate}

Properties (1) and (2) are easy consequences of the formula for the Poisson
kernel, see \cite[(7.10) and (7.11)]{SaWi} for details. Since $L$ is a
positive sum of the kernels $J_{W}$ we see that the symmetic operator $L$
satisfies the two properties (1) and (2) above, i.e. $L$ is essentially
decreasing in the variable $y$ away from $y^{\prime}$, and is roughly constant
in the variable $y$ on intervals away from $y^{\prime}$.

Now as shown in \cite[Section 11, the appendix]{SaWi}, the full testing
condition for $T$ can be controlled as in \cite[Subsections 7.3 and 7.4]{SaWi}
(the case $p=2$ simplifies matters somewhat), i.e.%
\[
\mathfrak{FT}_{T}\left(  \sigma,\omega\right)  \leq C\left[  \mathfrak{T}%
_{H}^{\operatorname{loc}}\left(  \sigma,\omega\right)  +\mathcal{A}%
_{2}^{\operatorname*{hole}}\left(  \sigma,\omega\right)  \right]  .
\]
See the appendix for a detailed proof oof this.

Thus we obtain that the single full testing condition for the operator
$T^{\ast}T$ satisfies,%
\begin{align}
\mathfrak{FT}_{T^{\ast}T}\left(  \sigma,\sigma\right)    & =\sup_{I}%
\frac{\left\Vert T^{\ast}T\mathbf{1}_{I}\right\Vert _{L^{2}\left(
\sigma\right)  }}{\left\Vert \mathbf{1}_{I}\right\Vert _{L^{2}\left(
\sigma\right)  }}=\sup_{I}\sup_{\left\Vert g\right\Vert _{L^{2}\left(
\sigma\right)  }\leq1}\frac{\left\vert \left\langle T^{\ast}T\mathbf{1}%
_{I},g\right\rangle \right\vert }{\left\Vert \mathbf{1}_{I}\right\Vert
_{L^{2}\left(  \sigma\right)  }}\label{single full test}\\
& =\sup_{I}\sup_{\left\Vert g\right\Vert _{L^{2}\left(  \sigma\right)  }\leq
1}\frac{\left\vert \left\langle T\mathbf{1}_{I},Tg\right\rangle \right\vert
}{\left\Vert \mathbf{1}_{I}\right\Vert _{L^{2}\left(  \sigma\right)  }%
}\lesssim\mathfrak{FT}_{T}\left(  \sigma,\omega\right)  \mathfrak{N}%
_{T}\left(  \sigma,\omega\right)  .\nonumber
\end{align}

At this point we recall the two weight norm inequality for potential operators
in \cite[Theorem 1]{Saw3}, an appeal to a straightforward generalization of
this theorem to \textbf{non}convolution operators $L$ satisfying the two
properties (1) and (2) above, i.e.%
\[
\mathfrak{N}_{L}\left(  \sigma,\omega\right)  =\left\Vert L\right\Vert
_{L^{p}\left(  \sigma\right)  \rightarrow L^{p}\left(  \omega\right)  }%
\approx\mathfrak{FT}_{L}\left(  \sigma,\omega\right)  +\mathfrak{FT}_{L^{\ast
}}\left(  \omega,\sigma\right)  ,
\]
where $\mathfrak{FT}_{L}\left(  \sigma,\omega\right)  $ and $\mathfrak{FT}%
_{L^{\ast}}\left(  \omega,\sigma\right)  $ denote the full testing constants,%
\[
\mathfrak{FT}_{L}\left(  \sigma,\omega\right)  \equiv\sup_{I}\frac{\left\Vert
L\mathbf{1}_{I}\right\Vert _{L^{p}\left(  \omega\right)  }}{\left\Vert
\mathbf{1}_{I}\right\Vert _{L^{p}\left(  \sigma\right)  }}\text{ and
}\mathfrak{FT}_{L^{\ast}}\left(  \omega,\sigma\right)  \equiv\sup_{I}%
\frac{\left\Vert L^{\ast}\mathbf{1}_{I}\right\Vert _{L^{p^{\prime}}\left(
\sigma\right)  }}{\left\Vert \mathbf{1}_{I}\right\Vert _{L^{p^{\prime}}\left(
\omega\right)  }}.
\]
The equal weights case of this generalization shows that $L=T^{\ast}T$ is
bounded from $L^{2}\left(  \sigma\right)  $ to $L^{2}\left(  \sigma\right)
$\footnote{We only need the one weight version of this generalization.} with
norm%
\[
\mathfrak{N}_{T^{\ast}T}\left(  \sigma,\sigma\right)  \approx\mathfrak{FT}%
_{T^{\ast}T}\left(  \sigma,\sigma\right)  \leq\mathfrak{FT}_{T}\left(
\sigma,\omega\right)  \mathfrak{N}_{T}\left(  \sigma,\omega\right)  ,
\]
where the final inequality is (\ref{single full test}). However, Hilbert space
theory shows that $\mathfrak{N}_{T^{\ast}T}\left(  \sigma,\omega\right)
=\mathfrak{N}_{T}\left(  \sigma,\omega\right)  ^{2}$, and so we conclude that
$\mathfrak{N}_{T}\left(  \sigma,\omega\right)  \leq C\mathfrak{FT}_{T}\left(
\sigma,\omega\right)  $. Thus $T$ is bounded from $L^{2}\left(  \sigma\right)
$ to $L^{2}\left(  \omega\right)  $ with norm at most
\[
C\mathfrak{FT}_{T}\left(  \sigma,\omega\right)  \leq C\left[  \mathfrak{T}%
_{H}^{\operatorname{loc}}\left(  \sigma,\omega\right)  +\mathcal{A}%
_{2}^{\operatorname*{hole}}\left(  \sigma,\omega\right)  \right]  ,
\]
which was shown above to imply that the refined functional energy inequality
is controlled by the testing condition for the Hilbert transform with weight
pair $\left(  \sigma,\omega\right)  $, and the one-tailed Muckenhoupt
characteristic with holes for the weight pair $\left(  \sigma,\omega\right)
$. This completes our alternate proof of control of the $\operatorname*{far}$
form in (\ref{B far}).

\section{Reduction of the diagonal form by the NTV reach}

We first apply the\ clever `NTV reach' of \cite{NTV3}, which splits the
diagonal form%
\[
\mathsf{B}_{\operatorname{diag}}\left(  f,g\right)  =\sum_{\left(  I,J\right)
\in\mathcal{P}_{\operatorname{diag}}}\left\langle H_{\sigma}\left(
\mathbf{1}_{I_{J}}\bigtriangleup_{I}^{\sigma}f\right)  ,\bigtriangleup
_{J}^{\omega}g\right\rangle _{\omega}=\sum_{F\in\mathcal{F}}\sum
_{\substack{\left(  I,J\right)  \in\mathcal{C}_{\mathcal{F}}\left(  F\right)
\times\mathcal{C}_{\mathcal{F}}\left(  F\right)  \\J\subset_{\tau}I}}\left(
E_{I_{J}}^{\sigma}\bigtriangleup_{I}^{\sigma}f\right)  \left\langle H_{\sigma
}\mathbf{1}_{I_{J}}f,\bigtriangleup_{J}^{\omega}g\right\rangle _{\omega},
\]
into a paraproduct and stopping form,%
\begin{align*}
\mathsf{B}_{\operatorname{diag}}\left(  f,g\right)   & =\sum_{F\in\mathcal{F}%
}\sum_{\substack{\left(  I,J\right)  \in\mathcal{C}_{\mathcal{F}}\left(
F\right)  \times\mathcal{C}_{\mathcal{F}}\left(  F\right)  \\J\subset_{\tau}%
I}}\left(  E_{I_{J}}^{\sigma}\bigtriangleup_{I}^{\sigma}f\right)  \left\langle
H_{\sigma}\mathbf{1}_{F},\bigtriangleup_{J}^{\omega}g\right\rangle _{\omega}\\
& +\sum_{F\in\mathcal{F}}\sum_{\substack{\left(  I,J\right)  \in
\mathcal{C}_{\mathcal{F}}\left(  F\right)  \times\mathcal{C}_{\mathcal{F}%
}\left(  F\right)  \\J\subset_{\tau}I}}\left(  E_{I_{J}}^{\sigma
}\bigtriangleup_{I}^{\sigma}f\right)  \left\langle H_{\sigma}\mathbf{1}%
_{F\setminus I_{J}},\bigtriangleup_{J}^{\omega}g\right\rangle _{\omega}\\
& \equiv\sum_{F\in\mathcal{F}}\mathsf{B}_{\operatorname*{para}}^{F}\left(
f,g\right)  +\sum_{F\in\mathcal{F}}\mathsf{B}_{\operatorname*{stop}}%
^{F}\left(  f,g\right) \\
& \equiv\mathsf{B}_{\operatorname*{para}}\left(  f,g\right)  +\mathsf{B}%
_{\operatorname*{stop}}\left(  f,g\right)  .
\end{align*}

\subsection{Control of the paraproduct form by the testing characteristic}

We control the local paraproduct form $\mathsf{B}_{\operatorname*{para}}%
^{F}\left(  f,g\right)  $ by the testing condition for $H$. Indeed, from the
telescoping identity for Haar projections, see e.g. \cite{LaSaUr2}, we have%
\begin{align*}
& \mathsf{B}_{\operatorname*{para}}^{F}\left(  f,g\right)  =\sum
_{\substack{I\in\mathcal{C}_{\mathcal{F}}\left(  F\right)  \text{ and }%
J\in\mathcal{C}_{\mathcal{F}}\left(  F\right)  \\J\subset I\text{ and }%
\ell\left(  J\right)  \leq2^{-\mathbf{r}}\ell\left(  I\right)  }}\left(
E_{I_{J}}^{\sigma}\bigtriangleup_{I}^{\sigma}f\right)  \left\langle H_{\sigma
}\mathbf{1}_{F},\bigtriangleup_{J}^{\omega}g\right\rangle _{\omega}\\
& =\sum_{J\in\mathcal{C}_{\mathcal{F}}\left(  F\right)  }\left\langle
H_{\sigma}\mathbf{1}_{F},\bigtriangleup_{J}^{\omega}g\right\rangle _{\omega
}\left\{  \sum_{I\in\mathcal{C}_{F}\text{:\ }J\subset I\text{ and }\ell\left(
J\right)  \leq2^{-r}\ell\left(  I\right)  }E_{I_{J}}^{\sigma}\bigtriangleup
_{I}^{\sigma}f\right\} \\
& =\sum_{J\in\mathcal{C}_{\mathcal{F}}\left(  F\right)  }\left\langle
H_{\sigma}\mathbf{1}_{F},\bigtriangleup_{J}^{\omega}g\right\rangle _{\omega
}\left\{  E_{I^{\natural}\left(  J\right)  _{J}}^{\sigma}f-E_{F}^{\sigma
}f\right\}  =\left\langle H_{\sigma}\mathbf{1}_{F},\sum_{J\in\mathcal{C}%
_{\mathcal{F}}\left(  F\right)  }\left\{  E_{I^{\natural}\left(  J\right)
_{J}}^{\sigma}f-E_{F}^{\sigma}f\right\}  \bigtriangleup_{J}^{\omega
}g\right\rangle _{\omega}\ ,
\end{align*}
where $I^{\natural}\left(  J\right)  $ is the smallest $I\in\mathcal{C}%
_{\mathcal{F}}\left(  F\right)  \cap\Lambda_{f}^{\sigma}$ (where $\Lambda
_{f}^{\sigma}$ is the Haar support of $f)$, that contains $J$ and satisfies
$J\subset_{r}I$.

Thus from Cauchy-Schwarz and the bound on the coefficients $\lambda_{J}\equiv
E_{I^{\natural}\left(  J\right)  _{J}}^{\sigma}f-E_{F}^{\sigma}f$ given by
$\left\vert \lambda_{J}\right\vert \lesssim\alpha_{\mathcal{F}}\left(
F\right)  $, we have%
\begin{align}
& \left\vert \mathsf{B}_{\operatorname*{para}}^{F}\left(  f,g\right)
\right\vert =\left\vert \left\langle H_{\sigma}\mathbf{1}_{F},\sum
_{J\in\mathcal{C}_{\mathcal{F}}\left(  F\right)  }\left\{  E_{I^{\natural
}\left(  J\right)  _{J}}^{\sigma}f-E_{F}^{\sigma}f\right\}  \bigtriangleup
_{J}^{\omega}g\right\rangle _{\omega}\right\vert \label{est para}\\
& \leq\left\Vert \mathbf{1}_{F}H_{\sigma}\mathbf{1}_{F}\right\Vert
_{L^{2}\left(  \omega\right)  }\left\Vert \sum_{J\in\mathcal{C}_{\mathcal{F}%
}\left(  F\right)  }\lambda_{J}\bigtriangleup_{J}^{\omega}g\right\Vert
_{L^{2}\left(  \omega\right)  }\lesssim\alpha_{\mathcal{F}}\left(  F\right)
\ \left\Vert \mathbf{1}_{F}H_{\sigma}\mathbf{1}_{F}\right\Vert _{L^{2}\left(
\omega\right)  }\ \left\Vert \sum_{J\in\mathcal{C}_{\mathcal{F}}\left(
F\right)  }\bigtriangleup_{J}^{\omega}g\right\Vert _{L^{2}\left(
\omega\right)  }\nonumber\\
& \leq\mathfrak{T}_{H}\left(  \sigma,\omega\right)  \ \alpha_{\mathcal{F}%
}\left(  F\right)  \ \sqrt{\left\vert F\right\vert _{\sigma}}\ \left\Vert
\mathsf{P}_{\mathcal{C}_{\mathcal{F}}\left(  F\right)  }^{\omega}g\right\Vert
_{L^{2}\left(  \omega\right)  }.\nonumber
\end{align}
Then summing in $F$ we obtain the following bound for the paraproduct form
$\mathsf{B}_{\operatorname*{para}}\left(  f,g\right)  \equiv\sum
_{F\in\mathcal{F}}\mathsf{B}_{\operatorname*{para}}^{F}\left(  f,g\right)  $,%
\begin{align*}
\left\vert \mathsf{B}_{\operatorname*{para}}\left(  f,g\right)  \right\vert  &
\leq\sum_{F\in\mathcal{F}}\left\vert \mathsf{B}_{\operatorname*{para}}%
^{F}\left(  f,g\right)  \right\vert \leq\sum_{F\in\mathcal{F}}\mathfrak{T}%
_{H}\left(  \sigma,\omega\right)  \ \alpha_{\mathcal{F}}\left(  F\right)
\ \sqrt{\left\vert F\right\vert _{\sigma}}\ \left\Vert \mathsf{P}%
_{\mathcal{C}_{\mathcal{F}}\left(  F\right)  }^{\omega}g\right\Vert
_{L^{2}\left(  \omega\right)  }\\
& \leq\mathfrak{T}_{H}\left(  \sigma,\omega\right)  \sqrt{\sum_{F\in
\mathcal{F}}\alpha_{\mathcal{F}}\left(  F\right)  ^{2}\ \left\vert
F\right\vert _{\sigma}}\sqrt{\sum_{F\in\mathcal{F}}\left\Vert \mathsf{P}%
_{\mathcal{C}_{\mathcal{F}}\left(  F\right)  }^{\omega}g\right\Vert
_{L^{2}\left(  \omega\right)  }^{2}}\\
& \lesssim\mathfrak{T}_{H}\left(  \sigma,\omega\right)  \left\Vert
f\right\Vert _{L^{2}(\sigma)}\left\Vert g\right\Vert _{L^{2}(\omega)}.
\end{align*}
by quasi-orthogonality in (\ref{qorth}) and orthogonality in Haar projections
$\mathsf{P}_{\mathcal{C}_{\mathcal{F}}\left(  F\right)  }^{\omega}g$.

\subsection{Control of the stopping form by the Poisson-Energy characteristic}

To control the stopping form, it suffices to assume the Haar supports of $f$
and $g$ are contained in a fixed large, but finite subset $\digamma$ of the
grid $\mathcal{D}$, and to uniformly control over these subsets, each local
stopping form%
\[
\mathsf{B}_{\operatorname*{stop}}^{F}\left(  f,g\right)  \equiv\sum
_{\substack{I\in\mathcal{C}_{F}\text{ and }J\in\mathcal{C}_{\mathcal{F}%
}\left(  F\right)  \\J\subset I\text{ and }\ell\left(  J\right)  \leq2^{-\tau
}\ell\left(  I\right)  }}\left(  E_{I_{J}}^{\sigma}\bigtriangleup_{I}^{\sigma
}f\right)  \left\langle H_{\sigma}\left(  \mathbf{1}_{F\setminus I_{J}%
}\right)  ,\bigtriangleup_{J}^{\omega}g\right\rangle _{\omega}.
\]
Indeed, then quasi-orthogonality in (\ref{qorth}) and orthogonality in Haar
projections will be used to finish control of the stopping form. In fact we
will prove%
\begin{equation}
\left\vert \mathsf{B}_{\operatorname*{stop}}^{F}\left(  f,g\right)
\right\vert \leq C\mathrm{P}\mathsf{E}_{\mathcal{F}}\left(  \sigma
,\omega\right)  \left\Vert f\right\Vert _{L^{2}(\sigma)}\left\Vert
g\right\Vert _{L^{2}(\omega)}\ ,\ \ \ \ \ \text{for all }F\in\mathcal{F}%
,\label{stop est}%
\end{equation}
where $\mathrm{P}\mathsf{E}_{\mathcal{F}}\left(  \sigma,\omega\right)  $ is
defined in (\ref{PE char}). The key technical estimate needed is the following
Stopping Child Lemma, which is a reformulation of the `straddling' lemmas in
M. Lacey \cite[Lemmas 3.19 and 3.16]{Lac}.

\subsubsection{The Stopping Child Lemma}

Let$\,\mathcal{F}$ be a collection of $\operatorname{good}$\ stopping times.
Suppose $F\in\mathcal{F}$ is fixed for the moment and let $\mathcal{A}%
\subset\mathcal{C}_{\mathcal{F}}\left(  F\right)  $ be a collection of
$\operatorname{good}$ stopping times with top interval $F$. For $A\in
\mathcal{A}$ we define the `stopping child' bilinear form,%
\begin{align*}
\mathsf{B}_{\mathcal{A}}^{A}\left(  f,g\right)   & \equiv\sum_{S\in
\mathfrak{C}_{\mathcal{A}}\left(  A\right)  }\sum_{\substack{\left(
I,J\right)  \in\left(  S,A\right]  \times\mathcal{D}\left[  S\right]
\\J\subset_{\tau}I}}\left(  E_{I_{J}}^{\sigma}\bigtriangleup_{I}^{\sigma
}f\right)  \left\langle H_{\sigma}\left(  \mathbf{1}_{F\setminus S}\right)
,\bigtriangleup_{J}^{\omega}g\right\rangle _{\omega}\\
& =\sum_{S\in\mathfrak{C}_{\mathcal{A}}\left(  A\right)  }\sum_{J\in
\mathcal{D}\left[  S\right]  :\ J\subset_{\tau}\Lambda_{f}^{\sigma}\left[
S\right]  }\left\langle H_{\sigma}\varphi_{J}^{S},\bigtriangleup_{J}^{\omega
}g\right\rangle _{\omega}\ ,
\end{align*}
where $\left(  I,J\right]  \equiv\left\{  K\in\mathcal{D}:J\subset
K\subsetneqq I\right\}  $ denotes the tower in the grid $\mathcal{D}$ with
endpoints $I$ (not included) and $J$ (included), and $\Lambda_{f}^{\sigma
}\left[  S\right]  $ is the smallest interval in the Haar support $\Lambda
_{f}^{\sigma}$ of $f$ that contains $S$, and $\varphi_{J}^{S}\equiv\sum
_{I\in\left(  S,A\right]  :\ J\subset_{\tau}\Lambda_{f}^{\sigma}\left[
S\right]  }\left(  E_{I_{J}}^{\sigma}\bigtriangleup_{I}^{\sigma}f\right)
\mathbf{1}_{A_{0}\setminus S}$. The presence of the indicator $\mathbf{1}%
_{F\setminus S}$ suggests the name `stopping child' bilinear form. We also
define a refined Poisson-Energy characteristic by,%
\begin{equation}
\mathrm{P}\mathsf{E}_{\mathcal{A}}^{A,\operatorname*{trip}}\left(  \Lambda
_{g}^{\omega};\sigma,\omega\right)  \equiv\sup_{S\in\mathfrak{C}_{\mathcal{A}%
}\left(  A\right)  }\sup_{K\in\mathcal{W}_{\operatorname{good}}%
^{F,\operatorname*{trip}}\left(  S\right)  }\frac{1}{\sqrt{\left\vert
K\right\vert _{\sigma}}}\frac{\mathrm{P}\left(  K,\mathbf{1}_{F\setminus
S}\sigma\right)  }{\ell\left(  K\right)  }\sqrt{\sum_{J\in\Lambda_{g}^{\omega
}\cap\mathcal{D}\left[  K\right]  }\left\Vert \bigtriangleup_{J}^{\omega
}x\right\Vert _{L^{2}\left(  \omega\right)  }^{2}}\ ,\label{refined PE char}%
\end{equation}
where $\Lambda_{g}^{\omega}\equiv\left\{  J\in\mathcal{D}:\bigtriangleup
_{J}^{\omega}g\neq0\right\}  $ is the Haar support of $g$ in $L^{2}\left(
\omega\right)  $, and%
\[
\mathcal{W}_{\operatorname{good}}^{F,\operatorname*{trip}}\left(  S\right)
\equiv\left[  \left\{  S\right\}  \cup\mathcal{M}_{\operatorname{good}%
}^{\operatorname*{trip}}\left(  S\right)  \right]  \cap\mathcal{C}%
_{\mathcal{F}}\left(  F\right)  ,
\]
where $\mathcal{M}_{\operatorname{good}}^{\operatorname*{trip}}\left(
S\right)  $ is the collection of maximal $\operatorname{good}$ subintervals
$I$ of $S$ whose \emph{triples} are contained in $S$. This characteristic
depends on $\mathcal{A}$, $A$ and $\Lambda_{g}^{\omega}$, as well as on
$\sigma$ and $\omega$, and is a variation on the crucial size function
introduced by M. Lacey in \cite{Lac}. Set
\[
\mathrm{P}\mathsf{E}_{\mathcal{A}\left[  F\right]  }^{\operatorname*{trip}%
}\left(  \Lambda_{g}^{\omega};\sigma,\omega\right)  \equiv\sup_{A\in
\mathcal{A}\left[  F\right]  }\mathrm{P}\mathsf{E}_{\mathcal{A}}%
^{A,\operatorname*{trip}}\left(  \Lambda_{g}^{\omega};\sigma,\omega\right)  ,
\]
and note the trivial inequality
\begin{equation}
\mathrm{P}\mathsf{E}_{\mathcal{A}\left[  F\right]  }^{\operatorname*{trip}%
}\left(  \Lambda_{g}^{\omega};\sigma,\omega\right)  \leq\mathrm{P}%
\mathsf{E}_{\mathcal{F}}^{F}\left(  \sigma,\omega\right)  \leq\mathrm{P}%
\mathsf{E}_{\mathcal{F}}\left(  \sigma,\omega\right)  ,\label{triv inequ}%
\end{equation}
where $\mathrm{P}\mathsf{E}_{\mathcal{F}}^{F}\left(  \sigma,\omega\right)  $
and $\mathrm{P}\mathsf{E}_{\mathcal{F}}\left(  \sigma,\omega\right)  $ are
defined in (\ref{PE char}).

\begin{lemma}
[Stopping Child Lemma]\label{straddle 3}(a reformulation of \cite[Lemmas 3.19
and 3.26]{Lac}) Let $f\in L^{2}\left(  \sigma\right)  $, $g\in L^{2}\left(
\omega\right)  $ have $\operatorname{good}$ Haar supports, along with their
children, let $\mathcal{F}$ be a collection of $\operatorname{good}$\ stopping
times, and let $\mathcal{A}\subset\mathcal{C}_{\mathcal{F}}\left(  F\right)  $
be a collection of $\operatorname{good}$ stopping times with top $F$. Set
\[
\alpha_{A}\left(  S\right)  \equiv\sup_{I\in\left(  \Lambda_{f}^{\sigma}
\left[  S\right]  ,A\right]  \cap\mathcal{D}_{\operatorname{good}}}\left\vert
E_{I}^{\sigma}f\right\vert ,\ \ \ \ \ \text{for }S\in\mathfrak{C}%
_{\mathcal{A}}\left(  A\right)  ,\ A\in\mathcal{A}.
\]
Then for all $A\in\mathcal{A}$ we have the nonlinear bound,
\[
\left\vert \mathsf{B}_{\mathcal{A}}^{A}\left(  f,g\right)  \right\vert \leq
C\mathrm{P}\mathsf{E}_{\mathcal{A}}^{A,\operatorname*{trip}}\left(
\Lambda_{g}^{\omega};\sigma,\omega\right)  \sqrt{\sum_{S\in\mathfrak{C}%
_{\mathcal{A}}\left(  A\right)  }\left\vert S\right\vert _{\sigma}\alpha
_{A}\left(  S\right)  ^{2}}\left\Vert g\right\Vert _{L^{2}\left(
\omega\right)  }.
\]

\end{lemma}

\begin{proof}
By the telescoping property of martingale differences, together with the bound
$\alpha_{A}\left(  S\right)  $ on the averages of $\mathsf{P}_{\mathcal{C}%
_{\mathcal{A}}\left(  A\right)  }^{\sigma}f$ in the $\operatorname{good}$
intervals in the tower $\left(  \Lambda_{f}^{\sigma}\left[  S\right]
,A\right]  $, and the goodness assumption on $f$, we have%
\begin{equation}
\left\vert \varphi_{J}^{S}\right\vert =\left\vert \sum_{I\in\left(
\Lambda_{f}^{\sigma}\left[  S\right]  ,A\right]  :\ J\subset_{\tau}I}\left(
E_{I_{J}}^{\sigma}\bigtriangleup_{I}^{\sigma}f\right)  \mathbf{1}%
_{A_{0}\setminus I_{J}}\right\vert \lesssim\alpha_{A}\left(  S\right)
\mathbf{1}_{F\setminus S}\ .\label{bfi 3}%
\end{equation}
From the Monotonicity Lemma and the fact that%
\[
\Lambda_{g}^{\omega}\cap\mathcal{D}\left[  S\right]  \subset\left(
\bigcup_{K\in\mathcal{W}_{\operatorname{good}}^{F}\left(  S\right)  }K\right)
\cup\mathcal{N}_{\tau}\left(  S\right)
\]
where $\mathcal{N}_{\tau}\left(  S\right)  \equiv\left\{  J\subset
S:\ell\left(  J\right)  \geq2^{-\tau}\ell\left(  S\right)  \right\}  $ is the
set of `$\tau$-nearby' dyadic intervals in $S$, we have%
\begin{align*}
& \left\vert \mathsf{B}_{\mathcal{A}}^{A}\left(  f,g\right)  \right\vert
\leq\left\vert \sum_{S\in\mathfrak{C}_{\mathcal{A}}\left(  A\right)  }%
\sum_{J\in\mathcal{D}\left[  S\right]  :\ J\subset_{\tau}\Lambda_{f}^{\sigma
}\left[  S\right]  }\left\langle H_{\sigma}\varphi_{J}^{S},\bigtriangleup
_{J}^{\omega}g\right\rangle _{\omega}\right\vert \\
& \leq\sum_{S\in\mathfrak{C}_{\mathcal{A}}\left(  A\right)  }\sum_{J\in
\Lambda_{g}^{\omega}:\ J\subset_{\tau}\Lambda_{f}^{\sigma}\left[  S\right]
}\frac{\mathrm{P}\left(  J,\left\vert \varphi_{J}^{S}\right\vert
\sigma\right)  }{\ell\left(  J\right)  }\left\Vert \bigtriangleup_{J}^{\omega
}x\right\Vert _{L^{2}\left(  \omega\right)  }\left\Vert \bigtriangleup
_{J}^{\omega}g\right\Vert _{L^{2}\left(  \omega\right)  }\\
& \leq\sum_{S\in\mathfrak{C}_{\mathcal{A}}\left(  A\right)  }\alpha_{A}\left(
S\right)  \sum_{K\in\mathcal{W}_{\operatorname{good}}^{F,\operatorname*{trip}%
}\left(  S\right)  }\sum_{\substack{J\in\Lambda_{g}^{\omega}:\ J\subset_{\tau
}\Lambda_{f}^{\sigma}\left[  S\right]  \\J\subset K}}\frac{\mathrm{P}\left(
J,\mathbf{1}_{F\setminus S}\sigma\right)  }{\ell\left(  J\right)  }\left\Vert
\bigtriangleup_{J}^{\omega}x\right\Vert _{L^{2}\left(  \omega\right)
}\left\Vert \bigtriangleup_{J}^{\omega}g\right\Vert _{L^{2}\left(
\omega\right)  }\\
& +\sum_{S\in\mathfrak{C}_{\mathcal{A}}\left(  A\right)  }\alpha_{A}\left(
S\right)  \sum_{J\in\Lambda_{g}^{\omega}\cap\mathcal{N}_{\tau}\left(
S\right)  :\ J\subset_{\tau}\Lambda_{f}^{\sigma}\left[  S\right]  }%
\frac{\mathrm{P}\left(  J,\mathbf{1}_{F\setminus S}\sigma\right)  }%
{\ell\left(  J\right)  }\left\Vert \bigtriangleup_{J}^{\omega}x\right\Vert
_{L^{2}\left(  \omega\right)  }\left\Vert \bigtriangleup_{J}^{\omega
}g\right\Vert _{L^{2}\left(  \omega\right)  }\\
& \equiv\left\vert \mathsf{B}_{\mathcal{A}}^{A}\right\vert
_{\operatorname{straddle}}^{\operatorname*{trip}}\left(  f,g\right)
+\left\vert \mathsf{B}_{\mathcal{A}}^{A}\right\vert _{\operatorname{straddle}%
}^{\operatorname{near}}\left(  f,g\right)  .
\end{align*}
Since%
\[
\frac{\mathrm{P}\left(  J,\mathbf{1}_{F\setminus S}\sigma\right)  }%
{\ell\left(  J\right)  }\lesssim\frac{\mathrm{P}\left(  K,\mathbf{1}%
_{F\setminus S}\sigma\right)  }{\ell\left(  K\right)  },
\]
for $K\in\mathcal{W}_{\operatorname{good}}^{F,\operatorname*{trip}}\left(
S\right)  \cup\mathcal{N}_{\tau}\left(  S\right)  $, we have with
$\widehat{\mathsf{P}}_{S;K}^{\omega}\equiv\sum_{\substack{J\in\Lambda
_{g}^{\omega}:\ J\subset_{\tau}\Lambda_{f}^{\sigma}\left[  S\right]
\\J\subset W}}\bigtriangleup_{J}^{\omega}$,
\begin{align*}
& \left\vert \mathsf{B}_{\mathcal{A}}^{A}\right\vert _{\operatorname{straddle}%
}^{\operatorname*{trip}}\left(  f,g\right)  \leq\sum_{S\in\mathfrak{C}%
_{\mathcal{A}}\left(  A\right)  }\alpha_{A}\left(  S\right)  \sum
_{K\in\mathcal{W}_{\operatorname{good}}^{F}\left(  S\right)  }\sum
_{\substack{J\in\Lambda_{g}^{\omega}:\ J\subset_{\tau}\Lambda_{f}^{\sigma
}\left[  S\right]  \\J\subset K}}\frac{\mathrm{P}\left(  K,\mathbf{1}%
_{F\setminus S}\sigma\right)  }{\ell\left(  K\right)  }\left\Vert
\bigtriangleup_{J}^{\omega}x\right\Vert _{L^{2}\left(  \omega\right)
}\left\Vert \bigtriangleup_{J}^{\omega}g\right\Vert _{L^{2}\left(
\omega\right)  }\\
& \leq\sum_{S\in\mathfrak{C}_{\mathcal{A}}\left(  A\right)  }\alpha_{A}\left(
S\right)  \sum_{K\in\mathcal{W}_{\operatorname{good}}^{F}\left(  S\right)
}\frac{\mathrm{P}\left(  K,\mathbf{1}_{F\setminus S}\sigma\right)  }%
{\ell\left(  K\right)  }\left\Vert \widehat{\mathsf{P}}_{S;K}^{\omega
}x\right\Vert _{L^{2}\left(  \omega\right)  }\left\Vert \widehat{\mathsf{P}%
}_{S;K}^{\omega}g\right\Vert _{L^{2}\left(  \omega\right)  }\\
& \leq C\mathrm{P}\mathsf{E}_{\mathcal{A}}^{A,\operatorname*{trip}}\left(
\Lambda_{g}^{\omega};\sigma,\omega\right)  \sum_{S\in\mathfrak{C}%
_{\mathcal{A}}\left(  A\right)  }\alpha_{A}\left(  S\right)  \sum
_{K\in\mathcal{W}_{\operatorname{good}}^{F}\left(  S\right)  }\sqrt{\left\vert
K\right\vert _{\sigma}}\left\Vert \widehat{\mathsf{P}}_{S;K}^{\omega
}g\right\Vert _{L^{2}\left(  \omega\right)  }\\
& \leq C\mathrm{P}\mathsf{E}_{\mathcal{A}}^{A,\operatorname*{trip}}\left(
\Lambda_{g}^{\omega};\sigma,\omega\right)  \sum_{S\in\mathfrak{C}%
_{\mathcal{A}}\left(  A\right)  }\alpha_{A}\left(  S\right)  \sqrt{\sum
_{K\in\mathcal{W}_{\operatorname{good}}^{F}\left(  S\right)  }\left\vert
K\right\vert _{\sigma}}\sqrt{\sum_{K\in\mathcal{W}_{\operatorname{good}}%
^{F}\left(  S\right)  }\left\Vert \widehat{\mathsf{P}}_{S;K}^{\omega
}g\right\Vert _{L^{2}\left(  \omega\right)  }^{2}}\\
& \leq C\mathrm{P}\mathsf{E}_{\mathcal{A}}^{A,\operatorname*{trip}}\left(
\Lambda_{g}^{\omega};\sigma,\omega\right)  \sqrt{\sum_{S\in\mathfrak{C}%
_{\mathcal{A}}\left(  A\right)  }\left\vert S\right\vert _{\sigma}\alpha
_{A}\left(  S\right)  ^{2}}\left\Vert g\right\Vert _{L^{2}\left(
\omega\right)  }.
\end{align*}

The corresponding bound for the form $\left\vert \mathsf{B}_{\mathcal{A}}%
^{A}\right\vert _{\operatorname{straddle}}^{\operatorname{near}}\left(
f,g\right)  $ is similar, but easier since there are at most $2^{\tau+1}$
intervals $J$ in $\Lambda_{g}^{\omega}\cap\mathcal{N}_{\tau}\left(  S\right)
$, and this completes the proof of the Stopping Child Lemma.
\end{proof}

\subsubsection{Dual tree decomposition}

To control the local stopping forms $\mathsf{B}_{\operatorname*{stop}}%
^{F}\left(  f,g\right)  $, we need to introduce further corona decompositions
within each corona $\mathcal{C}_{\mathcal{F}}\left(  F\right)  $. These
coronas will be associated to stopping intervals $\mathcal{U}\left[  F\right]
\subset\mathcal{C}_{\mathcal{F}}\left(  F\right)  $, whose construction uses a
dual tree decomposition originating with M. Lacey in \cite{Lac}. For the sake
of generality, we present the decomposition in the setting of trees that are
not necessarily dyadic.

\begin{definition}
Let $\mathcal{T}$ be a tree with root $o$.

\begin{enumerate}
\item Let $P\left(  \alpha\right)  \equiv\left\{  \beta\in\mathcal{T}%
:\beta\succeq\alpha\right\}  $ and $S\left(  \alpha\right)  \equiv\left\{
\beta\in\mathcal{T}:\beta\preceq\alpha\right\}  $ denote the predessor and
successor sets of $\alpha\in\mathcal{T}$.

\item A \emph{geodesic} $\mathfrak{g}$ is a maximal linearly ordered subset of
$\mathcal{T}$. A finite geodesic $\mathfrak{g}$ is an interval $\mathfrak{g}%
=\left[  \alpha,\beta\right]  =P\left(  \beta\right)  \setminus S\left(
\alpha\right)  $, and an infinite geodesic is an interval $\mathfrak{g}%
=\mathfrak{g}\setminus P\left(  \alpha\right)  $ for some $\alpha
\in\mathfrak{g}$.

\item A \emph{stopping time}\footnote{This definition of stopping time used in
the theory of trees is a slight variant of what has been used above, but
should cause no confusion.} $T$ for a tree $\mathcal{T}$ is a subset
$T\subset\mathcal{T}$ such that
\[
S\left(  \beta\right)  \cap S\left(  \beta^{\prime}\right)  =\emptyset\text{
for all }\beta,\beta^{\prime}\in T\text{ with }\beta\neq\beta^{\prime}.
\]

\item A sequence $\left\{  T_{n}\right\}  _{n=0}^{N}$ of stopping times
$T_{n}$ is \emph{decreasing} if, for every $\beta\in T_{n+1}$ with $0\leq
n<N$, there is $\beta^{\prime}\in T_{n}$ such that $S\left(  \beta\right)
\subset S\left(  \beta^{\prime}\right)  $. We think of such a sequence as
getting further from the root as $n$ increases.

\item For $T$ a stopping time in $\mathcal{T}$ and $\alpha\in\mathcal{T}$, we
define
\[
\left[  T,\alpha\right)  \equiv\bigcup_{\beta\in T}\left[  \beta
,\alpha\right)  ,
\]
where the interval $\left[  \beta,\alpha\right)  =\emptyset$ unless
$\beta\prec\alpha$. In the case $\left[  T,\alpha\right)  =\emptyset$, we
write $\alpha\preceq T$, and in the case $\left[  T,\alpha\right)
\not =\emptyset$, we write $\alpha\succ T$. The set $\left[  T,\alpha\right)
$ can be thought of as the set of points in the tree $\mathcal{T}$ that `lie
between' $T$ and $\alpha$ but are strictly less than $\alpha$.

\item For any $\alpha\in\mathcal{T}$, we define the set of its \emph{children}
$\mathfrak{C}_{T}\left(  \alpha\right)  $ to consist of the \emph{maximal}
elements $\beta\in\mathcal{T}$ such that $\beta\prec\alpha$.
\end{enumerate}
\end{definition}

We define the dual integration operator $I^{\ast}$ on $\mathcal{T}$ by
$I^{\ast}\mu\left(  \alpha\right)  \equiv\sum_{\beta\in\mathcal{T}%
:\ \beta\preccurlyeq\alpha}\mu\left(  \beta\right)  $. Here is the dual
stopping time lemma that abstracts that of M. Lacey in \cite{Lac}.

\begin{lemma}
\label{lem I*}Let $\mathcal{T}$ be a tree with root $o$, and suppose there is
a uniform bound on the number of children of any tree element. Suppose further
that $\mu:\mathcal{T}\rightarrow\left[  0,\infty\right)  $ is nontrivial with
finite support, and let $T_{0}$ be the stopping time consisting of the minimal
tree elements in the support of $\mu$. Fix $\Gamma>1$. If there is no element
$\alpha\in\mathcal{T}$ with $I^{\ast}\mu\left(  \alpha\right)  >\Gamma
\sum_{\beta\in\mathcal{T}:\ \beta\prec\alpha}I^{\ast}\mu\left(  \beta\right)
$, we say that $\mu$ is $\Gamma$-irreducible. Otherwise, there is a unique
increasing sequence $\left\{  T_{n}\right\}  _{n=0}^{N+1}$, with
$T_{N+1}=\left\{  o\right\}  $, of stopping times $T_{n}$ such that for all
$n\in\mathbb{N}$ with $n\leq N$,
\begin{align}
I^{\ast}\mu\left(  \alpha\right)   & >\Gamma\sum_{\beta\in T_{n-1}%
:\ \beta\prec\alpha}I^{\ast}\mu\left(  \beta\right)  ,\ \ \ \ \ \text{for all
}\alpha\in T_{n}\ ;\label{dec corona}\\
I^{\ast}\mu\left(  \gamma\right)   & \leq\Gamma\sum_{\beta\in T_{n-1}%
:\ \beta\prec\gamma}I^{\ast}\mu\left(  \beta\right)  ,\text{ \ \ \ \ for all
}\gamma\in\left[  \alpha,T_{n-1}\right)  \text{ with }\alpha\in T_{n}%
\ ;\nonumber\\
I^{\ast}\mu\left(  o\right)   & \leq\Gamma\sum_{\beta\in T_{N}:\ \beta
\prec\gamma}I^{\ast}\mu\left(  \beta\right)  \ .\nonumber
\end{align}
Moreover, for $1\leq n\leq N+1$, this unique sequence $\left\{  T_{n}\right\}
_{n=0}^{N+1}$ satisfies%
\begin{align}
\left\vert \left(  T_{n-1},\alpha\right)  \right\vert _{\mu}  & \leq\left(
\Gamma-1\right)  \sum_{\beta\in T_{n-1}:\ \beta\prec\alpha}I^{\ast}\mu\left(
\beta\right)  ,\ \ \ \ \ \text{for all }\alpha\in T_{n}%
\ ,\label{dec corona small}\\
\left\vert \left(  T_{n-1},\gamma\right]  \right\vert _{\mu}  & \leq\left(
\Gamma-1\right)  \sum_{\beta\in T_{n-1}:\ \beta\prec\gamma}I^{\ast}\mu\left(
\beta\right)  ,\ \ \ \ \ \text{for all }\gamma\in S\left(  \alpha\right)
\setminus\left\{  \alpha\right\}  \setminus\bigcup_{\delta\in T_{n-1}}S\left(
\delta\right)  \ .\nonumber
\end{align}

\end{lemma}

\begin{proof}
If $T_{n}$ is already defined, let $T_{n+1}$ consist of all minimal points
$\alpha\in\mathcal{T}$ satisfying $I^{\ast}\mu\left(  \alpha\right)
>\Gamma\sum_{\beta\in T_{n}:\ \beta\prec\alpha}I^{\ast}\mu\left(
\beta\right)  $, provided at least one such point $\alpha$ exists. If not then
set $N=n$ and define $T_{N+1}\equiv\left\{  o\right\}  $. It is easy to see
that $\left\{  T_{n}\right\}  _{n=0}^{N+1}$ is an increasing sequence of
stopping times that satisfies (\ref{dec corona}), and is unique with these
properties. Moreover (\ref{dec corona small}) also holds since for $\alpha\in
T_{n}$ we have%
\begin{align*}
\left\vert \left(  T_{n-1},\alpha\right)  \right\vert _{\mu}  & =\sum
_{\gamma\in\mathfrak{C}_{\mathcal{T}}\left(  \alpha\right)  }\left\vert
\left(  T_{n-1},\gamma\right)  \right\vert _{\mu}=\sum_{\gamma\in
\mathfrak{C}_{\mathcal{T}}\left(  \alpha\right)  }\left(  I^{\ast}\mu\left(
\gamma\right)  -\sum_{\beta\in T_{n-1}:\ \beta\preccurlyeq\gamma}I^{\ast}%
\mu\left(  \beta\right)  \right) \\
& \leq\sum_{\gamma\in\mathfrak{C}_{\mathcal{T}}\left(  \alpha\right)  }\left(
\Gamma\sum_{\beta\in T_{n-1}:\ \beta\prec\gamma}I^{\ast}\mu\left(
\beta\right)  -\sum_{\beta\in T_{n-1}:\ \beta\preccurlyeq\gamma}I^{\ast}%
\mu\left(  \beta\right)  \right) \\
& \leq\left(  \Gamma-1\right)  \sum_{\gamma\in\mathfrak{C}_{\mathcal{T}%
}\left(  \alpha\right)  }\sum_{\beta\in T_{n-1}:\ \beta\prec\gamma}I^{\ast}%
\mu\left(  \beta\right)  =\left(  \Gamma-1\right)  \sum_{\beta\in
T_{n-1}:\ \beta\prec\alpha}I^{\ast}\mu\left(  \beta\right)  ,
\end{align*}
and the same argument proves the second line in (\ref{dec corona small}),
since $\gamma\in S\left(  \alpha\right)  \setminus\left\{  \alpha\right\}
\setminus\bigcup_{\delta\in T_{n-1}}S\left(  \delta\right)  $ was \emph{not}
chosen by the stopping criterion in the first line of (\ref{dec corona}), and
hence%
\[
I^{\ast}\mu\left(  \gamma\right)  \leq\Gamma\sum_{\beta\in T_{n-1}%
:\ \beta\prec\gamma}I^{\ast}\mu\left(  \beta\right)  .
\]

\end{proof}

\subsubsection{Completion of the proof}

Momentarily fix $\theta>0$ in Lemma \ref{lem I*}, which will be chosen at the
very end of the proof. For $F\in\mathcal{F}$, we consider the dyadic tree
$\mathcal{T}\equiv\mathcal{D}$, and let $\mathcal{U}\left[  F\right]
=\mathcal{U}_{\Lambda_{g}^{\omega}}\left[  F\right]  $ denote the collection
of intervals constructed in Lemma \ref{lem I*} with $o=F$, $\Gamma=1+\theta$
and $\mu_{g}:\mathcal{D}\rightarrow\left[  0,\infty\right)  $ defined
by$\ \mu_{g}\left(  J\right)  =\left\{
\begin{array}
[c]{ccc}%
\left\Vert \bigtriangleup_{J}^{\omega}x\right\Vert _{L^{2}\left(
\omega\right)  }^{2} & \text{ for } & J\in\Lambda_{g}^{\omega}\\
0 & \text{ for } & J\notin\Lambda_{g}^{\omega}%
\end{array}
\right.  $ for $J\in\mathcal{D}$. If $\mu_{g}$ is irreducible, the ensuing
arguments are greatly simplified, as the reader can easily check. Note that
$\mathcal{U}\left[  F\right]  $ depends only only the $\omega$-Haar support
$\Lambda_{g}^{\omega}$ of $g$, and not on $g$ itself. We will sometimes write
just $\mathcal{U}\left[  F\right]  $ when $\Lambda_{g}^{\omega}$ is
understood. Define $\operatorname{MIN}\left(  \Lambda_{g}^{\omega}\right)  $
to be the minimal intervals $J\in\mathcal{C}_{\mathcal{F}}%
^{\operatorname{good}}\left(  F\right)  $ for which $\bigtriangleup
_{J}^{\omega}x\neq0$. From (\ref{dec corona small}), we have tight control of
the $\operatorname{good}$ $\omega$-projections of $g$ in the coronas,%
\begin{equation}
\left\Vert \mathsf{P}_{\left[  \mathcal{D}^{\operatorname*{no}%
\operatorname{top}}\left(  K\right)  \setminus\bigcup_{S\in\mathfrak{C}%
_{\mathcal{U}\left[  F\right]  }\left[  U\right]  }\mathcal{D}\left(
S\right)  \right]  \cap\Lambda_{g}^{\omega}}^{\omega,\operatorname{good}%
}x\right\Vert _{L^{2}\left(  \omega\right)  }^{2}\leq\theta\left\Vert
\mathsf{P}_{\mathcal{D}\left(  K\right)  \cap\Lambda_{g}^{\omega}}%
^{\omega,\operatorname{good}}x\right\Vert _{L^{2}\left(  \omega\right)  }%
^{2}\ ,\label{tight}%
\end{equation}
for $K\in\mathcal{C}_{\mathcal{U}\left[  F\right]  }^{\operatorname{good}%
}\left(  U\right)  \setminus\operatorname{MIN}\left(  \Lambda_{g}^{\omega
}\right)  $, where $\mathcal{D}^{\operatorname*{no}\operatorname{top}}\left(
U\right)  \equiv\mathcal{D}\left(  U\right)  \setminus\left\{  U\right\}  $,
and the projections are restricted to $\operatorname{good}$ intervals. We also
have geometric decay in grandchildren from iterating the first line in
(\ref{dec corona}),%
\begin{equation}
\sum_{U^{\prime}\in\mathfrak{C}_{\mathcal{U}\left[  F\right]  }^{\left(
m\right)  }\left[  U\right]  }\left\Vert \mathsf{P}_{\mathcal{C}%
_{\mathcal{U}\left[  F\right]  }\left(  U^{\prime}\right)  \cap\Lambda
_{g}^{\omega}}^{\omega,\operatorname{good}}x\right\Vert _{L^{2}\left(
\omega\right)  }^{2}\leq\frac{1}{\Gamma^{m}}\left\Vert \mathsf{P}%
_{\mathcal{D}\left(  U\right)  \cap\Lambda_{g}^{\omega}}^{\omega
,\operatorname{good}}x\right\Vert _{L^{2}\left(  \omega\right)  }%
^{2}\ ,\ \ \ \ \ U\in\mathcal{U}\left[  F\right]  .\label{geo dec}%
\end{equation}

\begin{definition}
Fix $f\in L^{2}\left(  \sigma\right)  $ and let $\mathcal{F}$ be the stopping
times constructed using the criterion (\ref{energy stop crit}), that depends
only on $f$, $\sigma$ and $\omega$. For $g\in L^{2}\left(  \omega\right)  $
let $\mathcal{U}\left[  F\right]  =\mathcal{U}_{\Lambda_{g}^{\omega}}\left[
F\right]  ,F\in\mathcal{F}$ be as constructed above, depending only on $g$ and
$\omega$. Then for $F\in\mathcal{F}$, define the form,
\[
\mathsf{B}_{\operatorname*{stop}}^{F}\left(  f,g\right)  \equiv\sum
_{\substack{I\in\mathcal{C}_{F}\text{ and }J\in\mathcal{C}_{\mathcal{F}%
}\left(  F\right)  \\J\subset I\text{ and }\ell\left(  J\right)  \leq2^{-\tau
}\ell\left(  I\right)  }}\left(  E_{I_{J}}^{\sigma}\bigtriangleup_{I}^{\sigma
}f\right)  \left\langle H_{\sigma}\left(  \mathbf{1}_{F\setminus I_{J}%
}\right)  ,\bigtriangleup_{J}^{\omega}g\right\rangle _{\omega}\ .
\]

\end{definition}

We will prove the bound (\ref{stop est}) for $f$ and $g$. Denote the
collection of pairs $\left(  I,J\right)  $ arising in the sum defining
$\mathsf{B}_{\operatorname*{stop}}^{F}\left(  f,g\right)  $ by $\mathcal{P}%
^{F}$, so that%
\[
\mathsf{B}_{\operatorname*{stop}}^{F}\left(  f,g\right)  =\sum_{\left(
I,J\right)  \in\mathcal{P}^{F}}\left(  E_{I_{J}}^{\sigma}\bigtriangleup
_{I}^{\sigma}f\right)  \left\langle H_{\sigma}\left(  \mathbf{1}_{F\setminus
I_{J}}\right)  ,\bigtriangleup_{J}^{\omega}g\right\rangle _{\omega}\ .
\]
We use the corona decomposition associated with the stopping times
$\mathcal{U}$, namely%
\begin{align*}
\mathcal{P}^{F}  & =\bigcup_{U,V\in\mathcal{U}:\ V\subset U}\mathcal{C}%
_{\mathcal{U}}\left(  U\right)  \times\mathcal{C}_{\mathcal{U}}\left(
V\right) \\
& =\left\{  \bigcup_{U\in\mathcal{U}}\mathcal{C}_{\mathcal{U}}\left(
U\right)  \times\mathcal{C}_{\mathcal{U}}\left(  U\right)  \right\}
\bigcup\left\{  \bigcup_{U,VG\in\mathcal{U}:\ V\subsetneqq U}\mathcal{C}%
_{\mathcal{U}}\left(  U\right)  \times\mathcal{C}_{\mathcal{U}}\left(
V\right)  \right\} \\
& \equiv\mathcal{P}_{\operatorname{diag}}^{^{F}}\bigcup\mathcal{P}%
_{\operatorname{far}\operatorname{below}}^{^{F}}\ ,
\end{align*}
to obtain the decomposition of the stopping form $\mathsf{B}%
_{\operatorname*{stop}}^{F}\left(  f,g\right)  $ into `diagonal' and `far'
stopping forms,%
\begin{align*}
\mathsf{B}_{\operatorname*{stop}}^{F}\left(  f,g\right)   & =\mathsf{B}%
_{\operatorname{diag}\operatorname*{stop}}^{F}\left(  f,g\right)
+\mathsf{B}_{\operatorname{far}\operatorname*{stop}}^{F}\left(  f,g\right)
,\\
\mathsf{B}_{\operatorname{diag}\operatorname*{stop}}^{F}\left(  f,g\right)   &
\equiv\sum_{\left(  I,J\right)  \in\mathcal{P}_{\operatorname{diag}}^{^{F}}%
}\left(  E_{I_{J}}^{\sigma}\bigtriangleup_{I}^{\sigma}f\right)  \left\langle
H_{\sigma}\left(  \mathbf{1}_{F\setminus I_{J}}\right)  ,\bigtriangleup
_{J}^{\omega}g\right\rangle _{\omega}\ ,\\
\mathsf{B}_{\operatorname{far}\operatorname*{stop}}^{F}\left(  f,g\right)   &
\equiv\sum_{\left(  I,J\right)  \in\mathcal{P}_{\operatorname{far}%
\operatorname{below}}^{^{F}}}\left(  E_{I_{J}}^{\sigma}\bigtriangleup
_{I}^{\sigma}f\right)  \left\langle H_{\sigma}\left(  \mathbf{1}_{F\setminus
I_{J}}\right)  ,\bigtriangleup_{J}^{\omega}g\right\rangle _{\omega}\ ,
\end{align*}
where%
\begin{align*}
\mathsf{B}_{\operatorname{far}\operatorname*{stop}}^{F}\left(  f,g\right)   &
=\sum_{U,V\in\mathcal{U}:\ V\subsetneqq U}\sum_{I\in\mathcal{C}_{\mathcal{U}%
}\left(  U\right)  }\left(  E_{I_{J}}^{\sigma}\bigtriangleup_{I}^{\sigma
}f\right)  \left\langle H_{\sigma}\left(  \mathbf{1}_{F\setminus I_{J}%
}\right)  ,\sum_{\substack{J\in\mathcal{C}_{\mathcal{U}}\left(  V\right)
\\J\subset I\text{ and }\ell\left(  J\right)  \leq2^{-\tau}\ell\left(
I\right)  }}\bigtriangleup_{J}^{\omega}g\right\rangle _{\omega}\\
& =\sum_{t=1}^{\infty}\sum_{U\in\mathcal{U}}\sum_{V\in\mathfrak{C}%
_{\mathcal{U}}^{\left(  t\right)  }\left(  U\right)  }\sum_{I\in
\mathcal{C}_{\mathcal{U}}\left(  U\right)  }\left(  E_{I_{J}}^{\sigma
}\bigtriangleup_{I}^{\sigma}\mathsf{P}_{\mathcal{C}_{\mathcal{U}}\left(
U\right)  }^{\sigma}f\right)  \left\langle H_{\sigma}\left(  \mathbf{1}%
_{F\setminus I_{J}}\right)  ,\mathsf{P}_{\mathcal{C}_{\mathcal{U}}\left(
V\right)  }^{\omega}g\right\rangle _{\omega}.
\end{align*}

We will control this far stopping form $\mathsf{B}_{\operatorname{far}%
\operatorname*{stop}}^{F}\left(  f,g\right)  $ using the Stopping Child Lemma,
in which we will derive geometric decay from (\ref{geo dec}) and the
Poisson-Energy characteristic $\mathrm{P}\mathsf{E}_{\mathcal{A}%
}^{A,\operatorname*{trip}}\left(  \Lambda_{g_{t}}^{\omega};\sigma
,\omega\right)  $ with%
\begin{align*}
& A_{0}\text{ replaced by }F,\ \ \ \mathcal{A}\text{ replaced by }%
\mathcal{U}_{\Lambda_{g}^{\omega}}\left[  F\right]  \text{, \ \ }A\text{
replaced by }U,\\
& \text{and where }g_{t}\equiv\mathsf{P}_{\bigcup_{G\in\mathfrak{C}%
_{\mathcal{U}}^{\left(  t\right)  }\left(  U\right)  }\mathcal{C}%
_{\mathcal{U}}\left(  G\right)  }^{\omega}g.
\end{align*}
To reduce notational clutter, we write $\mathcal{U}_{g}$ instead of
$\mathcal{U}_{\Lambda_{g}^{\omega}}$, even though it depends only on the
$\omega$-Haar support of $g$. Indeed, we then have for $U\in\mathcal{U}_{g}$
and $t\geq1$,%
\begin{align*}
& \left\vert \sum_{V\in\mathfrak{C}_{\mathcal{U}_{g}}^{\left(  t\right)
}\left(  U\right)  }\sum_{I\in\mathcal{C}_{\mathcal{U}_{g}}\left(  U\right)
}\left(  E_{I_{J}}^{\sigma}\bigtriangleup_{I}^{\sigma}\mathsf{P}%
_{\mathcal{C}_{\mathcal{U}_{g}}\left(  U\right)  }^{\sigma}f\right)
\left\langle H_{\sigma}\left(  \mathbf{1}_{F\setminus I_{J}}\right)
,\mathsf{P}_{\mathcal{C}_{\mathcal{U}_{g}}\left(  V\right)  }^{\omega
}g\right\rangle _{\omega}\right\vert \\
& \leq C\mathrm{P}\mathsf{E}_{\mathcal{U}_{g}}^{U,\operatorname*{trip}}\left(
\Lambda_{g_{t}}^{\omega};\sigma,\omega\right)  \sqrt{\sum_{S\in\mathfrak{C}%
_{\mathcal{U}_{g}}\left(  U\right)  }\left\vert S\right\vert _{\sigma}%
\alpha_{F}\left(  S\right)  ^{2}}\left\Vert g_{t}\right\Vert _{L^{2}\left(
\omega\right)  },
\end{align*}
where%
\[
\alpha_{U}\left(  S\right)  \equiv\sup_{I\in\left(  \Lambda_{\mathsf{P}%
_{\mathcal{C}_{\mathcal{U}_{g}}\left(  U\right)  }^{\sigma}f}^{\sigma}\left[
S\right]  ,F\right]  \cap\mathcal{D}_{\operatorname{good}}}\left\vert
E_{I}^{\sigma}\mathsf{P}_{\mathcal{C}_{\mathcal{U}_{g}}\left(  U\right)
}^{\sigma}f\right\vert =\sup_{I\in\left(  S,U\right]  \cap\mathcal{D}%
_{\operatorname{good}}}\left\vert E_{I}^{\sigma}\mathsf{P}_{\mathcal{C}%
_{\mathcal{U}_{g}}\left(  U\right)  }^{\sigma}f\right\vert \leq E_{S}^{\sigma
}\left(  M_{\sigma}\mathsf{P}_{\mathcal{C}_{\mathcal{U}_{g}}\left(  U\right)
}^{\sigma}f\right)  ,
\]
and $g_{t}\equiv\sum_{V\in\mathfrak{C}_{\mathcal{U}_{g}}^{\left(  t\right)
}\left(  U\right)  }\mathsf{P}_{\mathcal{C}_{\mathcal{U}_{g}}\left(  V\right)
}^{\omega}g\equiv\mathsf{P}_{\mathcal{C}_{\mathcal{U}_{g}}^{\left(  t\right)
}\left(  U\right)  }^{\omega}g$, and where%
\begin{align*}
& \mathrm{P}\mathsf{E}_{\mathcal{U}_{g}}^{U,\operatorname*{trip}}\left(
\Lambda_{g_{t}}^{\omega};\sigma,\omega\right)  =\sup_{S\in\mathfrak{C}%
_{\mathcal{U}_{g}}\left(  U\right)  }\sup_{K\in\mathcal{W}%
_{\operatorname{good}}^{\operatorname*{trip}}\left(  S\right)  }\frac{1}%
{\sqrt{\left\vert K\right\vert _{\sigma}}}\frac{\mathrm{P}\left(
K,\mathbf{1}_{F\setminus S}\sigma\right)  }{\ell\left(  K\right)  }\sqrt
{\sum_{J\in\Lambda_{g}^{\omega}\cap\bigcup_{V\in\mathfrak{C}_{\mathcal{U}_{g}%
}^{\left(  t\right)  }\left(  U\right)  :G\subset K}\mathcal{C}_{\mathcal{U}%
_{g}}\left(  V\right)  }\left\Vert \bigtriangleup_{J}^{\omega}x\right\Vert
_{L^{2}\left(  \omega\right)  }^{2}}\\
& \leq\sup_{S\in\mathfrak{C}_{\mathcal{U}_{g}}\left(  U\right)  }\sup
_{K\in\mathcal{W}_{\operatorname{good}}^{\operatorname*{trip}}\left(
S\right)  }\frac{1}{\sqrt{\left\vert K\right\vert _{\sigma}}}\frac
{\mathrm{P}\left(  K,\mathbf{1}_{F\setminus S}\sigma\right)  }{\ell\left(
K\right)  }\sqrt{\left(  1-\varepsilon\right)  ^{t}\sum_{J\in\Lambda
_{g}^{\omega}\cap\mathcal{D}\left[  K\right]  }\left\Vert \bigtriangleup
_{J}^{\omega}x\right\Vert _{L^{2}\left(  \omega\right)  }^{2}}\leq\left(
1-\varepsilon\right)  ^{\frac{t}{2}}\mathrm{P}\mathsf{E}_{\mathcal{U}%
_{g}\left[  F\right]  }^{U,\operatorname*{trip}}\left(  \Lambda_{g}^{\omega
};\sigma,\omega\right)  ,
\end{align*}
i.e. we gain geometrically in $t$ when passing from the Haar support of
$g_{t}$ to that of $g$.

We now introduce the size functional of Lacey \cite{Lac}, denoted here by%
\[
\mathrm{P}\mathsf{E}_{\mathcal{F}}^{F}\left(  \Lambda;\sigma,\omega\right)
\equiv\sup_{K\in\mathcal{C}_{\mathcal{F}}^{\operatorname{good}}\left(
F\right)  }\frac{1}{\sqrt{\left\vert K\right\vert _{\sigma}}}\frac
{\mathrm{P}\left(  K,\mathbf{1}_{F\setminus K}\sigma\right)  }{\ell\left(
K\right)  }\sqrt{\sum_{J\in\Lambda\cap\mathcal{D}\left[  K\right]  }\left\Vert
\bigtriangleup_{J}^{\omega}x\right\Vert _{L^{2}\left(  \omega\right)  }^{2}}.
\]
Recalling the\ previous characteristics from (\ref{PE char}) and
(\ref{refined PE char}), we have%
\begin{equation}
\mathrm{P}\mathsf{E}_{\mathcal{A}\left[  F\right]  }^{\operatorname*{trip}%
}\left(  \Lambda;\sigma,\omega\right)  \leq\mathrm{P}\mathsf{E}_{\mathcal{F}%
}^{F}\left(  \Lambda;\sigma,\omega\right)  \leq\mathrm{P}\mathsf{E}%
_{\mathcal{F}}^{F}\left(  \sigma,\omega\right)  \leq\mathrm{P}\mathsf{E}%
_{\mathcal{F}}\left(  \sigma,\omega\right)  .\label{note}%
\end{equation}
Thus we have from the Stopping Child Lemma applied to each $U\in
\mathcal{U}_{g}$,%
\begin{align}
& \ \ \ \ \ \ \ \ \ \ \ \ \ \ \ \left\vert \mathsf{B}_{\operatorname{far}%
\operatorname*{stop}}^{F}\left(  f,g\right)  \right\vert \leq\sum
_{t=1}^{\infty}\sum_{U\in\mathcal{U}_{g}}C\mathrm{P}\mathsf{E}_{\mathcal{U}%
_{g}\left[  F\right]  }^{U,\operatorname*{trip}}\left(  \Lambda_{g_{t}%
}^{\omega};\sigma,\omega\right)  \sqrt{\sum_{S\in\mathfrak{C}_{\mathcal{U}%
_{g}}\left(  U\right)  }\left\vert S\right\vert _{\sigma}E_{S}^{\sigma}\left(
M_{\sigma}\mathsf{P}_{\mathcal{C}_{\mathcal{U}_{g}}\left(  U\right)  }%
^{\sigma}f\right)  ^{2}}\left\Vert g_{t}\right\Vert _{L^{2}\left(
\omega\right)  }\label{far below stop}\\
& \leq\sum_{t=1}^{\infty}\sum_{U\in\mathcal{U}_{g}}C\left(  1-\varepsilon
\right)  ^{\frac{t}{2}}\mathrm{P}\mathsf{E}_{\mathcal{U}_{g}\left[  F\right]
}^{U,\operatorname*{trip}}\left(  \Lambda_{g}^{\omega};\sigma,\omega\right)
\sqrt{\sum_{S\in\mathfrak{C}_{\mathcal{U}_{g}}\left(  U\right)  }\left\vert
S\right\vert _{\sigma}\left(  M_{\sigma}\mathsf{P}_{\mathcal{C}_{\mathcal{U}%
_{g}}\left(  U\right)  }^{\sigma}f\right)  ^{2}}\left\Vert \sum_{V\in
\mathfrak{C}_{\mathcal{U}_{g}}^{\left(  t\right)  }\left(  F\right)
}\mathsf{P}_{\mathcal{C}_{\mathcal{U}_{g}}\left(  V\right)  }^{\omega
}g\right\Vert _{L^{2}\left(  \omega\right)  }\nonumber\\
& \leq C\mathrm{P}\mathsf{E}_{\mathcal{U}_{g}\left[  F\right]  }%
^{\operatorname*{trip}}\left(  \Lambda_{g}^{\omega};\sigma,\omega\right)
\sqrt{\sum_{U\in\mathcal{U}_{g}}\sum_{S\in\mathfrak{C}_{\mathcal{U}_{g}%
}\left(  U\right)  }\left\vert S\right\vert _{\sigma}\left(  M_{\sigma
}\mathsf{P}_{\mathcal{C}_{\mathcal{U}_{g}}\left(  U\right)  }^{\sigma
}f\right)  ^{2}}\left\Vert \sum_{U\in\mathcal{U}_{g}}\sum_{V\in\mathfrak{C}%
_{\mathcal{U}_{g}}^{\left(  t\right)  }\left(  U\right)  }\mathsf{P}%
_{\mathcal{C}_{\mathcal{U}_{g}}\left(  V\right)  }^{\omega}g\right\Vert
_{L^{2}\left(  \omega\right)  }\nonumber\\
& \leq C\mathrm{P}\mathsf{E}_{\mathcal{U}_{g}\left[  F\right]  }%
^{\operatorname*{trip}}\left(  \Lambda_{g}^{\omega};\sigma,\omega\right)
\left\Vert f\right\Vert _{L^{2}\left(  \sigma\right)  }\left\Vert g\right\Vert
_{L^{2}\left(  \omega\right)  },\nonumber
\end{align}
by Cauchy-Schwarz, boundedness of the dyadic maximal function $M_{\sigma}$ on
$L^{2}\left(  \sigma\right)  $, and orthogonality in both $f$ and $g$.

Finally it remains to control the diagonal stopping form $\mathsf{B}%
_{\operatorname{diag}\operatorname*{stop}}^{F}\left(  f,g\right)  $, where%
\[
\mathsf{B}_{\operatorname{diag}\operatorname*{stop}}^{F}\left(  f,g\right)
=\sum_{U\in\mathcal{U}_{g}}\sum_{I\in\mathcal{C}_{\mathcal{U}_{g}}\left(
U\right)  }\left(  E_{I_{J}}^{\sigma}\bigtriangleup_{I}^{\sigma}%
\mathsf{P}_{\mathcal{C}_{\mathcal{U}_{g}}\left(  U\right)  }^{\sigma}f\right)
\left\langle H_{\sigma}\left(  \mathbf{1}_{F\setminus I_{J}}\right)
,\sum_{J\in\mathcal{C}_{\mathcal{U}_{g}}\left(  U\right)  :\ J\subset I\text{
and }\ell\left(  J\right)  \leq2^{-\tau}\ell\left(  I\right)  }\bigtriangleup
_{J}^{\omega}g\right\rangle _{\omega}.
\]
Note that the intervals $J$ arising in $\mathsf{B}_{\operatorname{diag}%
\operatorname*{stop}}^{F}\left(  f,g\right)  $ actually lie in $\mathcal{C}%
_{\mathcal{U}_{g}}^{\operatorname*{no}\operatorname{top}}\equiv\mathcal{C}%
_{\mathcal{U}_{g}}\setminus\left\{  U\right\}  $ by the requirements that
$\ell\left(  J\right)  \leq2^{-\tau}\ell\left(  I\right)  $ and $I\in
\mathcal{C}_{\mathcal{U}_{g}}\left(  U\right)  $. In particular, the minimal
stopping intervals $U\in\mathcal{U}_{g}$ do not contribute to the sum
$\sum_{U\in\mathcal{U}_{g}}$ defining $\mathsf{B}_{\operatorname{diag}%
\operatorname*{stop}}^{F}\left(  f,g\right)  $.

Next we note that%
\begin{align*}
\mathsf{B}_{\operatorname{diag}\operatorname*{stop}}^{F}\left(  f,g\right)   &
\equiv\sum_{U\in\mathcal{U}_{g}}\sum_{I\in\mathcal{C}_{\mathcal{U}_{g}}\left(
U\right)  }\left(  E_{I_{J}}^{\sigma}\bigtriangleup_{I}^{\sigma}%
\mathsf{P}_{\mathcal{C}_{\mathcal{U}_{g}}\left(  U\right)  }^{\sigma}f\right)
\left\langle H_{\sigma}\left(  \mathbf{1}_{F\setminus I_{J}}\right)
,\sum_{J\in\mathcal{C}_{\mathcal{U}_{g}}^{\operatorname*{no}\operatorname{top}%
}\left(  U\right)  :\ J\subset I\text{ and }\ell\left(  J\right)  \leq
2^{-\tau}\ell\left(  I\right)  }\bigtriangleup_{J}^{\omega}g\right\rangle
_{\omega}\\
& =\sum_{U\in\mathcal{U}_{g}\left[  F\right]  }\mathsf{B}%
_{\operatorname*{stop}}^{F}\left(  \mathsf{P}_{\mathcal{C}_{\mathcal{U}_{g}%
}\left(  U\right)  }^{\sigma}f,\mathsf{P}_{\mathcal{C}_{\mathcal{U}_{g}%
}^{\operatorname*{no}\operatorname{top}}\left(  U\right)  }^{\omega}g\right)
=\sum_{U\in\mathcal{U}_{g}\left[  F\right]  }\mathsf{B}_{\operatorname*{stop}%
}^{F}\left(  f_{U},g_{U}^{\operatorname*{no}\operatorname{top}}\right)  ,\\
\text{where }f_{U}  & \equiv\mathsf{P}_{\mathcal{C}_{\mathcal{U}_{g}}\left(
U\right)  }^{\sigma}f\text{ and }g_{U}^{\operatorname*{no}\operatorname{top}%
}\equiv\mathsf{P}_{\mathcal{C}_{\mathcal{U}_{g}}^{\operatorname*{no}%
\operatorname{top}}\left(  U\right)  }^{\omega}g\ .
\end{align*}
For $F\in\mathcal{F}$, set%
\[
\Theta_{\digamma}^{F}\equiv\sup_{\substack{f,g\neq0 \\\Lambda_{f}^{\sigma
},\Lambda_{g}^{\omega}\in\digamma}}\frac{\left\vert \mathsf{B}%
_{\operatorname*{stop}}^{F}\left(  f,g\right)  \right\vert }{\mathrm{P}%
\mathsf{E}_{\mathcal{F}}^{F}\left(  \Lambda_{g}^{\omega};\sigma,\omega\right)
\left\Vert f\right\Vert _{L^{2}(\sigma)}\left\Vert g\right\Vert _{L^{2}%
(\omega)}},
\]
which is finite because of our assumption that the Haar supports of $f$ and
$g$ are restricted to a fixed finite subset $\digamma$ of the grid
$\mathcal{D}$. We next use (\ref{tight}), (\ref{note}) and orthogonality of
the projections $\mathsf{P}_{\mathcal{C}_{\mathcal{U}}\left(  U\right)  }%
^{\mu}$ to deduce that%
\begin{align*}
\Theta_{\digamma}^{F}  & \leq\sup_{\substack{f,g\neq0 \\\Lambda_{f}^{\sigma
},\Lambda_{g}^{\omega}\in\digamma}}\frac{\left\vert \mathsf{B}%
_{\operatorname{far}\operatorname*{stop}}^{F}\left(  f,g\right)  \right\vert
+\left\vert \mathsf{B}_{\operatorname{diag}\operatorname*{stop}}%
^{F,\operatorname{top}\operatorname{only}}\left(  f,g\right)  \right\vert
+\left\vert \mathsf{B}_{\operatorname{diag}\operatorname*{stop}}%
^{F,\operatorname*{no}\operatorname{top}}\left(  f,g\right)  \right\vert
}{\mathrm{P}\mathsf{E}_{\mathcal{F}}^{F}\left(  \Lambda_{g}^{\omega}%
;\sigma,\omega\right)  \left\Vert f\right\Vert _{L^{2}(\sigma)}\left\Vert
g\right\Vert _{L^{2}(\omega)}}\\
& \leq C+C+\sup_{\substack{f,g\neq0 \\\Lambda_{f}^{\sigma},\Lambda_{g}%
^{\omega}\in\digamma}}\frac{\sum_{U\in\mathcal{U}_{\Lambda_{g}^{\omega}%
}\left[  F\right]  }\left\vert \mathsf{B}_{\operatorname*{stop}}^{F}\left(
f_{U},g_{U}^{\operatorname*{no}\operatorname{top}}\right)  \right\vert
}{\mathrm{P}\mathsf{E}_{\mathcal{F}}^{F}\left(  \Lambda_{g}^{\omega}%
;\sigma,\omega\right)  \left\Vert f\right\Vert _{L^{2}(\sigma)}\left\Vert
g\right\Vert _{L^{2}(\omega)}}\\
& \leq C+C+\sup_{\substack{f,g\neq0 \\\Lambda_{f}^{\sigma},\Lambda_{g}%
^{\omega}\in\digamma}}\sum_{U\in\mathcal{U}_{\Lambda_{g}^{\omega}}\left[
F\right]  }\frac{\mathrm{P}\mathsf{E}_{\mathcal{F}}^{F}\left(  \Lambda
_{g_{U}^{\operatorname*{no}\operatorname{top}}}^{\omega};\sigma,\omega\right)
}{\mathrm{P}\mathsf{E}_{\mathcal{F}}^{F}\left(  \Lambda_{g}^{\omega}%
;\sigma,\omega\right)  }\frac{\left\Vert f_{U}\right\Vert _{L^{2}(\sigma
)}\left\Vert g_{U}^{\operatorname*{no}\operatorname{top}}\right\Vert
_{L^{2}(\omega)}}{\left\Vert f\right\Vert _{L^{2}(\sigma)}\left\Vert
g\right\Vert _{L^{2}(\omega)}}\\
& \ \ \ \ \ \ \ \ \ \ \ \ \ \ \ \ \ \ \ \ \times\frac{\left\vert
\mathsf{B}_{\operatorname*{stop}}^{F}\left(  f_{U},g_{U}^{\operatorname*{no}%
\operatorname{top}}\right)  \right\vert }{\mathrm{P}\mathsf{E}_{\mathcal{F}%
}^{F}\left(  \Lambda_{g_{U}^{\operatorname*{no}\operatorname{top}}}^{\omega
};\sigma,\omega\right)  \left\Vert f_{U}\right\Vert _{L^{2}(\sigma)}\left\Vert
g_{U}^{\operatorname*{no}\operatorname{top}}\right\Vert _{L^{2}(\omega)}}\\
& \leq2C+\sup_{\substack{f,g\neq0 \\\Lambda_{f}^{\sigma},\Lambda_{g}^{\omega
}\in\digamma}}\sup_{U\in\mathcal{U}_{\Lambda_{g}^{\omega}}\left[  F\right]
}\frac{\mathrm{P}\mathsf{E}_{\mathcal{F}}^{F}\left(  \Lambda_{g_{U}%
^{\operatorname*{no}\operatorname{top}}}^{\omega};\sigma,\omega\right)
}{\mathrm{P}\mathsf{E}_{\mathcal{F}}^{F}\left(  \Lambda_{g}^{\omega}%
;\sigma,\omega\right)  }\Theta_{\digamma}^{F}\leq2C+\theta\Theta_{\digamma
}^{F},
\end{align*}
where $\theta>0$ is as in (\ref{tight}). Note that (\ref{tight}) applies here
precisely because the tops of the coronas are missing in $\mathcal{C}%
_{\mathcal{U}_{\Lambda_{g}^{\omega}}}^{\operatorname*{no}\operatorname{top}%
}\left(  U\right)  $. Indeed, for $K\in\mathcal{C}_{\mathcal{F}}%
^{\operatorname{good}}\left(  F\right)  $, it suffices to show,%
\begin{equation}
\sqrt{\sum_{J\in\Lambda_{g_{U}^{\operatorname*{no}\operatorname{top}}}%
^{\omega}\cap\mathcal{D}\left[  K\right]  }\left\Vert \bigtriangleup
_{J}^{\omega}x\right\Vert _{L^{2}\left(  \omega\right)  }^{2}}\leq\theta
\sqrt{\sum_{J\in\Lambda_{g}^{\omega}\cap\mathcal{D}\left[  K\right]
}\left\Vert \bigtriangleup_{J}^{\omega}x\right\Vert _{L^{2}\left(
\omega\right)  }^{2}}.\label{claim theta}%
\end{equation}
If $\Lambda_{g_{U}^{\operatorname*{no}\operatorname{top}}}^{\omega}%
\cap\mathcal{D}\left[  K\right]  $ is empty, there is nothing to prove, and if
$K\supset U$, the first line in (\ref{tight}) yields (\ref{claim theta}). In
the remaining case $K\in\mathcal{C}_{\mathcal{U}_{\Lambda_{g}^{\omega}}%
}^{\operatorname*{no}\operatorname{top}}\left(  U\right)  \setminus
\operatorname{MIN}\left(  \Lambda_{g}^{\omega}\right)  $, the second line in
(\ref{tight}) yields (\ref{claim theta}).

Since $\Theta_{\digamma}^{F}<\infty$, we conclude that $\Theta_{\digamma}%
^{F}\leq\frac{2C}{1-\theta}\leq4C$ provided $0<\theta<\frac{1}{2}$. Hence we
have the inequality,%
\[
\left\vert \mathsf{B}_{\operatorname*{stop}}^{F}\left(  f,g\right)
\right\vert \leq\Theta_{\digamma}^{F}\mathrm{P}\mathsf{E}_{\mathcal{F}}%
^{F}\left(  \Lambda_{g}^{\omega};\sigma,\omega\right)  \left\Vert f\right\Vert
_{L^{2}(\sigma)}\left\Vert g\right\Vert _{L^{2}(\omega)}\leq4C\mathrm{P}%
\mathsf{E}_{\mathcal{F}}\left(  \sigma,\omega\right)  \left\Vert f\right\Vert
_{L^{2}(\sigma)}\left\Vert g\right\Vert _{L^{2}(\omega)},
\]
using (\ref{note}), which when combined with (\ref{far below stop}), completes
the required control (\ref{stop est}) of the stopping form $\mathsf{B}%
_{\operatorname*{stop}}^{F}\left(  f,g\right)  $ when the Haar supports of $f$
and $g$ are in a fixed finite subset $\digamma$ of the grid $\mathcal{D} $.
Since such functions are dense in $L^{2}\left(  \sigma\right)  $ and
$L^{2}\left(  \omega\right)  $ as $\digamma$ ranges over all finite subsets
$\digamma$ of $\mathcal{D}$, the proof of (\ref{stop est}) is complete.

\begin{proof}
[Proof of Theorem \ref{main}]Collecting all of the above form estimates proves
Theorem \ref{main}.
\end{proof}

\section{Appendix}

Here we reproduce the case $p=2$ of an argument in the first version of
\cite{SaWi} on the arXiv, in order to prove that the full testing condition
for $T$ can be controlled by the testing and Muckenhoupt characteristics of
the Hilbert transform $H$, i.e.
\[
\mathfrak{FT}_{T}\left(  \sigma,\omega\right)  \leq C\left[  \mathfrak{T}%
_{H}^{\operatorname{loc}}\left(  \sigma,\omega\right)  +\mathcal{E}_{2}\left(
\sigma,\omega\right)  +\mathcal{A}_{2}^{\operatorname*{hole}}\left(
\sigma,\omega\right)  \right]  \leq C\left[  \mathfrak{T}_{H}%
^{\operatorname{loc}}\left(  \sigma,\omega\right)  +\mathcal{A}_{2}%
^{\operatorname*{hole}}\left(  \sigma,\omega\right)  \right]  .
\]
We use the notation from \cite{SaWi}, except that $\mathcal{A}_{2}%
^{\operatorname*{hole}}\left(  \sigma,\omega\right)  $ denotes the one-tailed
Muckenhoupt characteristic with holes here.

Splitting the integral $\int_{\mathbb{R}}T_{\sigma}\mathbf{1}_{I}\left(
x\right)  ^{2}d\omega\left(  x\right)  $ into a local and global piece, we
obtain%
\[
\left\Vert T\mathbf{1}_{I}\right\Vert _{L^{2}\left(  \omega\right)  }^{2}%
=\int_{I}T_{\sigma}\mathbf{1}_{I}\left(  x\right)  ^{2}d\omega\left(
x\right)  +\int_{\mathbb{R}\setminus I}T_{\sigma}\mathbf{1}_{I}\left(
x\right)  ^{2}d\omega\left(  x\right)  \equiv\mathbf{Local}\left(  I\right)
+\mathbf{Global}\left(  I\right)  .
\]
Here is a brief schematic diagram of the decomposition, with bounds in
$\fbox{}$, used in this subsection:%
\[
\fbox{$%
\begin{array}
[c]{ccccc}%
\mathbf{Local} &  &  &  & \\
\downarrow &  &  &  & \\
\mathbf{Local}^{\operatorname*{plug}} & + & \mathbf{Local}%
^{\operatorname*{hole}} &  & \\
\downarrow &  & \downarrow &  & \\
\downarrow &  & E & + & F\\
\downarrow &  & \fbox{$\mathfrak{T}_{H}^{\operatorname{loc}}\left(
\sigma,\omega\right)  $} &  & \fbox{$\mathfrak{T}_{H}^{\operatorname{loc}%
}\left(  \sigma,\omega\right)  $}\\
\downarrow &  &  &  & \\
A & + & C & + & D\\
\fbox{$\mathcal{E}_{2}\left(  \sigma,\omega\right)  $} &  & \fbox{$\mathcal{E}%
_{2}\left(  \sigma,\omega\right)  $} &  & \fbox{$\mathcal{A}_{2}%
^{\operatorname*{hole}}\left(  \sigma,\omega\right)  $}%
\end{array}
$}\text{,}%
\]
and%
\[
\text{ }\fbox{$%
\begin{array}
[c]{ccccccc}%
\mathbf{Global} &  &  &  &  &  & \\
\downarrow &  &  &  &  &  & \\
A & + & B & + & C & + & D\\
\fbox{$\mathcal{A}_{2}^{\operatorname*{hole}}\left(  \sigma,\omega\right)  $}
&  & \fbox{$\mathcal{A}_{2}^{\operatorname*{hole}}\left(  \sigma
,\omega\right)  $} &  & \fbox{$\mathcal{A}_{2}^{\operatorname*{hole}}\left(
\sigma,\omega\right)  $} &  & \fbox{$\mathcal{A}_{2}^{\operatorname*{hole}%
}\left(  \sigma,\omega\right)  $}%
\end{array}
$},
\]
where all of these bounds are controlled by the global testing condition
$\mathfrak{T}_{2}^{\operatorname*{glob}}\left(  \sigma,\omega\right)  $ as well.

Here, and in the next section as well, we will make critical use of the
following consequence of the fact that for any $W\in\mathcal{D}$, there is at
most one $F\in\mathcal{F}$ with $W\in\mathcal{M}_{\left(  r,\varepsilon
\right)  -\operatorname*{deep}}\left(  F\right)  \cap\mathcal{C}_{\mathcal{F}%
}\left(  F\right)  $,%
\begin{equation}
\int\sum_{F\in\mathcal{F}:\ W\in\mathcal{M}_{\left(  r,\varepsilon\right)
-\operatorname*{deep}}\left(  F\right)  \cap\mathcal{C}_{\mathcal{F}}\left(
F\right)  }\left\vert \mathsf{P}_{\mathcal{C}_{\mathcal{F}}\left(  F\right)
\cap\mathcal{D}\left[  W\right]  }^{\omega}\right\vert Z\left(  x\right)
^{2}d\omega\left(  x\right)  \leq\ell\left(  W\right)  ^{2}\left\vert
I\right\vert _{\omega},\ \ \ \ \ \text{for each }W\in\mathcal{D}%
,\label{overlap}%
\end{equation}
since
\[
\sum_{I\subset K}\int\left\vert \bigtriangleup_{I}^{\omega}Z\right\vert
^{2}d\omega\left(  x\right)  =\left\Vert Z-c_{K}\right\Vert _{L^{2}\left(
\omega\right)  }^{2}\leq\ell\left(  W\right)  ^{2}\left\vert I\right\vert
_{\omega}.
\]

We turn first to estimating the local term $\mathbf{Local}$.

\subsection{Local forward testing}

Orthogonality of the haar projections $\left\{  \mathsf{P}_{\mathcal{C}%
_{\mathcal{F}}\left(  F\right)  \cap\mathcal{D}\left[  W\right]  }^{\omega
}\right\}  _{W}$ shows that%
\begin{align*}
\mathbf{Local}\left(  I\right)    & =\int_{I}\left\vert \int_{I}K\left(
x,y\right)  d\sigma\left(  y\right)  \right\vert ^{2}d\omega\left(  x\right)
\\
& =\int_{I}\left(  \sum_{F\in\mathcal{F}}\sum_{W\in\mathcal{M}_{\left(
r,\varepsilon\right)  -\operatorname*{deep}}\left(  F\right)  \cap
\mathcal{C}_{\mathcal{F}}\left(  F\right)  }\frac{\mathrm{P}\left(
W,\mathbf{1}_{I}\sigma\right)  }{\ell\left(  W\right)  }\mathsf{P}%
_{\mathcal{C}_{\mathcal{F}}\left(  F\right)  \cap\mathcal{D}\left[  W\right]
}^{\omega}Z\left(  x\right)  \right)  ^{2}d\omega\left(  x\right)  \\
& \approx\int_{I}\sum_{F\in\mathcal{F}}\sum_{W\in\mathcal{M}_{\left(
r,\varepsilon\right)  -\operatorname*{deep}}\left(  F\right)  \cap
\mathcal{C}_{\mathcal{F}}\left(  F\right)  }\left(  \frac{\mathrm{P}\left(
W,\mathbf{1}_{I}\sigma\right)  }{\ell\left(  W\right)  }\right)
^{2}\left\vert \mathsf{P}_{\mathcal{C}_{\mathcal{F}}\left(  F\right)
\cap\mathcal{D}\left[  W\right]  }^{\omega}\right\vert Z\left(  x\right)
^{2}d\omega\left(  x\right)  \\
& \lesssim\mathbf{Local}^{\operatorname{plug}}\left(  I\right)
+\mathbf{Local}^{\operatorname{hole}}\left(  I\right)  ,
\end{align*}
where%
\begin{align*}
\mathbf{Local}^{\operatorname{plug}}\left(  I\right)    & \equiv\int_{I}%
\sum_{F\in\mathcal{F}}\sum_{W\in\mathcal{M}_{\left(  r,\varepsilon\right)
-\operatorname*{deep}}\left(  F\right)  \cap\mathcal{C}_{\mathcal{F}}\left(
F\right)  }\left(  \frac{\mathrm{P}\left(  W,\mathbf{1}_{I\cap F}%
\sigma\right)  }{\ell\left(  W\right)  }\right)  ^{2}\left\vert \mathsf{P}%
_{\mathcal{C}_{\mathcal{F}}\left(  F\right)  \cap\mathcal{D}\left[  W\right]
}^{\omega}\right\vert Z\left(  x\right)  ^{2}d\omega\left(  x\right)  ,\\
\mathbf{Local}^{\operatorname{hole}}\left(  I\right)    & \equiv\int_{I}%
\sum_{F\in\mathcal{F}}\sum_{W\in\mathcal{M}_{\left(  r,\varepsilon\right)
-\operatorname*{deep}}\left(  F\right)  \cap\mathcal{C}_{\mathcal{F}}\left(
F\right)  }\left(  \frac{\mathrm{P}\left(  W,\mathbf{1}_{I\setminus F}%
\sigma\right)  }{\ell\left(  W\right)  }\right)  ^{2}\left\vert \mathsf{P}%
_{\mathcal{C}_{\mathcal{F}}\left(  F\right)  \cap\mathcal{D}\left[  W\right]
}^{\omega}\right\vert Z\left(  x\right)  ^{2}d\omega\left(  x\right)  .
\end{align*}
Then we write,%
\begin{align*}
\mathbf{Local}^{\operatorname{plug}}\left(  I\right)    & \approx\int
_{I}\left(  \sum_{F\in\mathcal{F}:\ F\subset I}\sum_{W\in\mathcal{M}_{\left(
r,\varepsilon\right)  -\operatorname*{deep}}\left(  F\right)  \cap
\mathcal{C}_{\mathcal{F}}\left(  F\right)  }\left(  \frac{\mathrm{P}\left(
W,\mathbf{1}_{F\cap I}\sigma\right)  }{\ell\left(  W\right)  }\right)
^{2}\left\vert \mathsf{P}_{\mathcal{C}_{\mathcal{F}}\left(  F\right)
\cap\mathcal{D}\left[  W\right]  }^{\omega}\right\vert Z\left(  x\right)
^{2}\right)  d\omega\left(  x\right)  \\
& +\int_{I}\left(  \sum_{F\in\mathcal{F}:\ F\supsetneqq I}\sum_{W\in
\mathcal{M}_{\left(  r,\varepsilon\right)  -\operatorname*{deep}}\left(
F\right)  \cap\mathcal{C}_{\mathcal{F}}\left(  F\right)  }\left(
\frac{\mathrm{P}\left(  W,\mathbf{1}_{F\cap I}\sigma\right)  }{\ell\left(
W\right)  }\right)  ^{2}\left\vert \mathsf{P}_{\mathcal{C}_{\mathcal{F}%
}\left(  F\right)  \cap\mathcal{D}\left[  W\right]  }^{\omega}\right\vert
Z\left(  x\right)  ^{2}\right)  d\omega\left(  x\right)  \\
& \equiv A+B,
\end{align*}
where term $A$ is easily handled by the energy characteristic,%
\[
A\leq\mathcal{E}_{2}\left(  \sigma,\omega\right)  ^{2}\left\vert I\right\vert
_{\sigma}\ ,
\]
since%
\[
\mathcal{E}_{2}\left(  \sigma,\omega\right)  ^{2}=\sup\frac{1}{\left\vert
I_{0}\right\vert _{\sigma}}\int_{I_{0}}\sum_{F\in\mathcal{F}:\ F\subset I_{0}%
}\sum_{i=1}^{\infty}\left(  \frac{\mathrm{P}\left(  W_{i}^{F},\mathbf{1}%
_{F}\sigma\right)  }{\ell\left(  W_{i}^{F}\right)  }\right)  ^{2}\left\vert
\mathsf{P}_{\mathcal{C}_{\mathcal{F}}\left(  F\right)  \cap\mathcal{D}\left[
W_{i}^{F}\right]  }^{\omega}\right\vert Z\left(  x\right)  ^{2}d\omega\left(
x\right)  .
\]
For term $B$ we write%
\begin{align*}
B  & \lesssim\int_{I}\sum_{F\in\mathcal{F}:\ F\supsetneqq I}\sum
_{W\in\mathcal{M}_{\left(  r,\varepsilon\right)  -\operatorname*{deep}}\left(
F\right)  \cap\mathcal{C}_{\mathcal{F}}\left(  F\right)  \cap\mathcal{D}%
\left[  I\right]  }\left(  \frac{\mathrm{P}\left(  W,\mathbf{1}_{F\cap
I}\sigma\right)  }{\ell\left(  W\right)  }\right)  ^{2}\left\vert
\mathsf{P}_{\mathcal{C}_{\mathcal{F}}\left(  F\right)  \cap\mathcal{D}\left[
W\right]  }^{\omega}\right\vert Z\left(  x\right)  ^{2}d\omega\left(
x\right)  \\
& +\int_{I}\sum_{F\in\mathcal{F}:\ F\supsetneqq I}\sum_{W\in\mathcal{M}%
_{\left(  r,\varepsilon\right)  -\operatorname*{deep}}\left(  F\right)
\cap\mathcal{C}_{\mathcal{F}}\left(  F\right)  :\ I\subsetneqq W}\left(
\frac{\mathrm{P}\left(  W,\mathbf{1}_{F\cap I}\sigma\right)  }{\ell\left(
W\right)  }\right)  ^{2}\left\vert \mathsf{P}_{\mathcal{C}_{\mathcal{F}%
}\left(  F\right)  \cap\mathcal{D}\left[  W\right]  }^{\omega}\right\vert
Z\left(  x\right)  ^{2}d\omega\left(  x\right)  \\
& \equiv C+D.
\end{align*}

For term $C$ there is by (\ref{overlap}) at most one $F\in\mathcal{F}$ for
which both
\[
F\supsetneqq I\text{ and }W\in\mathcal{M}_{\left(  r,\varepsilon\right)
-\operatorname*{deep}}\left(  F\right)  \cap\mathcal{C}_{\mathcal{F}}\left(
F\right)  \cap\mathcal{D}\left[  I\right]  .
\]
If we denote this $F$ by $F_{I}$, then the estimate is again easy using the
energy characteristic,%
\[
C\lesssim\int_{I}\left(  \sum_{W\in\mathcal{M}_{\left(  r,\varepsilon\right)
-\operatorname*{deep}}\left(  F_{I}\right)  \cap\mathcal{C}_{\mathcal{F}%
}\left(  F_{I}\right)  \cap\mathcal{D}\left[  I\right]  }\left(
\frac{\mathrm{P}\left(  W,\mathbf{1}_{I}\sigma\right)  }{\ell\left(  W\right)
}\right)  ^{2}\left\vert \mathsf{P}_{\mathcal{C}_{\mathcal{F}}\left(
F\right)  \cap\mathcal{D}\left[  W\right]  }^{\omega}\right\vert Z\left(
x\right)  ^{2}\right)  d\omega\left(  x\right)  \leq\mathcal{E}_{2}\left(
\sigma,\omega\right)  ^{2}\left\vert I\right\vert _{\sigma}.
\]
For term $D$ we have%
\begin{align*}
D  & =\int_{I}\sum_{F\in\mathcal{F}:\ F\supsetneqq I}\sum_{W\in\mathcal{M}%
_{\left(  r,\varepsilon\right)  -\operatorname*{deep}}\left(  F\right)
\cap\mathcal{C}_{\mathcal{F}}\left(  F\right)  :\ I\subsetneqq W}\left(
\frac{\mathrm{P}\left(  W,\mathbf{1}_{I}\sigma\right)  }{\ell\left(  W\right)
}\right)  ^{2}\left\vert \mathsf{P}_{\mathcal{C}_{\mathcal{F}}\left(
F\right)  \cap\mathcal{D}\left[  W\right]  }^{\omega}\right\vert Z\left(
x\right)  ^{2}d\omega\left(  x\right)  \\
& \lesssim\int_{I}\sum_{F\in\mathcal{F}:\ F\supsetneqq I}\sum_{W\in
\mathcal{M}_{\left(  r,\varepsilon\right)  -\operatorname*{deep}}\left(
F\right)  \cap\mathcal{C}_{\mathcal{F}}\left(  F\right)  :\ I\subsetneqq
W}\left(  \frac{\left\vert I\right\vert _{\sigma}}{\ell\left(  W\right)  ^{2}%
}\right)  ^{2}\ell\left(  W\right)  ^{2}\mathbf{1}_{W}\left(  x\right)
d\omega\left(  x\right)  \\
& =\int_{I}\left(  \sum_{F\in\mathcal{F}:\ F\supsetneqq I}\sum_{W\in
\mathcal{M}_{\left(  r,\varepsilon\right)  -\operatorname*{deep}}\left(
F\right)  \cap\mathcal{C}_{\mathcal{F}}\left(  F\right)  :\ I\subsetneqq
W}\frac{1}{\ell\left(  W\right)  ^{2}}\mathbf{1}_{W}\left(  x\right)  \right)
d\omega\left(  x\right)  \ \left\vert I\right\vert _{\sigma}^{2}\\
& \lesssim\int_{I}\left(  \frac{1}{\left(  \ell\left(  I\right)  +\left\vert
x-c_{I}\right\vert \right)  ^{2}}\right)  d\omega\left(  x\right)
\ \left\vert I\right\vert _{\sigma}^{2}\\
& =\left(  \frac{1}{\left\vert I\right\vert }\int_{I}\left(  \frac{\ell\left(
I\right)  }{\ell\left(  I\right)  +\left\vert x-c_{I}\right\vert }\right)
^{2}d\omega\left(  x\right)  \right)  \ \left(  \frac{\left\vert I\right\vert
_{\sigma}}{\left\vert I\right\vert }\right)  \left\vert I\right\vert _{\sigma
}\leq\mathcal{A}_{2}\left(  \sigma,\omega\right)  ^{2}\left\vert I\right\vert
_{\sigma}\ .
\end{align*}

Now we estimate the local holed term,%
\begin{align*}
\mathbf{Local}^{\operatorname{hole}}\left(  I\right)    & =\int_{I}\sum
_{F\in\mathcal{F}}\sum_{W\in\mathcal{M}_{\left(  r,\varepsilon\right)
-\operatorname*{deep}}\left(  F\right)  \cap\mathcal{C}_{\mathcal{F}}\left(
F\right)  }\left(  \frac{\mathrm{P}\left(  W,\mathbf{1}_{I\setminus F}%
\sigma\right)  }{\ell\left(  W\right)  }\right)  ^{2}\left\vert \mathsf{P}%
_{\mathcal{C}_{\mathcal{F}}\left(  F\right)  \cap\mathcal{D}\left[  W\right]
}^{\omega}\right\vert Z\left(  x\right)  ^{2}d\omega\left(  x\right)  \\
& \lesssim\int_{I}\sum_{F\in\mathcal{F}:\ F\subset I}\sum_{W\in\mathcal{M}%
_{\left(  r,\varepsilon\right)  -\operatorname*{deep}}\left(  F\right)
\cap\mathcal{C}_{\mathcal{F}}\left(  F\right)  }\left\vert \mathsf{P}%
_{\mathcal{C}_{\mathcal{F}}\left(  F\right)  \cap\mathcal{D}\left[  W\right]
}^{\omega}H_{\sigma}\mathbf{1}_{I\setminus F}\left(  x\right)  \right\vert
^{2}d\omega\left(  x\right)  \\
& \lesssim\int_{I}\sum_{F\in\mathcal{F}:\ F\subset I}\sum_{W\in\mathcal{M}%
_{\left(  r,\varepsilon\right)  -\operatorname*{deep}}\left(  F\right)
\cap\mathcal{C}_{\mathcal{F}}\left(  F\right)  }\left\vert \mathsf{P}%
_{\mathcal{C}_{\mathcal{F}}\left(  F\right)  \cap\mathcal{D}\left[  W\right]
}^{\omega}H_{\sigma}\mathbf{1}_{I}\left(  x\right)  \right\vert ^{2}%
d\omega\left(  x\right)  \\
& +\int_{I}\sum_{F\in\mathcal{F}:\ F\subset I}\sum_{W\in\mathcal{M}_{\left(
r,\varepsilon\right)  -\operatorname*{deep}}\left(  F\right)  \cap
\mathcal{C}_{\mathcal{F}}\left(  F\right)  }\left\vert \mathsf{P}%
_{\mathcal{C}_{\mathcal{F}}\left(  F\right)  \cap\mathcal{D}\left[  W\right]
}^{\omega}H_{\sigma}\mathbf{1}_{F}\left(  x\right)  \right\vert ^{2}%
d\omega\left(  x\right)  \\
& \equiv E+F,
\end{align*}
where%
\[
E\leq\int_{I}\left(  \sum_{F\in\mathcal{F}:\ F\subset I}\left\vert
\mathsf{P}_{\mathcal{C}_{\mathcal{F}}\left(  F\right)  }^{\omega}H_{\sigma
}\mathbf{1}_{I}\left(  x\right)  \right\vert ^{2}\right)  d\omega\left(
x\right)  \leq\int_{I}\left\vert H_{\sigma}\mathbf{1}_{I}\left(  x\right)
\right\vert ^{2}d\omega\left(  x\right)  \leq\mathfrak{T}_{H}%
^{\operatorname{loc}}\left(  \sigma,\omega\right)  ^{2}\left\vert I\right\vert
_{\sigma}\ ,
\]
and%
\begin{align*}
F  & \leq\int_{I}\left(  \sum_{F\in\mathcal{F}:\ F\subset I}\left\vert
\mathsf{P}_{\mathcal{C}_{\mathcal{F}}\left(  F\right)  }^{\omega}H_{\sigma
}\mathbf{1}_{F}\left(  x\right)  \right\vert ^{2}\right)  d\omega\left(
x\right)  \leq\int_{I}\left(  \sum_{F\in\mathcal{F}:\ F\subset I}\left\vert
M_{\omega}\mathbf{1}_{F}H_{\sigma}\mathbf{1}_{F}\left(  x\right)  \right\vert
^{2}\right)  d\omega\left(  x\right)  \\
& \lesssim\int_{I}\left(  \sum_{F\in\mathcal{F}:\ F\subset I}\left\vert
\mathbf{1}_{F}H_{\sigma}\mathbf{1}_{F}\left(  x\right)  \right\vert
^{2}\right)  d\omega\left(  x\right)  \leq\mathfrak{T}_{H}^{\operatorname{loc}%
}\left(  \sigma,\omega\right)  ^{2}\int_{I}\left(  \sum_{F\in\mathcal{F}%
:\ F\subset I}\mathbf{1}_{F}\left(  y\right)  \right)  d\sigma\left(
y\right)  \lesssim\mathfrak{T}_{H}^{\operatorname{loc}}\left(  \sigma
,\omega\right)  ^{2}\left\vert I\right\vert _{\sigma}\ .
\end{align*}

\subsection{Global forward testing}

We begin by decomposing the integral on the left of the global term into four
pieces. We have,%
\begin{align*}
& \mathbf{Global}\left(  I\right)  =\int_{\mathbb{R}\setminus I}T_{\sigma
}\mathbf{1}_{I}\left(  x\right)  ^{2}d\omega\left(  x\right)  =\int
_{\mathbb{R}\setminus I}\left(  \sum_{F\in\mathcal{F}}\sum_{W\in
\mathcal{M}_{\left(  r,\varepsilon\right)  -\operatorname*{deep}}\left(
F\right)  \cap\mathcal{C}_{\mathcal{F}}\left(  F\right)  }\frac{\mathrm{P}%
\left(  W,\mathbf{1}_{I}\sigma\right)  }{\ell\left(  W\right)  }%
\mathsf{P}_{\mathcal{C}_{\mathcal{F}}\left(  F\right)  \cap\mathcal{D}\left[
W\right]  }^{\omega}Z\left(  x\right)  \right)  ^{2}d\omega\left(  x\right)
\\
& =\int_{\mathbb{R}\setminus I}\sum_{F\in\mathcal{F}}\sum_{W\in\mathcal{M}%
_{\left(  r,\varepsilon\right)  -\operatorname*{deep}}\left(  F\right)
\cap\mathcal{C}_{\mathcal{F}}\left(  F\right)  }\left(  \frac{\mathrm{P}%
\left(  W,\mathbf{1}_{I}\sigma\right)  }{\ell\left(  W\right)  }\right)
^{2}\mathsf{P}_{\mathcal{C}_{\mathcal{F}}\left(  F\right)  \cap\mathcal{D}%
\left[  W\right]  }^{\omega}Z\left(  x\right)  ^{2}d\omega\left(  x\right)
\\
& =\int_{\mathbb{R}\setminus I}\left\{  \sum_{\substack{W\cap3I=\emptyset
\\\ell\left(  W\right)  \leq\ell\left(  I\right)  }}+\sum_{W\subset3I\setminus
I}+\sum_{\substack{W\cap I=\emptyset\\\ell\left(  W\right)  >\ell\left(
I\right)  }}+\sum_{W\supsetneqq I}\right\}  \sum_{F\in\mathcal{F}%
:\ W\in\mathcal{M}_{\left(  r,\varepsilon\right)  -\operatorname*{deep}%
}\left(  F\right)  \cap\mathcal{C}_{\mathcal{F}}\left(  F\right)  }\left(
\frac{\mathrm{P}\left(  W,\mathbf{1}_{I}\sigma\right)  }{\ell\left(  W\right)
}\right)  ^{2}\left\vert \mathsf{P}_{\mathcal{C}_{\mathcal{F}}\left(
F\right)  \cap\mathcal{D}\left[  W\right]  }^{\omega}\right\vert Z\left(
x\right)  ^{2}d\omega\left(  x\right)  \\
& \lesssim A+B+C+D,
\end{align*}
where the four sums over $W$ in braces are taken over $W\in\mathcal{M}%
_{r-\operatorname*{deep}}\left(  F\right)  $, and where the\ four terms
$A,B,C,D$ equal the integral in the previous line taken over the respective sum.

We claim that
\begin{align*}
A+B  & \lesssim A_{2}^{\operatorname*{offset}}\left(  \sigma,\omega\right)
^{2}\lesssim\mathcal{A}_{2}\left(  \sigma,\omega\right)  ^{2},\\
C+D  & \lesssim\mathcal{A}_{2}\left(  \sigma,\omega\right)  ^{2}.
\end{align*}
First we further decompose term $A$ according to the length of $W$ and its
distance from $I$, and then use (\ref{overlap}) to obtain:%
\begin{align*}
& A=\int_{\mathbb{R}\setminus I}\sum_{m=0}^{\infty}\sum_{k=1}^{\infty}%
\sum_{\substack{W\subset3^{k+1}I\setminus3^{k}I\\\ell\left(  W\right)
=2^{-m}\ell\left(  I\right)  }}\sum_{F\in\mathcal{F}:\ W\in\mathcal{M}%
_{\left(  r,\varepsilon\right)  -\operatorname*{deep}}\left(  F\right)
\cap\mathcal{C}_{\mathcal{F}}\left(  F\right)  }\left(  \frac{\mathrm{P}%
\left(  W,\mathbf{1}_{I}\sigma\right)  }{\ell\left(  W\right)  }\right)
^{2}\left\vert \mathsf{P}_{\mathcal{C}_{\mathcal{F}}\left(  F\right)
\cap\mathcal{D}\left[  W\right]  }^{\omega}\right\vert ^{2}Z\left(  x\right)
d\omega\left(  x\right)  \\
& \lesssim\int_{\mathbb{R}\setminus I}\sum_{m=0}^{\infty}\sum_{k=1}^{\infty
}\sum_{\substack{W\subset3^{k+1}I\setminus3^{k}I\\\ell\left(  W\right)
=2^{-m}\ell\left(  I\right)  }}\mathrm{P}\left(  W,\mathbf{1}_{I}%
\sigma\right)  ^{2}\mathbf{1}_{W}\left(  x\right)  d\omega\left(  x\right)
\\
& \lesssim\sum_{m=0}^{\infty}\sum_{k=1}^{\infty}\sum_{\substack{W\subset
3^{k+1}I\setminus3^{k}I\\\ell\left(  W\right)  =2^{-m}\ell\left(  I\right)
}}\int_{\mathbb{R}\setminus I}\left(  2^{-m}\frac{\ell\left(  I\right)
}{\operatorname*{dist}\left(  W,I\right)  ^{2}}\left\vert I\right\vert
_{\sigma}\mathbf{1}_{W}\left(  x\right)  \right)  ^{2}d\omega\left(  x\right)
\\
& \lesssim\sum_{m=0}^{\infty}\sum_{k=1}^{\infty}\int_{\mathbb{R}\setminus
I}\left(  2^{-m}\frac{\ell\left(  I\right)  }{\left(  3^{k}\ell\left(
I\right)  \right)  ^{2}}\left\vert I\right\vert _{\sigma}\mathbf{1}%
_{3^{k+1}I\setminus3^{k}I}\left(  x\right)  \right)  ^{2}d\omega\left(
x\right)  ,
\end{align*}
which equals,%
\begin{align*}
& \sum_{m=0}^{\infty}\sum_{k=1}^{\infty}\left(  2^{-m}\frac{\ell\left(
I\right)  }{\left(  3^{k}\ell\left(  I\right)  \right)  ^{2}}\left\vert
I\right\vert _{\sigma}\right)  ^{2}\left\vert 3^{k+1}I\setminus3^{k}%
I\right\vert _{\omega}=\sum_{m=0}^{\infty}\sum_{k=1}^{\infty}2^{-m}\frac
{\ell\left(  I\right)  }{\left(  3^{k}\ell\left(  I\right)  \right)  ^{2}%
}\left\vert I\right\vert _{\sigma}\left\vert 3^{k+1}I\setminus3^{k}%
I\right\vert _{\omega}\\
& =\sum_{m=0}^{\infty}\sum_{k=1}^{\infty}2^{-m}3^{-k}\left(  \frac{\left\vert
I\right\vert _{\sigma}\left\vert 3^{k+1}I\setminus3^{k}I\right\vert _{\omega}%
}{3^{k}\ell\left(  I\right)  }\right)  \left\vert I\right\vert _{\sigma}%
\leq3A_{2}^{\operatorname*{offset}}\left(  \sigma,\omega\right)  ^{2}%
\sum_{m=0}^{\infty}\sum_{k=1}^{\infty}2^{-m}3^{-k}\left\vert I\right\vert
_{\sigma}\lesssim A_{2}^{\operatorname*{offset}}\left(  \sigma,\omega\right)
^{2}\left\vert I\right\vert _{\sigma}\ .
\end{align*}

We further decompose term $B$ according to the length of $W$ and use the
Poisson inequality (\ref{e.Jsimeq}) in Lemma \ref{Poisson inequality} on the
(not necessarily dyadic) sibling $I^{\prime}$ of $I$ containing $W$,%
\[
\mathrm{P}\left(  W,\mathbf{1}_{I}\sigma\right)  \lesssim\left(  \frac
{\ell\left(  W\right)  }{\ell\left(  I\right)  }\right)  ^{1-2\varepsilon
}\mathrm{P}\left(  I,\mathbf{1}_{I}\sigma\right)  ,\ \ \ \ \ W\in
\mathcal{M}_{\mathbf{r}-\operatorname*{deep}}\left(  F\right)  ,W\subset
3I\setminus I,
\]
where we have used that $\mathrm{P}\left(  I^{\prime},\mathbf{1}_{I}%
\sigma\right)  \approx\mathrm{P}\left(  I,\mathbf{1}_{I}\sigma\right)  $ and
that the intervals $W\in\mathcal{M}_{\mathbf{r}-\operatorname*{deep}}\left(
F\right)  $ are good. We then obtain from (\ref{overlap}),%
\begin{align*}
& B=\int_{\mathbb{R}\setminus I}\sum_{m=0}^{\infty}\sum_{\substack{W\subset
3I\setminus I\\\ell\left(  W\right)  =2^{-m}\ell\left(  I\right)  }}\sum
_{F\in\mathcal{F}:\ W\in\mathcal{M}_{\left(  r,\varepsilon\right)
-\operatorname*{deep}}\left(  F\right)  \cap\mathcal{C}_{\mathcal{F}}\left(
F\right)  }\left(  \frac{\mathrm{P}\left(  W,\mathbf{1}_{I}\sigma\right)
}{\ell\left(  W\right)  }\right)  ^{2}\left\vert \mathsf{P}_{\mathcal{C}%
_{\mathcal{F}}\left(  F\right)  \cap\mathcal{D}\left[  W\right]  }^{\omega
}\right\vert ^{2}Z\left(  x\right)  d\omega\left(  x\right)  \\
& \lesssim\sum_{m=0}^{\infty}\int_{\mathbb{R}\setminus I}\sum
_{\substack{W\subset3I\setminus I\\\ell\left(  W\right)  =2^{-m}\ell\left(
I\right)  }}\mathrm{P}\left(  W,\mathbf{1}_{I}\sigma\right)  ^{2}%
\mathbf{1}_{W}\left(  x\right)  d\omega\left(  x\right)  \\
& \lesssim\sum_{m=0}^{\infty}\int_{\mathbb{R}\setminus I}\sum
_{\substack{W\subset3I\setminus I\\\ell\left(  W\right)  =2^{-m}\ell\left(
I\right)  }}\left(  2^{-m}\right)  ^{2-4\varepsilon}\mathrm{P}\left(
I,\mathbf{1}_{I}\sigma\right)  ^{2}\mathbf{1}_{W}\left(  x\right)
d\omega\left(  x\right)  \\
& =\sum_{m=0}^{\infty}\left(  2^{-m}\right)  ^{2-4\varepsilon}\mathrm{P}%
\left(  I,\mathbf{1}_{I}\sigma\right)  ^{2}\left\vert 3I\setminus I\right\vert
_{\omega}\approx\sum_{m=0}^{\infty}\left(  2^{-m}\right)  ^{2-4\varepsilon
}\frac{\left\vert I\right\vert _{\sigma}\left\vert 3I\setminus I\right\vert
_{\omega}}{\left\vert I\right\vert }\left\vert I\right\vert _{\sigma}\lesssim
A_{2}^{\operatorname*{offset}}\left(  \sigma,\omega\right)  ^{2}\left\vert
I\right\vert _{\sigma}\ .
\end{align*}

For term $C$ we will have to group the intervals $W$ into blocks $B_{i}$, and
then exploit (\ref{overlap}). We first split the sum according to whether or
not $I$ intersects the triple of $W$:%
\begin{align*}
C  & =\int_{\mathbb{R}\setminus I}\left\{  \sum_{\substack{W:\ I\cap
3W=\emptyset\\\ell\left(  W\right)  >\ell\left(  I\right)  }}+\sum
_{\substack{W:\ I\subset3W\setminus W\\\ell\left(  W\right)  >\ell\left(
I\right)  }}\right\}  \sum_{F\in\mathcal{F}:\ W\in\mathcal{M}_{\left(
r,\varepsilon\right)  -\operatorname*{deep}}\cap\mathcal{C}_{\mathcal{F}%
}\left(  F\right)  }\left(  \frac{\mathrm{P}\left(  W,\mathbf{1}_{I}%
\sigma\right)  }{\ell\left(  W\right)  }\right)  ^{2}\left\vert \mathsf{P}%
_{\mathcal{C}_{F}^{\operatorname*{good}};W}^{\omega}\right\vert Z\left(
x\right)  ^{2}d\omega\left(  x\right)  \\
& \lesssim C_{1}+C_{2}\ .
\end{align*}
For convenience we recall the scalar one-tailed Muckenhoupt condition with
holes,%
\[
\mathcal{A}_{2}\left(  \sigma,\omega\right)  \approx\sup_{Q\text{ an
interval}}\left(  \frac{1}{\left\vert Q\right\vert }\int_{\mathbb{R}\setminus
Q}\left(  \frac{\ell\left(  Q\right)  }{\ell\left(  Q\right)
+\operatorname*{dist}\left(  x,Q\right)  }\right)  ^{2}d\omega\left(
x\right)  \right)  ^{\frac{1}{2}}\ \left(  \frac{\left\vert Q\right\vert
_{\sigma}}{\left\vert Q\right\vert }\right)  ^{\frac{1}{2}}.
\]

We first consider $C_{1}$. Let $\mathcal{M}$ be the maximal dyadic intervals
in $\left\{  Q:3Q\cap I=\emptyset\right\}  $, and then let $\left\{
B_{i}\right\}  _{i=1}^{\infty}$ be an enumeration of those $Q\in\mathcal{M}$
whose side length is at least $\ell\left(  I\right)  $. Now we further
decompose the sum in $C_{1}$ by grouping the intervals $W$ into the Whitney
intervals $B_{i}$, and then using (\ref{overlap}),%
\begin{align*}
C_{1}  & =\int_{\mathbb{R}\setminus I}\sum_{i=1}^{\infty}\sum_{W:\ W\subset
B_{i}}\sum_{F\in\mathcal{F}:\ W\in\mathcal{M}_{\left(  r,\varepsilon\right)
-\operatorname*{deep}}\cap\mathcal{C}_{\mathcal{F}}\left(  F\right)  }\left(
\frac{\mathrm{P}\left(  W,\mathbf{1}_{I}\sigma\right)  }{\ell\left(  W\right)
}\right)  ^{2}\left\vert \mathsf{P}_{\mathcal{C}_{\mathcal{F}}\left(
F\right)  \cap\mathcal{D}\left[  W\right]  }^{\omega}\right\vert Z\left(
x\right)  ^{2}d\omega\left(  x\right)  \\
& \lesssim\int_{\mathbb{R}\setminus I}\sum_{i=1}^{\infty}\left(
\frac{\left\vert I\right\vert _{\sigma}}{\left(  \ell\left(  B_{i}\right)
+\operatorname*{dist}\left(  B_{i},I\right)  \right)  ^{2}}\right)  ^{2}%
\sum_{W:\ W\subset B_{i}}\sum_{F\in\mathcal{F}:\ W\in\mathcal{M}_{\left(
r,\varepsilon\right)  -\operatorname*{deep}}\cap\mathcal{C}_{\mathcal{F}%
}\left(  F\right)  }\left\vert \mathsf{P}_{\mathcal{C}_{\mathcal{F}}\left(
F\right)  \cap\mathcal{D}\left[  W\right]  }^{\omega}\right\vert Z\left(
x\right)  ^{2}d\omega\left(  x\right)  \\
& \lesssim\int_{\mathbb{R}\setminus I}\left(  \sum_{i=1}^{\infty}%
\frac{\left\vert I\right\vert _{\sigma}}{\left(  \ell\left(  B_{i}\right)
+\operatorname*{dist}\left(  B_{i},I\right)  \right)  ^{2}}\sum_{W:\ W\subset
B_{i}}\ell\left(  W\right)  \mathbf{1}_{W}\left(  x\right)  \right)
^{2}d\omega\left(  x\right)  ,
\end{align*}
which is at most%
\begin{align*}
& \int_{\mathbb{R}\setminus I}\left(  \sum_{i=1}^{\infty}\frac{\left\vert
I\right\vert _{\sigma}}{\left(  \ell\left(  B_{i}\right)
+\operatorname*{dist}\left(  B_{i},I\right)  \right)  ^{2}}\ell\left(
B_{i}\right)  \mathbf{1}_{B_{i}}\left(  x\right)  \right)  ^{2}d\omega\left(
x\right)  \\
& =\left\vert I\right\vert _{\sigma}^{p}\int_{\mathbb{R}\setminus I}\left(
\sum_{i=1}^{\infty}\frac{\ell\left(  B_{i}\right)  }{\left(  \ell\left(
B_{i}\right)  +\operatorname*{dist}\left(  B_{i},I\right)  \right)  ^{2}%
}\mathbf{1}_{B_{i}}\left(  x\right)  \right)  ^{2}d\omega\left(  x\right)  \ .
\end{align*}

Since the intervals $B_{i}$ are pairwise disjoint, the last line above is at
most%
\[
\lesssim\left\vert I\right\vert _{\sigma}\sum_{i=1}^{\infty}\left(  \frac
{\ell\left(  B_{i}\right)  }{\left(  \ell\left(  B_{i}\right)
+\operatorname*{dist}\left(  B_{i},I\right)  \right)  ^{2}}\right)
^{2}\left\vert B_{i}\right\vert _{\omega}\ \left\vert I\right\vert _{\sigma
}\lesssim\left\{  \sum_{i=1}^{\infty}\frac{\left\vert B_{i}\right\vert
_{\omega}\left\vert I\right\vert _{\sigma}}{\ell\left(  B_{i}\right)  ^{p}%
}\right\}  \left\vert I\right\vert _{\sigma}\ ,
\]
and using,
\begin{align*}
\sum_{i=1}^{\infty}\frac{\left\vert B_{i}\right\vert _{\omega}\left\vert
I\right\vert _{\sigma}}{\left\vert B_{i}\right\vert ^{p}}  & =\frac{\left\vert
I\right\vert _{\sigma}}{\left\vert I\right\vert }\sum_{i=1}^{\infty}%
\frac{\left\vert I\right\vert }{\left\vert B_{i}\right\vert ^{2}}\left\vert
B_{i}\right\vert _{\omega}\approx\frac{\left\vert I\right\vert _{\sigma}%
}{\left\vert I\right\vert }\frac{1}{\left\vert I\right\vert }\sum
_{i=1}^{\infty}\int_{B_{i}}\left(  \frac{\ell\left(  I\right)  }%
{\operatorname*{dist}\left(  x,I\right)  }\right)  ^{2}d\omega\left(
x\right)  \\
& \lesssim\frac{\left\vert I\right\vert _{\sigma}}{\left\vert I\right\vert
}\frac{1}{\left\vert I\right\vert }\int_{\mathbb{R}\setminus I}\left(
\frac{\ell\left(  I\right)  }{\operatorname*{dist}\left(  x,I\right)
}\right)  ^{2}d\omega\left(  x\right)  \leq\mathcal{A}_{2}%
^{\operatorname*{hole}}\left(  \sigma,\omega\right)  ^{2},
\end{align*}
we obtain $C_{1}\lesssim\mathcal{A}_{2}\left(  \sigma,\omega\right)
^{2}\left\vert I\right\vert _{\sigma}$.

Next we turn to estimating term $C_{2}$ where the triple of $W$ contains $I$
but $W$ itself does not. Note that there are at most two such intervals $W$ of
a given side length, one to each side of $I$, and that these intervals are
pairwise disjoint. So with this in mind, and using (\ref{overlap}) again, we
sum over the intervals $W$ according to their lengths to obtain,%

\begin{align*}
C_{2}  & =\int_{\mathbb{R}\setminus I}\sum_{m=0}^{\infty}\sum
_{\substack{W:\ I\subset3W\setminus W\\\ell\left(  W\right)  =2^{m}\ell\left(
I\right)  }}\left(  \frac{\mathrm{P}\left(  W,\mathbf{1}_{I}\sigma\right)
}{\ell\left(  W\right)  }\right)  ^{2}\sum_{F\in\mathcal{F}:\ W\in
\mathcal{M}_{\left(  r,\varepsilon\right)  -\operatorname*{deep}}\left(
F\right)  \cap\mathcal{C}_{\mathcal{F}}\left(  F\right)  }\left\vert
\mathsf{P}_{\mathcal{C}_{\mathcal{F}}\left(  F\right)  \cap\mathcal{D}\left[
W\right]  }^{\omega}\right\vert Z\left(  x\right)  ^{2}d\omega\left(
x\right)  \\
& =\int_{\mathbb{R}\setminus I}\sum_{m=0}^{\infty}\sum_{\substack{W:\ I\subset
3W\setminus W\\\ell\left(  W\right)  =2^{m}\ell\left(  I\right)  }}\left(
\frac{\left\vert I\right\vert _{\sigma}}{\left(  \ell\left(  W\right)
+\operatorname*{dist}\left(  W,I\right)  \right)  ^{2}}\right)  ^{2}\sum
_{F\in\mathcal{F}:\ W\in\mathcal{M}_{\left(  r,\varepsilon\right)
-\operatorname*{deep}}\left(  F\right)  \cap\mathcal{C}_{\mathcal{F}}\left(
F\right)  }\left\vert \mathsf{P}_{\mathcal{C}_{\mathcal{F}}\left(  F\right)
\cap\mathcal{D}\left[  W\right]  }^{\omega}\right\vert \frac{Z\left(
x\right)  }{\ell\left(  W\right)  }^{2}d\omega\left(  x\right)  \\
& \lesssim\int_{\mathbb{R}\setminus I}\sum_{m=0}^{\infty}\sum
_{\substack{W:\ I\subset3W\setminus W\\\ell\left(  W\right)  =2^{m}\ell\left(
I\right)  }}\left(  \frac{\left\vert I\right\vert _{\sigma}}{\left(
\ell\left(  W\right)  +\operatorname*{dist}\left(  W,I\right)  \right)  ^{2}%
}\right)  ^{2}\mathbf{1}_{W}\left(  x\right)  d\omega\left(  x\right)  \\
& \lesssim\sum_{m=0}^{\infty}\sum_{\substack{W:\ I\subset3W\setminus
W\\\ell\left(  W\right)  =2^{m}\ell\left(  I\right)  }}\left(  \frac
{\left\vert I\right\vert _{\sigma}}{\left(  \ell\left(  W\right)
+\operatorname*{dist}\left(  W,I\right)  \right)  ^{2}}\right)  ^{2}\left\vert
W\right\vert _{\omega}\ ,
\end{align*}
which is at most,%
\begin{align*}
& \sum_{m=0}^{\infty}\sum_{\substack{W:\ I\subset3W\setminus W\\\ell\left(
W\right)  =2^{m}\ell\left(  I\right)  }}\left(  \frac{\left\vert I\right\vert
_{\sigma}}{\left(  \ell\left(  W\right)  +\operatorname*{dist}\left(
W,I\right)  \right)  ^{2}}\right)  ^{2}\left\vert W\right\vert _{\omega
}\lesssim\sum_{m=0}^{\infty}\left(  \frac{\left\vert I\right\vert _{\sigma}%
}{\left\vert 2^{m}I\right\vert ^{2}}\right)  ^{2}\mathbf{\ }\left\vert
3\cdot2^{m}I\right\vert _{\omega}\\
& =\left\{  \frac{\left\vert I\right\vert _{\sigma}}{\left\vert I\right\vert
}\sum_{m=0}^{\infty}\frac{\left\vert I\right\vert \left\vert 3\cdot
2^{m}I\right\vert _{\omega}}{\left\vert 2^{m}I\right\vert ^{4}}\right\}
\left\vert I\right\vert _{\sigma}\lesssim\mathcal{A}_{2}\left(  \sigma
,\omega\right)  ^{2}\left\vert I\right\vert _{\sigma}\ ,
\end{align*}
since in analogy with the corresponding estimate above,%
\begin{align*}
\sum_{m=0}^{\infty}\frac{\left\vert I\right\vert \left\vert 3\cdot
2^{m}I\right\vert _{\omega}}{\left\vert 2^{m}I\right\vert ^{4}}  & =\int
\sum_{m=0}^{\infty}\frac{\left\vert I\right\vert }{\left\vert 2^{m}%
I\right\vert ^{4}}\mathbf{1}_{3\cdot2^{m}I}\left(  x\right)  \ d\omega\left(
x\right)  \lesssim\frac{1}{\left\vert I\right\vert }\int\sum_{m=0}^{\infty
}\left(  \frac{\left\vert I\right\vert }{\left\vert 2^{m}I\right\vert ^{2}%
}\right)  ^{2}\mathbf{1}_{3\cdot2^{m}I}\left(  x\right)  \ d\omega\left(
x\right)  \\
& \lesssim\frac{1}{\left\vert I\right\vert }\int_{\mathbb{R}\setminus
I}\left(  \frac{\ell\left(  I\right)  }{\ell\left(  I\right)  +\left\vert
\operatorname*{dist}\left(  x,I\right)  \right\vert }\right)  ^{2}%
\ d\omega\left(  x\right)  .
\end{align*}
Altogether then we have%
\[
C\lesssim C_{1}+C_{2}\lesssim\mathcal{A}_{2}^{\operatorname*{hole}}\left(
\sigma,\omega\right)  ^{2}\left\vert I\right\vert _{\sigma}\ .
\]

Finally, we turn to term $D$, which is handled in the same way as term $C_{2}%
$. The intervals $W$ occurring here are included in the set of ancestors
$A_{k}\equiv\pi_{\mathcal{D}}^{\left(  k\right)  }I$ of $I$, $1\leq k<\infty$.
We thus have from (\ref{overlap}) once more,%
\begin{align*}
D  & =\int_{\mathbb{R}\setminus I}\left(  \sum_{W\supsetneqq I}\sum
_{F\in\mathcal{F}:\ W\in\mathcal{M}_{\left(  r,\varepsilon\right)
-\operatorname*{deep}}\left(  F\right)  \cap\mathcal{C}_{\mathcal{F}}\left(
F\right)  }\frac{\mathrm{P}\left(  W,\mathbf{1}_{I}\sigma\right)  }%
{\ell\left(  W\right)  }\mathsf{P}_{\mathcal{C}_{\mathcal{F}}\left(  F\right)
\cap\mathcal{D}\left[  W\right]  }^{\omega}Z\left(  x\right)  \right)
^{2}d\omega\left(  x\right)  \\
& \leq\int_{\mathbb{R}\setminus I}\sum_{k=1}^{\infty}\left(  \frac
{\mathrm{P}\left(  A_{k},\mathbf{1}_{I}\sigma\right)  }{\ell\left(
A_{k}\right)  }\right)  ^{2}\sum_{F\in\mathcal{F}:\ A_{k}\in\mathcal{M}%
_{\left(  r,\varepsilon\right)  -\operatorname*{deep}}\left(  F\right)
\cap\mathcal{C}_{\mathcal{F}}\left(  F\right)  }\left\vert \mathsf{P}%
_{\mathcal{C}_{\mathcal{F}}\left(  F\right)  \cap\mathcal{D}\left[
A_{k}\right]  }^{\omega}\right\vert Z\left(  x\right)  ^{2}d\omega\left(
x\right)  \\
& \lesssim\int_{\mathbb{R}\setminus I}\left(  \sum_{k=1}^{\infty}%
\mathrm{P}\left(  A_{k},\mathbf{1}_{I}\sigma\right)  \mathbf{1}_{A_{k}}\left(
x\right)  \right)  ^{2}d\omega\left(  x\right)  \lesssim\int_{\mathbb{R}%
\setminus I}\left(  \sum_{k=1}^{\infty}\frac{\left\vert I\right\vert _{\sigma
}}{\ell\left(  A_{k}\right)  ^{2}}\mathbf{1}_{A_{k}}\left(  x\right)  \right)
^{2}d\omega\left(  x\right)
\end{align*}
which is at most%
\begin{align*}
& \frac{\left\vert I\right\vert _{\sigma}}{\left\vert I\right\vert }\frac
{1}{\left\vert I\right\vert }\int_{\mathbb{R}\setminus I}\left(  \sum
_{k=1}^{\infty}\frac{\ell\left(  I\right)  }{\operatorname*{dist}\left(
x,I\right)  ^{2}}\mathbf{1}_{A_{k}}\left(  x\right)  \right)  ^{2}%
d\omega\left(  x\right)  \ \left\vert I\right\vert _{\sigma}\\
& \lesssim\frac{\left\vert I\right\vert _{\sigma}}{\left\vert I\right\vert
}\frac{1}{\left\vert I\right\vert }\int_{\mathbb{R}\setminus I}\left(
\frac{\ell\left(  I\right)  }{\operatorname*{dist}\left(  x,I\right)  ^{2}%
}\right)  ^{2}d\omega\left(  x\right)  \ \left\vert I\right\vert _{\sigma
}\lesssim\mathcal{A}_{2}^{\operatorname*{hole}}\left(  \sigma,\omega\right)
^{2}\left\vert I\right\vert _{\sigma}\ .
\end{align*}


\begin{thebibliography}{999999999}                                                                                        %
\bibitem[Hyt2]{Hyt2}\textsc{Hyt\"{o}nen, Tuomas, }\textit{The two-weight
inequality\ for the Hilbert transform with general measures, }Proc. London
Math. Soc. (3) \textbf{117} (2018), 483-526.

\bibitem[Lac]{Lac}\textsc{Lacey, Michael T.,}\textit{\ Two weight inequality
for the Hilbert transform: A real variable characterization, II}, Duke Math.
J. Volume \textbf{163}, Number 15 (2014), 2821-2840.

\bibitem[LaSaShUr3]{LaSaShUr3}\textsc{Lacey, Michael T., Sawyer, Eric T.,
Shen, Chun-Yen, Uriarte-Tuero, Ignacio,} \textit{Two weight inequality for the
Hilbert transform: A real variable characterization I}, Duke Math. J, Volume
\textbf{163}, Number 15 (2014), 2795-2820.

\bibitem[LaSaUr2]{LaSaUr2}\textsc{Lacey, Michael T., Sawyer, Eric T.,
Uriarte-Tuero, Ignacio,} \textit{A Two Weight Inequality for the Hilbert
transform assuming an energy hypothesis, } Journal of Functional Analysis,
Volume \textbf{263} (2012), Issue 2, 305-363.

\bibitem[LaWi]{LaWi}\textsc{Lacey, Michael T., Wick, Brett D.,} \textit{Two
weight inequalities for Riesz transforms: uniformly full dimension weights},
\textit{\texttt{arXiv:1312.6163v3}}.

\bibitem[NTV1]{NTV1}\textsc{F. Nazarov, S. Treil and A. Volberg,} \textit{The
Bellman function and two weight inequalities for Haar multipliers}, J. Amer.
Math. Soc. \textbf{12} (1999), 909-928, MR\{1685781 (2000k:42009)\}.

\bibitem[NTV2]{NTV2}\textsc{Nazarov, F., Treil, S. and Volberg, A.,}
\textit{The }$Tb$\textit{-theorem on non-homogeneous spaces,} Acta Math.
\textbf{190} (2003), no. 2, MR 1998349 (2005d:30053).

\bibitem[NTV3]{NTV3}\textsc{F. Nazarov, S. Treil and A. Volberg,} \textit{Two
weight estimate for the Hilbert transform and corona decomposition for
non-doubling measures}, preprint (2004) \texttt{arxiv:1003.1596}

\bibitem[Saw]{Saw}\textsc{Sawyer, Eric T.,} \textit{Energy conditions and
twisted localizations of operators}, \texttt{arXiv:1801.03706}.

\bibitem[Saw3]{Saw3}\textsc{E. Sawyer,} \textit{A characterization of two
weight norm inequalities for fractional and Poisson integrals}, Trans. A.M.S.
\textbf{308} (1988), 533-545, MR\{930072 (89d:26009)\}.

\bibitem[SaShUr7]{SaShUr7}\textsc{Sawyer, Eric T., Shen, Chun-Yen,
Uriarte-Tuero, Ignacio,} A \textit{two weight theorem for }$\alpha
$\textit{-fractional singular integrals with an energy side condition},
Revista Mat. Iberoam. \textbf{32} (2016), no. 1, 79-174.

\bibitem[SaShUr11]{SaShUr11}\textsc{Sawyer, Eric T., Shen, Chun-Yen,
Uriarte-Tuero, Ignacio,} \textit{A counterexample in the theory of
Calder\'{o}n-Zygmund operators, }\texttt{arXiv:16079.06071v3v1}.

\bibitem[SaShUr12]{SaShUr12}\textsc{Sawyer, Eric T., Shen, Chun-Yen,
Uriarte-Tuero, Ignacio, }\textit{A two weight local }$Tb$ \textit{theorem for
the Hilbert transform}, Rev. Mat. Iberoam. \textbf{37} (2021), no. 2, 415--641
(see also \texttt{arXiv:1709.09595)}.

\bibitem[SaWi]{SaWi}\textsc{Sawyer, Eric T., Wick, Brett D., }\textit{The
Hyt\"{o}nen-Vuorinen }$L^{p}$\textit{\ conjecture for the Hilbert transform
when }$\frac{4}{3}<p<4$\textit{\ and the measures share no point\ masses},
\texttt{arxiv:2308.10733}.

\bibitem[Vol]{Vol}\textsc{A. Volberg,} \textit{Calder\'{o}n-Zygmund capacities
and operators on nonhomogeneous spaces,} CBMS Regional Conference Series in
Mathematics (2003), MR\{2019058 (2005c:42015)\}.
\end{thebibliography}
\end{document}